\documentclass[letterpaper]{amsart}

\usepackage{amssymb, mathrsfs, hyperref, enumitem, mathtools, graphicx, tikz}
\usetikzlibrary{decorations.markings}
\usepackage[alphabetic,initials,nobysame]{amsrefs}

\BibSpec{collection.article}{%
	+{}  {\PrintAuthors}				{author}
	+{,} { \textit}                     {title}
	+{.} { }                            {part}
	+{:} { \textit}                     {subtitle}
	+{,} { \PrintContributions}         {contribution}
	+{,} { \PrintConference}            {conference}
	+{}  {\PrintBook}                   {book}
	+{,} { }                            {booktitle}
	+{,} { }							{series}
	+{,} { }							{publisher}
	+{,} { }							{address}
	+{,} { \PrintDateB}                 {date}
	+{,} { pp.~}                        {pages}
	+{,} { }                            {status}
	+{,} { \PrintDOI}                   {doi}
	+{,} { available at \eprint}        {eprint}
	+{}  { \parenthesize}               {language}
	+{}  { \PrintTranslation}           {translation}
	+{;} { \PrintReprint}               {reprint}
	+{.} { }                            {note}
	+{.} {}                             {transition}
	+{}  {\SentenceSpace \PrintReviews} {review}
}

\newtheorem{theorem}{Theorem}[section]
\newtheorem{lemma}[theorem]{Lemma}
\newtheorem{corollary}[theorem]{Corollary}
\newtheorem{proposition}[theorem]{Proposition}

\theoremstyle{remark}
\newtheorem{remark}[theorem]{Remark}

\theoremstyle{definition}

\newcommand{\R}{\mathbb{R}}
\newcommand{\Z}{\mathbb{Z}}
\newcommand{\C}{\mathbb{C}}

\newcommand{\N}{\mathbb{N}}
\newcommand{\Hyp}{\mathbb{H}}
\newcommand{\D}{\mathbb{D}}
\newcommand{\loc}{\mathrm{loc}}
\renewcommand{\mod}{\operatorname{Mod}}
\renewcommand{\bar}[1]{\overline{#1}}

\DeclareMathOperator{\dist}{dist}

\DeclareMathOperator{\diam}{diam}

\renewcommand{\Re}{\operatorname{Re}}
\renewcommand{\Im}{\operatorname{Im}}
\DeclareMathOperator{\id}{id}

\numberwithin{equation}{section}
\numberwithin{figure}{section}

\title{Uniformization of Julia sets of geometrically finite rational maps}

\author{Yusheng Luo}
\address[Yusheng Luo]{Department of Mathematics, Cornell University, 212 Garden Ave, Ithaca, NY 14853, USA}
\email{yusheng.s.luo@gmail.com}
\thanks{Y.~L.\ is partially supported by NSF Grant DMS-2349929}

\author{Dimitrios Ntalampekos}
\address[Dimitrios Ntalampekos]{Department of Mathematics, Aristotle University of Thessaloniki, Thessaloniki, 54152, Greece.}
\email{dntalam@math.auth.gr}
\thanks{D.~N.\ is supported by the European Union (ERC, GRComPaS, Grant Agreement no.\ 101214615). Views and opinions expressed are however those of the author(s) only and do not necessarily reflect those of the European Union or the European Research Council. Neither the European Union nor the granting authority can be held responsible for them.}


\begin{document}
\begin{abstract}
    A Schottky set is the complement of a collection of disjoint open disks in the Riemann sphere. We completely characterize geometrically finite rational maps whose Julia sets are quasiconformally equivalent to Schottky sets. We also prove that the closure of the basin of infinity of a geometrically finite unicritical polynomial is quasiconformally equivalent to a Schottky set if and only if either there exists an attracting cycle and all bounded Fatou components have disjoint closures or there exists a parabolic cycle of multiplicity three. An ingredient of the proof is a new dynamical characterization of pairs of arbitrary quasicircles that are quasiconformally equivalent to pairs of tangent circles. 
\end{abstract}

\maketitle
\setcounter{tocdepth}{1}
\tableofcontents

\section{Introduction}

\subsection{Background}
Uniformization, whose origins go back to Koebe and Poincar\'e, is a central theme in complex analysis. Broadly speaking, uniformization seeks to replace an analytically or geometrically complicated object with a canonical, geometrically simple model through a map with controlled distortion.

A compact set $X\subset\widehat{\C}$ is called a Schottky set if $\widehat{\C}\setminus X$ is the union of at least two pairwise disjoint open geometric disks with distinct boundaries. A natural uniformization problem is to determine when a given fractal set $E\subset \widehat \C$ is quasiconformally equivalent to a Schottky set. Such round models are particularly useful in conformal dynamics. For example, the rigidity of Schottky sets of measure zero allows one to identify the quasisymmetries of $E$ with M\"obius symmetries of its round model \cite{BonkKleinerMerenkov:schottky}. Uniformization can therefore be used to determine quasisymmetry groups and to distinguish fractal sets arising from different conformal dynamical systems \cites{BonkLyubichMerenkov:carpet-julia, merenkov2026quasiregularmapssierpinskicarpet}.

In the setting of Kleinian groups, a characterization of geometrically finite limit sets that are quasiconformally equivalent to Schottky sets follows from work of Thurston and McMullen \cites{Thurston:hyperbolic-structures-I, McMullen:teichmuller-space}. They showed that such an equivalence exists if and only if the associated pared $3$-manifold is acylindrical. In particular, if a geometrically finite limit set is homeomorphic to a Schottky set, then it is quasiconformally equivalent to one.

In rational dynamics, Bonk--Lyubich--Merenkov proved that every Sierpi\'nski carpet Julia set of a subhyperbolic rational map is quasiconformally equivalent to a round Sierpi\'nski carpet \cite{BonkLyubichMerenkov:carpet-julia}*{Theorem 1.10}. Their proof proceeds by applying Bonk's quasiconformal uniformization criterion for Sierpi\'nski carpets \cite{Bonk:uniformization}*{Theorem 1.1}. There are two hypotheses in this criterion: first, the Fatou components are required to be uniform quasidisks, and second, they must be uniformly relatively separated. More recently, He--Xiao--Yang used Bonk's criterion to obtain a sufficient condition for a Cantor-bubble Julia set to be quasisymmetrically equivalent to a Cantor set with round bubbles \cite{HeXiaoYang:buble}*{Theorem~1.4}. 

In this work, we characterize geometrically finite rational maps whose Julia sets are quasiconformally equivalent to Schottky sets. The principal difficulty is the possible presence of contact points between distinct Fatou components. If the boundaries of two Fatou components meet, their relative distance is zero, so the uniform separation hypothesis in Bonk's criterion fails. More fundamentally, even if every Fatou component is individually a quasidisk, this does not in general imply that all Fatou components can be mapped simultaneously onto geometric disks by a single quasiconformal homeomorphism of the sphere. 

In our recent work \cite{LN24b} we studied the uniformization problem under the following assumptions:
\begin{itemize}
    \item the Julia set contains no critical points;
    \item the nerve of the Julia set is connected, where the nerve is the graph whose vertices correspond to the Fatou components and in which two vertices are joined by an edge whenever the closures of the corresponding Fatou components intersect.
\end{itemize}
In fact, we assumed that the Julia set is a \textit{gasket}, a stronger restriction than the second assumption above; see \cite{LN24b}*{Definition 1.1} for the precise definition. We proved in \cite{LN24b}*{Theorem~1.2} that a gasket Julia set containing no critical points is quasiconformally equivalent to a Schottky set if and only if any two Fatou components that meet do so tangentially.

The present work removes both of the above restrictions. We give a complete characterization of geometrically finite rational maps whose Julia sets are quasiconformally equivalent to Schottky sets, or, equivalently, whose Fatou components can be mapped simultaneously onto geometric disks by a quasiconformal homeomorphism of the sphere. 

We remark that, in contrast to \cite{LN24b}, the arguments in the present work are analytic rather than combinatorial. Our approach is based primarily on the study of the local dynamics near repelling and parabolic points. We use criteria recently established by the second author in \cites{Ntalampekos:schottky, Ntalampekos:tangent} to quasiconformally map a pair of Fatou components onto a pair of disks with controlled distortion, and then to promote this pairwise uniformization to a simultaneous quasiconformal uniformization of all Fatou components by geometric disks.

\subsection{Main results}
We state our main results. The proofs are given in Section \ref{section:proofs}.

\begin{theorem}\label{theorem:rational_uniformization}
    Let $f$ be a geometrically finite rational map whose Julia set is not homeomorphic to $\mathbb S^1$ or $\mathbb S^2$. There exists a quasiconformal map $g\colon \widehat \C\to \widehat \C$ that maps each Fatou component of $f$ onto a disk if and only if
    \begin{enumerate}[label=\normalfont(\arabic*)]
        \item\label{r:jordan} each Fatou component is a Jordan region,
        \item\label{r:parabolic3} all parabolic points have multiplicity $3$,
        \item\label{r:parabolic_basins} parabolic points do not lie on boundaries of periodic Fatou components other than their immediate basins, and
        \item\label{r:parabolic_tangent} for every pair of distinct Fatou components $U,V$ with $\bar U\cap \bar V\neq \emptyset$  the intersection $\bar U\cap \bar V$ consists of a single point $a_0$ that is not critical and is eventually mapped to a parabolic periodic cycle.
    \end{enumerate}
    Moreover, in this case the map $g|_{J(f)}$ is unique up to post-composition with conformal automorphisms of $\widehat\C$.
\end{theorem}

If $f$ is a geometrically finite polynomial, then all bounded Fatou components are Jordan regions \cite{RoeschYin:boundary_fatou}. In this case, we can drop condition \ref{r:jordan} and we obtain the following result.

\begin{theorem}\label{theorem:polynomial_uniformization}
    Let $f$ be a geometrically finite polynomial whose Julia set is not homeomorphic to $\mathbb S^1$. There exists a quasiconformal map $g\colon \C\to \C$ that maps each bounded Fatou component of $f$ onto a disk and is conformal in the basin $\mathcal A(\infty)$ if and only if conditions \ref{r:parabolic3}, \ref{r:parabolic_basins}, \ref{r:parabolic_tangent} hold for bounded Fatou components. Moreover, in this case the map $g|_{\mathcal A(\infty)}$ is unique up to post-composition with conformal automorphisms of $\C$.
\end{theorem}

For the sake of completeness, we also include the next elementary statement for Julia sets homeomorphic to $\mathbb S^1$. Note that in this case, the uniformizing map $g$ is not unique. 

\begin{proposition}\label{prop:julia_jordan}
    Let $f$ be a rational map whose Julia set is homeomorphic to $\mathbb S^1$. There exists a quasiconformal map $g\colon \widehat \C\to \widehat \C$ that maps each Fatou component of $f$ onto a disk if and only if
    \begin{itemize}
        \item $f$ has a parabolic periodic point of multiplicity $3$ or
        \item $f$ is hyperbolic.
    \end{itemize}
\end{proposition}

In addition, if $f$ is a unicritical polynomial, then we can drop condition \ref{r:parabolic_basins} since there is only one periodic cycle of bounded Fatou components. Also, condition \ref{r:parabolic_tangent} can be described more explicitly.

\begin{corollary}\label{corollary:unicritical}
      Let $f(z)=z^d+c$, $d\geq 2$, be a geometrically finite unicritical polynomial with non-empty bounded Fatou set. There exists a quasiconformal map of $\C$ that maps each bounded Fatou component of $f$ onto a disk if and only if 
    \begin{itemize}
        \item $f$ has a parabolic periodic point of multiplicity $3$ or
        \item $f$ is hyperbolic and all bounded Fatou components have disjoint closures.
    \end{itemize} 
\end{corollary}

Finally, we can more explicitly describe the above two conditions.

\begin{proposition}\label{corollary:degree2}
    The conditions of Corollary \ref{corollary:unicritical} are equivalent to the following two conditions, respectively:
    \begin{itemize}
        \item $c$ is a period-doubling bifurcation point or
        \item $c$ is in the main hyperbolic component of a primitive Multibrot copy.
    \end{itemize}
\end{proposition}

\subsection{Proof outline}
We derive our main theorems from a more general uniformization result for any prescribed collection of at least two Fatou components whose forward orbits consist of Jordan regions; see Theorem~\ref{theorem:uniformization_general}. Our starting point is a recent uniformization criterion for Schottky sets established by the second author; see Theorem~\ref{theorem:schottky}. This criterion reduces the problem of simultaneously mapping all Fatou components onto geometric disks by a quasiconformal homeomorphism to the corresponding problem for \textit{pairs} of Fatou components, provided that the quasiconformal distortion can be controlled uniformly. Accordingly, our proof has two main ingredients:
\begin{enumerate}[label=\normalfont(\Alph*)]
    \item\label{intro:a} a dynamical characterization of pairs of Fatou components that can be mapped onto pairs of disks by a quasiconformal map of the sphere;
    \item\label{intro:b} uniform bounds for the quasiconformal distortion of such maps.
\end{enumerate} 

We begin by studying when a Fatou component of a geometrically finite rational map is a quasidisk and when the quasidisk constant can be controlled uniformly. A standard blow-up argument based on Koebe's distortion theorem and the Ahlfors three-point condition shows that if a Fatou component is a quasidisk, then it is in fact a uniform quasidisk; see Proposition \ref{proposition:quasidisk_uniform}. We then give necessary and sufficient conditions for an individual Fatou component of a geometrically finite rational map to be a quasidisk (without keeping track of the distortion); see Proposition~\ref{prop:quasidiskgf}. The proof is given in Appendix \ref{appendix} and is based on conformal welding and quasisymmetric conjugacies of circle covering maps, building on our recent work \cite{LuoNtalampekos:piecewise_qc}. Sufficient dynamical conditions for quasidisks have also appeared in \cite{CarlesonJonesYoccoz:JuliaJohn} and in \cite{Zhang:siegel}.

We now turn to \ref{intro:a} and \ref{intro:b}. There are two cases to consider: either two Fatou components have disjoint closures, or their closures intersect. Suppose first that two Fatou components have disjoint closures. In this case, they can be mapped quasiconformally onto two geometric disks if and only if both Fatou components are quasidisks, which settles \ref{intro:a}. A further blow-up argument, carried out in Proposition \ref{proposition:separated_modulus} and involving a careful choice of separating annuli, yields uniform bounds for the quasiconformal distortion, as required in \ref{intro:b}. This settles the case of components with disjoint closures.

The more delicate case is that of Fatou components whose closures intersect. To address \ref{intro:a}, we establish a general result of independent interest that dynamically characterizes pairs of arbitrary quasicircles that are quasiconformally equivalent to pairs of tangent circles. The result reduces this global uniformization problem to a local dynamical condition involving parabolic germs. Roughly speaking, if $K_1,K_2$ are quasicircles meeting only at $0$, then a necessary and sufficient condition for the existence of a quasiconformal map of $\widehat \C$ that maps $K_1,K_2$ to circles is that $K_1,K_2$ are locally invariant under a ``quasiconformal parabolic germ" at $0$ and they approach the point $0$ within certain parabolic petals. See Figures \ref{fig:germ3} and \ref{fig:germ2} for an illustration. The precise statement is as follows.  

\begin{theorem}\label{theorem:qr_germ}
    Let $K_1,K_2\subset \widehat \C$ be quasicircles such that $K_1\cap K_2=\{0\}$. Then there exists a quasiconformal map $\Phi\colon \widehat \C\to \widehat \C$ that maps $K_1$ and $K_2$ onto two tangent circles if and only if there exists a map $f$ in a neighborhood of $0$ that  fixes $0$ and is quasiconformally conjugate to a parabolic germ such that
    \begin{enumerate}[label=\normalfont(\arabic*)]
        \item for $j=1,2$ there exists an attracting (resp.\ repelling) petal $P_j$, where $P_1$ and $P_2$ correspond to distinct attraction or repulsion vectors of the parabolic germ, and
        \item for $i=1,2$ the set $K_i\cap (P_j\cup \{0\})$ contains a Jordan arc $E_{ij}$ with an endpoint at $0$ that satisfies $f(E_{ij})\subset E_{ij}$ (resp.\ $f^{-1}(E_{ij})\subset E_{ij}$).
    \end{enumerate} 
\end{theorem}
The petals $P_1,P_2$ are allowed to be of the same type, or one can be attracting and the other repelling. This theorem builds on, and complements, the non-dynamical characterizations of tangent quasicircles given by the second author in \cite{Ntalampekos:tangent}. The proof is given in Section \ref{section:tangent_quasicircles}.

\begin{figure}
    \centering
    \begin{tikzpicture}
        \draw[rounded corners=8] (0,0)--(0.8,0.9)--(0.6,1.4)--(0,1.7)--(-0.6,1.4)--(-0.8,0.9)--(0,0);
        \draw[rounded corners=8] (0,0)--(0.8,-0.9)--(0.6,-1.4)--(0,-1.7)--(-0.6,-1.4)--(-0.8,-0.9)--(0,0);
        \draw[color=red,dashed] (2,0) circle (2cm); 
        \draw[color=blue,dashed] (-2,0) circle (2cm);
        \fill (0,0) circle (1.5pt) node[right] {$0$};
        \node[left] at (4,0) {$K_2$};
        \node[right] at (-4,0) {$K_1$};
        \path[postaction={decorate,decoration={markings, mark=at position 0.42 with {\arrow{>}}, mark=at position 0.44 with {\arrow{>}}, mark=at position 0.46 with {\arrow{>}}, mark=at position 0.58 with {\arrow{<}}, mark=at position 0.56 with {\arrow{<}}, mark=at position 0.54 with {\arrow{<}}
        }}](2,0) circle (2cm);
        \path[postaction={decorate,decoration={markings,mark=at position 0.04 with {\arrow{<}}, mark=at position 0.06 with {\arrow{<}}, mark=at position 0.08 with {\arrow{<}},mark=at position 0.92 with {\arrow{>}}, mark=at position 0.94 with {\arrow{>}}, mark=at position 0.96 with {\arrow{>}},
        }}](-2,0) circle (2cm);
    \end{tikzpicture}
    \caption{Two quasicircles approaching $0$ through two distinct attracting petals. Here $f$ has parabolic multiplicity greater than $2$.}
    \label{fig:germ3}

    \bigskip
    
    \begin{tikzpicture}
        \draw[thick, domain=0:360, samples=100, smooth, scale=2]
        plot ( {0.5*sin(\x) - 0.25*sin(2*\x)}, {0.25- 0.5*cos(\x) + 0.25*cos(2*\x)}  );
        \fill (0,0) circle (1.5pt);
        \draw[color=red,dashed] (2,0) circle (2cm); 
        \draw[color=blue,dashed] (-2,0) circle (2cm);
        \fill (0,0) circle (1.5pt) node[right] {$0$};
        \node[left] at (4,0) {$K_2$};
        \node[right] at (-4,0) {$K_1$};
        \path[postaction={decorate,decoration={markings, mark=at position 0.42 with {\arrow{>}}, mark=at position 0.44 with {\arrow{>}}, mark=at position 0.46 with {\arrow{>}}, mark=at position 0.58 with {\arrow{>}}, mark=at position 0.56 with {\arrow{>}}, mark=at position 0.54 with {\arrow{>}}
        }}](2,0) circle (2cm);
        \path[postaction={decorate,decoration={markings,mark=at position 0.04 with {\arrow{<}}, mark=at position 0.06 with {\arrow{<}}, mark=at position 0.08 with {\arrow{<}},mark=at position 0.92 with {\arrow{<}}, mark=at position 0.94 with {\arrow{<}}, mark=at position 0.96 with {\arrow{<}},
        }}](-2,0) circle (2cm);
    \end{tikzpicture}    
    \caption{Two quasicircles approaching $0$ through an attracting and a repelling petal of $f$. Here $f$ has parabolic multiplicity $2$.}
    \label{fig:germ2}
\end{figure}

It remains to establish the quantitative statement in \ref{intro:b}. We show that if two Fatou components have intersecting closures and can be quasiconformally mapped onto a pair of tangent disks, then such a map can be chosen with uniformly controlled quasiconformal distortion; see Proposition \ref{proposition:circularizable_gf}. The proof again uses a blow-up argument, but the main difficulty lies in precisely identifying the geometric and dynamical quantities on which the distortion bound depends.

It is useful to compare this argument with the familiar quasidisk setting. Suppose that $f$ is a rational map, $U$ is a Fatou component that is a Jordan region, $f(U)$ is a quasidisk, and $\bar U$ is disjoint from the critical set of $f$. Since $f(U)$ is a quasidisk, its boundary satisfies the Ahlfors three-point condition; see Theorem \ref{theorem:quasidisk_quasicircle}. Koebe's distortion theorem then implies that $\partial U$ also satisfies the Ahlfors three-point condition, with constants depending only on the corresponding constants for $f(U)$ and on the distance from $\partial U$ to the critical set. Applying Ahlfors' theorem once more, we conclude that $U$ is a quasidisk with distortion controlled in terms of the same quantities. 

Our situation is analogous, but involves pairs of touching components. Suppose $U,V$ are Fatou components with $\bar U\cap \bar V\neq \emptyset$ and that $f(U)$ and $f(V)$ can be quasiconformally mapped onto a pair of tangent disks. One must then determine whether the same is true for $U$ and $V$, and, crucially, whether the quasiconformal distortion can be controlled uniformly. The required quantitative dependence is provided by a theorem of the second author from \cite{Ntalampekos:tangent}; see Theorem \ref{theorem:tangent}. This result involves hyperbolic geometry and plays, in the setting of tangent pairs, the role of Ahlfors' theorem in the quasidisk argument above. The details are given in Proposition~\ref{proposition:circularizable_gf}.

\section{Preliminaries in quasiconformal geometry}

\subsection{Quasiconformal maps and related concepts}

Let $K\geq 1$. An orientation-preserving homeomorphism $f\colon U\to V$ between open subsets of $\C$ is \textit{quasiconformal} if $f$ lies in the Sobolev space $W^{1,2}_{\loc}(U)$ and
$$\|Df(z)\|^2\leq KJ_f(z)\quad \text{for a.e.\ $z\in U$;}$$ 
here $\|\cdot \|$ denotes the operator norm of the matrix $Df$ and $J_f$ is its Jacobian determinant. In this case we say that $f$ is $K$-quasiconformal. A homeomorphism $f\colon U\to V$ between open subsets of $\widehat \C$ is quasiconformal if $f|_{U\setminus \{\infty, f^{-1}(\infty)\}}$ is quasiconformal in the above sense. We will freely use the fundamental facts that the inverse of a quasiconformal map is quasiconformal, that the composition of two quasiconformal maps is quasiconformal, and that $1$-quasiconformal maps are conformal. We direct the reader to \cite{LehtoVirtanen:quasiconformal} for further background.

Let $(X,d_X)$ be a metric space. We define the (metric) cross ratio of a quadruple of distinct points $a,b,c,d\in X$ to be
$$[a,b,c,d]= \frac{d_X(a,c)d_X(b,d)}{d_X(a,d)d_X(b,c)}.$$
Observe that if $X=\widehat \C$, equipped with the chordal metric, and $a,b,c,d\neq \infty$, then 
$$[a,b,c,d]= \frac{|a-c||b-d|}{|a-d||b-c|}.$$
The factors containing the point $\infty$ are omitted. For example, 
$$[a,b,c,\infty]= \frac{|a-c|}{|b-c|}.$$

A homeomorphism $\eta\colon [0,\infty)\to [0,\infty)$ is called a \textit{distortion function}. A homeomorphism $f\colon (X,d_X)\to (Y,d_Y)$ between metric spaces is \textit{quasi-M\"obius} if there exists a distortion function $\eta$ such that 
$$[f(a),f(b),f(c),f(d)] \leq \eta ([a,b,c,d])$$
for every quadruple of distinct points $a,b,c,d\in X$. In this case, we say that $f$ is $\eta$-quasi-M\"obius. Suppose that $X,Y\subset \widehat \C$ and $f\colon X\to Y$ is $\eta$-quasi-M\"obius. Then the cross ratios in $X$ and $Y$ may be computed with the chordal or the Euclidean distance interchangeably. Thus, we do not need to specify which of the two metrics is used in $X$ and $Y$. Also, since cross ratios are invariant under M\"obius transformations, the composition of $f$ with a M\"obius transformation is $\eta$-quasi-M\"obius as well.  

A homeomorphism $f\colon (X,d_X)\to (Y,d_Y)$ between metric spaces is \textit{quasisymmetric} if there exists a distortion function $\eta$ such that
$$ \frac{d_Y(f(a),f(b))}{d_Y(f(a),f(c))} \leq \eta\left(\frac{d_X(a,b)}{d_X(a,c)}\right)$$
for every triple of distinct points $a,b,c\in X$. In this case, we say that $f$ is $\eta$-quasisymmetric. 

We note that every bi-Lipschitz map is quasisymmetric and every quasisymmetric map is quasi-M\"obius. Also, these three classes of maps are invariant under composition and under taking the inverse map. These statements are \textit{quantitative}. For example, if $f$ is $\eta$-quasi-M\"obius, then $f^{-1}$ is $\eta'$-quasi-M\"obius for some distortion function $\eta'$ depending only on $\eta$.

\begin{theorem}[\cite{Heinonen:metric}*{Theorem 11.4}]\label{q:qc_qs}
    Let $f\colon U\to V$ be a $K$-quasiconformal map between open subsets of $\C$. For each $a>1$ there exists a distortion function $\eta$, depending only on $K$ and $a$, such that if $B(z,r)\subset B(z,ar)\subset U$, then $f$ is $\eta$-quasisymmetric in $B(z,r)$.
\end{theorem}

We record an immediate consequence of Koebe's distortion theorem.
\begin{lemma}\label{lemma:koebe_qm}
    Let $U\subset \widehat\C$ be a connected open set, $E\subset U$ be a compact set and $f\colon U\to \widehat \C$ be a conformal embedding. Then $f\colon (E,\chi)\to (f(E),\chi)$ is $\eta$-quasi-M\"obius for a distortion function $\eta$ depending only on $U$ and $E$.
\end{lemma}
\begin{proof}
    If $U=\widehat \C$, then $f$ is a M\"obius transformation, so it is $\eta$-quasi-M\"obius with $\eta(t)=t$. Suppose that $\widehat \C\setminus U\neq \emptyset$, so $\widehat \C\setminus f(U)\neq \emptyset$. Consider M\"obius transformations $\phi,\psi$ such that $\phi(U)\subset \C$ and $\psi(f(U))\subset \C$, and let $g=\psi \circ f\circ \phi^{-1}\colon \phi(U)\to \psi(f(U))$. By covering $\phi(E)$ with finitely many balls in $\phi(U)$ and applying Koebe's distortion theorem to each of them, we see that there exists a constant $M>0$ depending on $g$ and a constant $L\geq 1$ that depends only on $\phi(E)$ and $\phi(U)$ such that 
    $$L^{-1}M|z-w|\leq |g(z)-g(w)|\leq LM |z-w|\,\,\, \text{for all $z,w\in \phi(E)$}.$$
    Thus, $g\colon (\phi(E),|\cdot|)\to (\psi(f(E)), |\cdot|)$ is uniformly quasi-M\"obius, depending only on $U$ and $E$. The identity map from the Euclidean to the chordal metric preserves cross ratios, so $g\colon (\phi(E),\chi)\to (\psi(f(E)), \chi)$ is uniformly quasi-M\"obius, depending only on $U$ and $E$. Finally, since $\phi,\psi$ are M\"obius transformations, the map $f\colon (E,\chi)\to (f(E),\chi)$ is uniformly quasi-M\"obius depending only on $U$ and $E$, as desired. 
\end{proof}

\subsection{Quasidisks and quasicircles}

A set $U\subset \widehat \C$ is a \textit{quasidisk} if it is the image of the unit disk $\D$ under a quasiconformal homeomorphism $f\colon \widehat \C\to \widehat \C$. If $f$ is $K$-quasiconformal for some $K\geq1$, we say that $U$ is a $K$-quasidisk. Note that the complement of a quasidisk is also a quasidisk. 

Let $J\subset \widehat \C$ be a Jordan curve (resp.\ Jordan arc). We say that $J$ is a \textit{quasicircle} (resp.\ \textit{quasiarc}) if there exists a constant $L\geq 1$ such that for every pair of points $z,w\in J$ there exists an arc $E\subset J$ with endpoints $z,w$ such that 
\begin{align}\label{condition:ahlfors}
    \diam_\chi (E)\leq L\chi(z,w),
\end{align}
where $\diam_\chi(E)$ denotes the diameter of $E$ in the chordal metric. In this case we say that $J$ is an $L$-quasicircle (resp.\ $L$-quasiarc). We note that a set $J\subset \C$ is a quasicircle (resp.\ quasiarc) if and only if there exists a constant $L'\geq 1$ such that for every pair of points $z,w\in J$ there exists an arc $E\subset J$ with endpoints $z,w$ such that
$$\diam(E)\leq L'|z-w|,$$
where $\diam(E)$ denotes the diameter in the Euclidean metric. This equivalence is quantitative and follows from \cite{Bonk:uniformization}*{Proposition 4.4}. Moreover, the referenced result implies that the quasi-M\"obius image of a quasicircle in $\widehat \C$ is again a quasicircle, quantitatively.

The following deep result of Ahlfors provides the link between quasidisks and quasicircles. See \cite{Ahlfors:qc}*{Section IV.D} or \cite{LehtoVirtanen:quasiconformal}*{Theorem II.8.6} or \cite{Pommerenke:conformal}*{Section 5} for a proof.

\begin{theorem}[\cite{Ahlfors:reflections}]\label{theorem:quasidisk_quasicircle}
Let $U\subset \widehat \C$ be a Jordan region and $J=\partial U$. Then $U$ is a quasidisk if and only if $J$ is a quasicircle, quantitatively.
\end{theorem}

Rickman established the corresponding result for quasiarcs.

\begin{theorem}[\cite{Rickman:quasiconformal_arcs}*{Theorem 1}]\label{theorem:rickman}
Let $J\subset \widehat \C$ be a Jordan arc. There exists a quasiconformal map of $\widehat{\C}$ that maps a line segment onto $J$ if and only if $J$ is a quasiarc, quantitatively.
\end{theorem}

In particular, quasiarcs can be embedded into quasicircles. The next result is elementary. See \cite{Ntalampekos:tangent}*{Proposition 2.6} for a proof.
\begin{proposition}\label{proposition:quasidisk_extension}
    Let $U,V\subset \widehat \C$ be quasidisks and $f\colon U\to V$ be a quasiconformal homeomorphism. Then there exists an extension of $f$ to a quasiconformal homeomorphism of $\widehat\C$, quantitatively.
\end{proposition}

\subsection{Modulus}\label{section:modulus}
Let $\Gamma$ be a family of curves in $\C$. We say that a Borel function $\rho\colon \C\to [0,\infty]$ is \textit{admissible} for $\Gamma$ if 
$$\int_\gamma \rho\, ds\geq 1$$
for every rectifiable curve $\gamma\in \Gamma$. The \textit{(conformal) modulus} of $\Gamma$ is defined as 
$$\mod\Gamma = \inf_\rho \int \rho^2,$$
where the infimum is over all admissible functions $\rho$ for $\Gamma$.

An annulus $A$ is a domain in $\widehat \C$ whose complement has exactly two components. There exists a conformal map from each annulus $A$ onto an annulus $A(0;r,R)=\{z\in \C: r<|z|<R\}$, where $0 \leq r<R\leq  \infty$. Moreover, this annulus is unique up to scaling $z\mapsto az$ for some $a>0$. The modulus of the annulus $A$ is defined as
$$\mod A= \frac{1}{2\pi} \log \left(\frac{R}{r}\right)$$
if $0<r<R<\infty$. If $r=0$ or $R=\infty$, then $\mod A=\infty$. The modulus of the annulus $A$ agrees with the modulus of the family of curves separating the complementary components of $A$. If $E,F\subset \widehat \C$ are disjoint continua that do not separate $\widehat \C$, we denote by $A(E,F)$ the annulus $\widehat \C\setminus (E\cup F)$. If $A\subset A(E,F)$ is any annulus separating $E$ and $F$ we have $\mod A\leq \mod A(E,F)$.

We define a \textit{quadrilateral} $Q\subset \C$ to be a simply connected region together with a conformal map $f$ from a rectangle $R=(0,a)\times (0,b)$ onto $Q$ such that $f$ extends continuously to the closures. The images of the vertical sides of $R$ are called the \textit{vertical sides} of $Q$. The modulus of $Q$ is defined to be 
$$\mod Q= \frac{b}{a}.$$
Equivalently, $\mod Q$ is the modulus of the curve family $f(\Gamma)$, where $\Gamma$ is the family of curves in $R$ whose endpoints lie in distinct vertical sides of $R$.  

Each rectangle $R=(0,a)\times (0,b)$ can be conformally mapped to a quadrilateral
$$S=\{z\in \C: r<|z|<R, \,\, \Im z >0\}$$
with vertical sides $[-R,-r]$ and $[r,R]$. Moreover, it can be shown that
$$\frac{b}{a}=\frac{1}{\pi}\log \left(\frac{R}{r}\right).$$
Thus, if we consider the full annulus $A=\{z\in \C: r<|z|<R\}$, then 
$$\mod A= \frac{1}{2}\cdot \mod S.$$

\subsection{Uniformization by Schottky sets}\label{section:uniformization}

If $U\subset \widehat\C$ is an open set, we set $U^*=\widehat{\C}\setminus \bar  U$. If $U\subset \widehat \C$ is a hyperbolic domain, we denote by $h_U$ the hyperbolic metric on $U$.

Let $U,V\subset \widehat \C$ be disjoint Jordan regions such that $\bar U\cap \bar V$ is either empty or contains only one point; in particular, $\widehat \C\setminus (\bar U\cup \bar V)$ is connected. The {relative hyperbolic metric} corresponding to the pair $(V,U)$ is defined as
$$d_{V,U}(z,w)= \inf_\gamma\{\ell_{h_{U^*}}(\gamma)\},\quad z,w\in \partial V\setminus \bar U, $$
where the infimum is over all rectifiable curves $\gamma\colon [a,b]\to U^*\setminus V$ such that $\gamma(a)=z$, $\gamma(b)=w$. If $V$ is a quasidisk, then $d_{V,U}$ is a metric on $\partial V\setminus \bar U$. We direct the reader to \cite{Ntalampekos:tangent} for further details and examples with explicit calculations of the metric. 

We say that a pair of two disjoint Jordan regions $U,V$ is \textit{quasiconformally circularizable} if there exists $K\geq 1$ and a $K$-quasiconformal homeomorphism of $\widehat{\C}$ that maps $U$ and $V$ to two disks with distinct boundaries. In particular,  $\bar U\cap \bar V$ is either empty or contains a single point. In this case we say that $U,V$ are $K$-quasiconformally circularizable. 

Using the relative hyperbolic metric we can characterize when two disjoint quasidisks that touch each other at a boundary point are quasiconformally circularizable.

\begin{theorem}[\cite{Ntalampekos:tangent}*{Theorem 1.3}]\label{theorem:tangent}
    Let $U,V\subset \widehat \C$ be disjoint quasidisks such that $\bar U\cap \bar V$ is a single point $a_0$. The pair $U,V$ is quasiconformally circularizable if and only if the identity map
    $$\id\colon (\partial V\setminus \{a_0\}, d_{V,U})\to (\partial V\setminus \{a_0\},\chi)$$
    is quasi-M\"obius, quantitatively.
\end{theorem}

The next lemma gives a sufficient condition so that the identity map in the above theorem is quasi-M\"obius.

\begin{lemma}[\cite{Ntalampekos:tangent}*{Lemma 3.3}]\label{lemma:chordarc}
    Let $K\geq 1$. Let $U,V\subset \widehat \C$ be $K$-quasidisks such that $\bar U\cap \bar V=\{\infty\}$. Suppose that there exists a $K$-chord-arc curve $J\subset \widehat \C\setminus (\bar U\cup \bar V)$ that separates $\partial U\cap \C$ from $\partial V\cap \C$ and satisfies
    $$N_{1/K}(J)\subset \widehat \C\setminus (\bar U\cup \bar V)\subset N_K(J).$$
    Then the identity map
    $$\id\colon (\partial V\setminus \{\infty\}, d_{V,U})\to (\partial V\setminus \{\infty\},\chi)$$
    is quasi-M\"obius, quantitatively, depending only on $K$.
\end{lemma}

Here a set $J\subset \C$ is a \textit{chord-arc curve} if $J\cup \{\infty\}$ is a Jordan curve in $\widehat\C$ and for every $z,w\in J$ the length of the arc of $J$ between $z$ and $w$ is bounded above by $K|z-w|$. In this case we say that $J$ is a $K$-chord-arc curve. Also, $N_K(J)$ denotes the open $K$-neighborhood of $J$ in the Euclidean metric.

\begin{lemma}\label{lemma:tangent_local}
    Let $U_i,V_i\subset \widehat \C$ be quasidisks such that $\bar {U_i}\cap \bar {V_i}=\{z_i\}$ for $i=1,2$. Suppose that there exists a neighborhood $W_i$ of $z_i$ for $i=1,2$ and a quasiconformal map $\phi\colon W_1\to W_2$ such that $\phi(W_1\cap U_1)=W_2\cap U_2$ and $\phi(W_1\cap V_1)=W_2\cap V_2$. Then the pair $U_1,V_1$ is quasiconformally circularizable if and only if the pair $U_2,V_2$ is quasiconformally circularizable. 
\end{lemma}

\begin{proof}
    First, we extend $\phi$ to $\widehat \C$ as follows. By composing $\phi$ with isometries of $\C$ we assume that $z_1,z_2\in \C$. By shrinking $W_1$, we may assume that $W_1$ is a disk centered at $z_1$ and $\phi$ is a quasiconformal map from a neighborhood of $\bar{W_1}$ into $\C$. By Theorem \ref{q:qc_qs}, $\phi$ is quasisymmetric in $\bar {W_1}$. Thus, $\phi(\partial W_1)$ is a quasicircle (see \cite{Pommerenke:conformal}*{Proposition 5.10}) and $W_2\coloneqq \phi(W_1)$ is a quasidisk by Theorem \ref{theorem:quasidisk_quasicircle}. By Proposition \ref{proposition:quasidisk_extension}, $\phi|_{W_1}$ extends to a quasiconformal map of $\widehat \C$. We denote the extension by $\phi$ for simplicity.

    Without loss of generality, suppose that the pair $U_2,V_2$ is quasiconformally circularizable. Thus, there exists a quasiconformal map $\psi$ of $\widehat \C$ such that $\psi(U_2),\psi(V_2)$ are half-planes that are tangent at $\psi(z_2)=\infty$. Let $g=\psi\circ \phi$. Then $g(\partial U_1)$ and $g(\partial V_1)$ are Jordan curves in $\widehat \C$ that are straight lines in a neighborhood of $\infty$. It is elementary to find a chord-arc curve $J\subset \widehat \C\setminus (g(\bar {U_1})\cup g(\bar {V_1}))$ that is a straight line in a neighborhood of $\infty$ and otherwise it is smooth; see Figure \ref{figure:tangent_local}. By compactness, there exists a constant $K\geq 1$ such that $N_{1/K}(J)\subset \widehat \C\setminus (g(\bar {U_1})\cup g(\bar {V_1}))\subset N_K(J)$. Thus, the conclusion of Lemma \ref{lemma:chordarc} is true. By Theorem \ref{theorem:tangent} we conclude that the pair $g(U_1),g(V_1)$ is quasiconformally circularizable. Equivalently, the pair $U_1,V_1$ is quasiconformally circularizable. 
\end{proof}

\begin{figure}
    \centering
    \begin{tikzpicture}
        \draw[rounded corners=10pt] (-5,0)--(-2,0)--(-1,1)--(0,-0.5)--(1,0.6)--(2,0)--(5,0) node[right] {$g(\partial U_1)$};
        \draw[rounded corners=10pt] (-5,1)--(-2,1)--(-1,1.4)--(0,0.7)--(1,1.5)--(1.5,0.8)-- (2.2,1)--(5,1)node[right]{$g(\partial V_1)$};
        \draw[rounded corners=10pt,dashed] (-5,0.5)--(-2,0.5)--(-1,1.2)--(0,0)--(1,1)--(1.5,0.5)--(2.2,0.5)--(5,0.5) node[right] {$J$}; 

    \end{tikzpicture}
    \caption{A chord-arc curve between two Jordan curves that are straight lines in a neighborhood of $\infty$.}\label{figure:tangent_local}
    \label{fig:placeholder}
\end{figure}
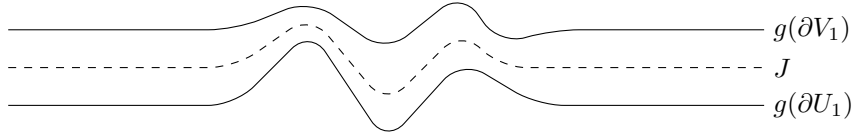

An \textit{unbounded Jordan curve} is a set $J\subset \C$ such that $J\cup\{\infty\}$ is a Jordan curve in the topology of $\widehat \C$.
\begin{theorem}[\cite{Ntalampekos:tangent}*{Theorem 1.6}]\label{theorem:tangent:graph}
    Let $J\subset \mathbb H$ be an unbounded Jordan curve. There exists a quasiconformal homeomorphism of $\C$ that preserves the real line and maps $J$ onto a line if and only if
    \begin{enumerate}[label=\normalfont(\arabic*)]
        \item $J\cup \{\infty\}$ is a quasicircle,
        \item there exists a Lipschitz function $f\colon \R\to (0,\infty)$ and $a\geq 1$ such that $J\subset \{(x,y)\in \R^2: a^{-1}f(x)\leq y\leq af(x)\}$, and
        \item an antiderivative of $1/f$ is quasisymmetric.
    \end{enumerate}
\end{theorem}
We will apply this result in the case that $f(x)=(1+|x|)^{-p}$, $x\in \R$, where $p\geq 1$. See \cite{Ntalampekos:tangent}*{Example 4.2} for details. The next result gives an obstruction for mapping quasiconformally two quasidisks onto two disks.

\begin{lemma}[\cite{Ntalampekos:schottky}*{Lemma 3.2}]\label{lemma:room3}
    Let $U_1,U_2,U_3\subset \widehat \C$ be disjoint quasidisks such that $\partial U_1\cap \partial U_2\cap \partial U_3\neq \emptyset$. Then no quasiconformal map of $\widehat \C$ can map $U_1$ and $U_2$ onto two disks.
\end{lemma}

Theorem \ref{theorem:tangent} can be combined with the next result, which gives a characterization of sets that are quasiconformally equivalent to Schottky sets.

\begin{theorem}[\cite{Ntalampekos:schottky}*{Theorem 1.1}]\label{theorem:schottky}
Let $\{U_i\}_{i\in I}$ be a collection of at least two disjoint Jordan regions in $\widehat \C$ such that $\bar{U_i}\cap \bar{U_j}$ contains at most one point for $i\neq j$. The following are quantitatively equivalent.
\begin{enumerate}[label=\normalfont(\roman*)]
\item For each $i,j\in I$ there exists a uniformly quasiconformal homeomorphism of $\widehat \C$ that maps $U_i$ and $U_j$ to disks.
\item There exists a quasiconformal homeomorphism $f$ of $\widehat \C$ that maps $U_i$ to a disk for each $i\in I$ and is $1$-quasiconformal on $S=\widehat \C\setminus \bigcup_{i\in I} U_i$. 
\end{enumerate}
In this case, the map $f|_S$ is unique up to post-composition with a conformal automorphism of $\widehat \C$. 
\end{theorem}

In the case that $S$ has area zero, the $1$-quasiconformality of $f$ in (ii) is vacuously true. Moreover, in this case the uniqueness of $f$ as in the last statement of the theorem follows immediately from \cite{BonkKleinerMerenkov:schottky}*{Theorem 1.1}.

\section{Tangent quasicircles in complex dynamics}
\subsection{Obstructions for tangent quasicircles}
The next lemma shows that tangent quasicircles cannot be locally invariant under hyperbolic dynamics. See also \cite{LN24b}*{Lemma 7.3} for a related statement. A \textit{repelling hyperbolic germ} at $0$ is a holomorphic map $f$ in a neighborhood of $0$ with $f(0)=0$ and $|f'(0)|>1$.

\begin{lemma}\label{lemma:repelling_invariant}
    Let $K_1,K_2\subset \widehat \C$ be Jordan curves such that $K_1\cap K_2=\{0\}$. Let $f$ be a repelling hyperbolic germ at $0$ such that each of $K_1,K_2$ is locally invariant under $f$ near $0$. Then no quasiconformal map of $\widehat \C$ can map $K_1$ and $K_2$ onto two circles. 
\end{lemma}

Here we say that \textit{$K_i$ is locally invariant under $f$ near $0$} if there exists a neighborhood $U$ of $0$ such that $f(K_i\cap U)\subset K_i$.

\begin{proof}
By the Koenigs linearization theorem \cite{Milnor:dynamics}*{Theorem 8.2} we can conformally conjugate $f$ near $0$ to the map $g(z)=\lambda z$, where $\lambda=f'(0)$, so $|\lambda|> 1$. The Jordan curves $K_1,K_2$ contain Jordan arcs having $0$ as an interior point that correspond under the conjugating map to some Jordan arcs $J_1,J_2$. The arcs $J_1,J_2$ are locally invariant under $g$ in the sense that there exists a ball $B(0,r)$ such that $g(J_i\cap B(0,r))\subset J_i$ for $i=1,2$. If there exists a quasiconformal map of $\widehat \C$ mapping $K_1,K_2$ to circles, then $K_1,K_2$ would be quasicircles, so $J_1,J_2$ would be quasiarcs by the analyticity of the Koenigs linearizing map. The quasiarcs $J_1,J_2$ can be extended to quasicircles by Theorem \ref{theorem:rickman}. Thus, by Lemma \ref{lemma:tangent_local} it suffices to show that no quasiconformal map of $\widehat \C$ maps $J_1$ and $J_2$ onto two circular arcs. 

Suppose that there exists a quasiconformal map $\phi$ of $\widehat \C$ such that $\phi(0)=0$, $\phi(J_1)\subset \R$, and $\phi(J_2)$ is an arc of a circle $|z-Ri|=R$ in the upper half-plane that is tangent to the real line at $0$. By Theorem \ref{q:qc_qs}, if $r$ is chosen to be sufficiently small so that $\phi^{-1}(\infty)\notin \bar B(0,r)$, then $\phi$ is $\eta$-quasisymmetric in $B(0,r)$ for some distortion function $\eta$. Let $w_1=x\in \phi(J_1)\cap (0,R)$ and $w_2=x+iy\in \phi(J_2)$ so that $y=R-\sqrt{R^2-x^2}= \frac{x^2}{2R}+O(x^4)$ as $x\to 0^+$. Note that for small $x$ we have $w_1,w_2\in \phi(B(0,r))$,
\begin{align*}
    1\leq \frac{|w_2|}{|w_1|} \leq \sqrt{2}\quad \text{and}\quad \frac{|w_1-w_2|}{|w_1|}= \frac{x}{2R}+O(x^3).
\end{align*}
Let $z_i=\phi^{-1}(w_i)\in J_i\cap B(0,r)$, $i=1,2$. Since $\phi$ is $\eta$-quasisymmetric and $\phi(0)=0$, we have 
\begin{align}\label{lemma:repelling_invariant:1}
     \eta^{-1}(1)\leq \frac{|z_2|}{|z_1|}\leq 1/\eta^{-1}(1/\sqrt{2})\quad \text{and} \quad \frac{|z_1-z_2|}{|z_1|}=o(1)\quad \text{as $x\to 0$}.
\end{align}
By the first inequalities, there exists $c_0\geq 1$ that depends only on $\lambda$ and $\eta$ such that for some $n\in \N\cup \{0\}$ we have
$$ \frac{r}{c_0}\leq |\lambda^nz_i|<r \quad \text{for $i=1,2$}.$$
Note that the point $z_i'=\lambda^nz_i$ lies in $J_i'\coloneqq \{z\in J_i: r/c_0\leq |z|< r\}$ for $i=1,2$. Since the sets $J_1',J_2'$ have disjoint closures, we have $|z_1'-z_2'|\geq d_0$ for some constant $d_0>0$ independent of $z_1,z_2$. This implies that 
$$\frac{|z_1-z_2|}{|z_1|} \geq \frac{d_0}{r}.$$
This contradicts \eqref{lemma:repelling_invariant:1} as $x\to 0^+$.
\end{proof}

\subsection{Preliminaries in parabolic dynamics}

We include some preliminaries on parabolic fixed points, following Milnor \cite{Milnor:dynamics}*{Section 10}. Consider a holomorphic map $f$ defined in a neighborhood of $0$ such that $f(0)=0$, $\lambda =f'(0)$ is a $k$-th root of unity, and $f^{\circ k}$ is not the identity map. Then $0$ is called a \textit{parabolic fixed point} of $f$. Consider the case that $\lambda=1$, so $f$ has an expansion of the form 
$$f(z)=z+az^{n+1}+O(z^{n+2})= z(1+az^n+O(z^{n+1}))$$
where $a\neq 0$, $n\in \N$. The number $n+1$ is called the \textit{multiplicity} of the parabolic fixed point $0$. A map $f$ with the above expansion is also called a \textit{parabolic germ} at $0$.

A complex number $v$ is called a \textit{repulsion} (resp.\ \textit{attraction}) \textit{vector} for $f$ at $0$ if $nav^n=+1$ (resp.\ $-1$). Thus, there are $n$ attraction and $n$ repulsion vectors. We restrict $f$ to a smaller neighborhood of $0$ if necessary, so that it is univalent. Note that $$f^{-1}(z)= z-az^{n+1}+O(z^{n+2}),$$
so $0$ is a parabolic fixed point of $f^{-1}$ and the attraction vectors for $f^{-1}$ are precisely the repulsion vectors for $f$ and vice versa.

If an orbit $z_k=f(z_{k-1})$, $k\in \N$, converges to $0$ as $k\to\infty$, then it does so in the direction of one attracting vector $v$, in the sense that $v=\lim_{k\to\infty} \sqrt[n]{k}z_k$; an equivalent statement is that, for any sector $S$ at $0$ that contains $v$ in its interior, the point $z_k$ lies in $S$ for all sufficiently large $k\in \N$. See \cite{Milnor:dynamics}*{Lemma 10.1}.

In fact, for each attracting direction $v$ we can specify a large open set, called an \textit{attracting petal}, that is attracted to $0$ under iteration. Specifically, an open set $\mathcal P$ on which $f$ is defined and is univalent is called an attracting petal for $f$ for the attracting direction $v$ if $f(\mathcal P)\subset \mathcal P$ and an orbit $z_k=f(z_{k-1})$, $k\in \N$, converges to $0$ from the direction $v$ if and only if the orbit is eventually absorbed by $\mathcal P$. A \textit{repelling petal} for $f$ for a repelling direction $v$ is by definition an attracting petal for $f^{-1}$ for the direction $v$. 

There are several petals for $f$, but one can specify particular petals with good geometry as follows. Suppose for simplicity that $v=1$ is a repulsion vector for $f$, so $na=1$; for the general case we can conjugate $f$ by a rotation and a scaling to achieve this normalization. Consider the map $\phi(z)=-1/z^n$, which maps conformally the sector $\Delta=\{re^{i\theta}: r>0,\,\, |\theta|<\pi/n\}$ onto the slit plane $\C\setminus [0,\infty)$. Also, let $\psi=\phi^{-1}$. For $R>0$ consider the left half-plane $\mathbb H_{<-R}=\{x+iy\in \C: x<-R\}$. If $R$ is sufficiently large, then $\mathcal P=\psi( \mathbb H_{<-R})$ is a repelling petal for $f$ for the direction $v$. This is proved in \cite{Milnor:dynamics}*{p.~108}. Note that $\partial \mathcal P$ is a smooth curve except at $0$. Moreover, it can be shown that $\mathcal P$ is a quasidisk. See Figure \ref{fig:repelling_petal} for an illustration. Given any sector $S$ at $0$ that contains $v$ in its interior we can find arbitrarily small repelling petals contained in $S$. Furthermore, there exist heart-shaped petals of angle arbitrarily close to the angle of the sector $\Delta$, as follows from the argument in \cite{Milnor:dynamics}*{p.~112}. We obtain the following consequence.

\begin{figure}
    \centering
    \begin{tikzpicture}
        \draw[dashed] (0,0)--(3,2);
        \draw[dashed] (0,0)--(3,-2);
        \draw[fill=black!20] (0,0) to[out=30, in=90] (3,0) to[out=-90, in=-30] (0,0);
        \node at (2,0) {$\mathcal P$};
        \node at (2.5,1.3) {$\Delta$};

        \begin{scope}[shift={(7,0)}]
        \draw[dashed] (-2,-2)--(2,2);
        \draw[dashed] (-2,2)--(2,-2);
        \draw[fill=black!20] (0,0) to[out=45, in=160] (1.5,0.5) to[out=-20, in=40] (1.7,0) to[out=-40, in=20] (1.5,-0.5) to[out= -160, in=-45] (0,0);
        \draw[fill=black!20, rotate=90] (0,0) to[out=45, in=160] (1.5,0.5) to[out=-20, in=40] (1.7,0) to[out=-40, in=20] (1.5,-0.5) to[out= -160, in=-45] (0,0);
        \draw[fill=black!20, rotate=180] (0,0) to[out=45, in=160] (1.5,0.5) to[out=-20, in=40] (1.7,0) to[out=-40, in=20] (1.5,-0.5) to[out= -160, in=-45] (0,0);
        \draw[fill=black!20, rotate=270] (0,0) to[out=45, in=160] (1.5,0.5) to[out=-20, in=40] (1.7,0) to[out=-40, in=20] (1.5,-0.5) to[out= -160, in=-45] (0,0);
        \node at (1,0) {$\mathcal P$};
        \node at (2.5,1.3) {$\Delta$};    

        \draw[thick, blue, fill=blue!30] (-2,-2)to[out=135, in=180] (-1.5,-0.7) to[out=0,in=-135](0,0) to[out=-135, in=90] (-0.7,-1.5) to[out=-90, in=315] (-2,-2);
        \node at (-1.5,-1.5) {$U$};        
        \end{scope}
        
    \end{tikzpicture}
    \caption{Left: A repelling petal that is a quasidisk. Right: A Jordan region that is disjoint from all (arbitrarily small) repelling petals must have a cusp.}
    \label{fig:repelling_petal} 
\end{figure}
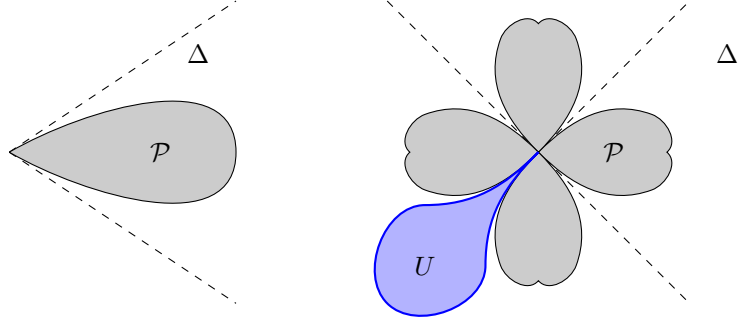

\begin{lemma}\label{lemma:avoid_petals}
    Let $f$ be a parabolic germ at $0$. Let $U\subset \widehat \C$ be a Jordan region with $0\in \partial U$ such that $U$ is disjoint from all attracting (resp.\ repelling) petals. 
    \begin{enumerate}[label=\normalfont(\arabic*)]
        \item Then $U$ is not a quasidisk. Moreover, each arc of $\partial U$ that contains $0$ in its interior is not a quasiarc.
        \item\label{lemma:avoid_petals:2} Suppose that $g$ is a holomorphic map defined in a neighborhood of a closed Jordan region $\bar W$ such that $g|_W$ is injective, $g(W)\subset U$, and $0\in \partial g(W)$. Then each arc $A\subset \partial W$ with $0$ lying in the interior of $g(A)$ is not a quasiarc.
    \end{enumerate}
\end{lemma}

In fact, if $U$ is as in the statement, then it must have a cusp at $0$, so $\partial U$ is not a quasicircle and $U$ is not a quasidisk by Theorem \ref{theorem:quasidisk_quasicircle}. See Figure \ref{fig:repelling_petal} for an illustration.

\begin{proof}
 We justify only the second part. Consider $W,A$ as in the statement. There exists a point $a_0$ that is an interior point of $A$ such that $g(a_0)=0$. Suppose that  $g$ is locally $m$-to-$1$ at $a_0$ for some $m\geq 1$ and suppose that $A$ is a quasiarc. By assumption, the intersection of a small neighborhood of $0$ with $U$ is contained in arbitrarily small sectors $Z$ at $0$. There exists a conformal map $\phi$ in a neighborhood of $a_0$ such that $\phi(a_0)=0$ and $g(z)=\phi(z)^m$. Thus, we see that there exists a small ball $B(0,r)$ and a quasiarc $J\subset B(0,r)\cap \phi(A)$ with endpoints in $\partial B(0,r)$ and with $0\in J$ such that $z\mapsto z^m$ sends a component $\Omega$ of $B(0,r)\setminus J$ univalently into a sector $Z$ at $0$ of angle $\theta$, which can be taken to be arbitrarily small; see Figure \ref{fig:quasiarc_repelling}. Thus, $\Omega$ must be contained in a sector of angle $\theta/m$ at $0$. If $\theta$ is small enough, we obtain a contradiction to the fact that $J$ is a quasiarc. 
\end{proof}

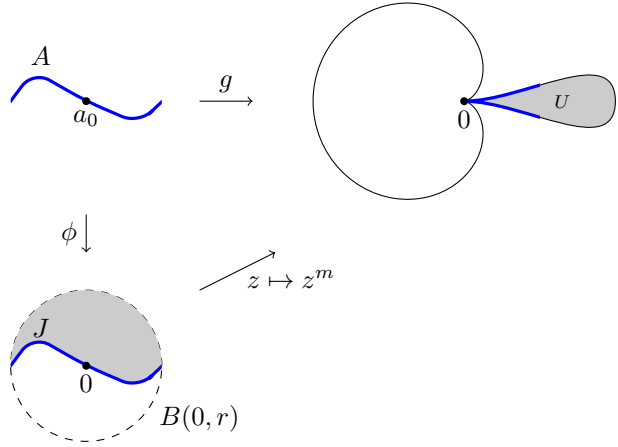
\begin{figure}
        \centering
        \begin{tikzpicture}
            \begin{scope}[shift={(0,4.5)}]
            \clip (0,-4.5) circle (1cm);
            \draw[line width=1.2pt,blue ] (-2,-4)--(-1,-4.5)[rounded corners=7]--(-0.7,-4.1) node[xshift=0.1cm, yshift=0.2cm]{$\color{black}A$}--(0,-4.5)--(0.7,-4.8)--(1,-4.5)[rounded corners=0]--(5,-1)--cycle;
            \end{scope}
            \fill (0,0) circle (1.5pt) node[below] {$a_0$};
            \draw [->] (1.5,0)--(2.2,0) node[above,pos=.5] {$g$};

            \def\a{1};\def\b{0.01};
            \draw[domain=0:360, samples=100, variable=\t]
                plot ({5+(\a*(1 - cos(\t)))*cos(\t) }, {(\a*(1 - cos(\t)))*sin(\t)} );
            \draw[fill=black!20] (7,0) to[out=270, in =0](5,0) to[out=0, in=90] (7,0);
            \begin{scope}
                \clip (5,-1) rectangle (6,1);
                \draw[blue, line width=1.2pt] (7,0) to[out=270, in =0](5,0) to[out=0, in=90] (7,0);
            \end{scope}
            \node at (6.3,0) {$\scriptstyle U$};
            \fill (5,0) circle (1.5pt) node[below] {$0$};

            \begin{scope}[shift={(0,1)}]
            \draw[->] (0,-2.5)--(0,-3) node[left, pos=.5] {$\phi$};
            \draw[dashed] (0,-4.5) circle (1cm);
            \node at (1.5,-5.2) {$B(0,r)$};
            \begin{scope}
            \clip (0,-4.5) circle (1cm);
            \draw[line width=1.2pt,blue , fill=black!20] (-2,-4)--(-1,-4.5)[rounded corners=7]--(-0.7,-4.1) node[xshift=0.1cm, yshift=0.1cm]{$\color{black}J$}--(0,-4.5)--(0.7,-4.8)--(1,-4.5)[rounded corners=0]--(5,-1)--cycle;
            \end{scope}
            \fill (0,-4.5) circle (1.5pt) node[below] {$0$};
            \end{scope}
            \draw[->] (1.5,-2.5)-- (2.5,-2) node[pos=.5, right, yshift=-0.1cm] {$z\mapsto z^m$};
        \end{tikzpicture}
        \caption{A quasiarc $A$ cannot be mapped by a holomorphic map $g$ as in Lemma \ref{lemma:avoid_petals} \ref{lemma:avoid_petals:2} onto an arc of $\partial U$ that contains $0$ in its interior.}
        \label{fig:quasiarc_repelling}
    \end{figure}

Now, if $f'(0)=e^{2\pi i p/q}$, where $p/q\in \mathbb Q$ is in lowest terms, then the multiplicity of the parabolic fixed point $0$ is by definition the multiplicity of $0$ as a fixed point of $f^{\circ q}$, which satisfies $(f^{\circ q})'(0)=1$. Finally, if $0$ is a periodic point of $f$ with period $p$, then $0$ is a parabolic periodic point if $0$ is a parabolic fixed point of $f^{\circ p}$.

If $f$ is a rational map and $0$ is a parabolic fixed point with $f'(0)=1$, for each attraction vector $v$ there exists a \textit{parabolic basin of attraction}, i.e., the set of all points whose orbit converges to $0$ in the direction $v$. The \textit{immediate basin} of $0$ in the direction $v$ is the unique component of the parabolic basin of attraction that maps into itself under $f$. 

\begin{lemma}\label{lemma:preserve_parabolic}
    Let $f$ be a rational map, $0$ be a parabolic fixed point of $f$ with $f'(0)=1$, and $U$ be a Jordan region that is an immediate basin of $0$ corresponding to an attraction vector $v$. Then there exist arbitrarily small repelling petals $P_1,P_2$ corresponding to repulsion vectors that are adjacent to $v$ and there exist non-overlapping Jordan arcs $E_{i}\subset \partial U\cap (P_i\cup \{0\})$ such that $g(E_i)\subset E_i$ for $i=1,2$, where $g$ is the local branch of $f^{-1}$ that fixes $0$.
\end{lemma}

\begin{proof}
    Note that $U$ is a Fatou component and $\partial U$ is contained in the Julia set \cite{Milnor:dynamics}*{Lemma 10.5}. Thus, $\partial U$ cannot intersect any petal corresponding to an attraction vector for $0$. By the parabolic flower theorem \cite{Milnor:dynamics}*{Theorem 10.7}, there exist arbitrarily small petals corresponding to the attraction and repulsion vectors of $f$ whose union, together with $0$, gives a neighborhood of $0$. Also, these petals are Jordan regions cyclically ordered and each attracting petal intersects exactly two repelling petals, except in the case that the parabolic multiplicity is $2$, where there exist only two petals, an attracting one and a repelling one. Since $U$ cannot intersect any attracting petal other than the one corresponding to $v$, we conclude that $\partial U$ contains two non-overlapping arcs $E_1,E_2$ contained in repelling petals $P_1,P_2$, respectively, which correspond to vectors adjacent to $v$. Moreover, since $f(\partial U)=\partial U$, if $g$ is the local branch of $f^{-1}$ that fixes $0$, then $g$ maps each arc $E\subset \partial U$ near $0$ onto an arc $g(E)\subset \partial U$.
    
    If the multiplicity of $0$ is not $2$, then $P_1,P_2$ are disjoint attracting petals of $g$. For $i\in \{1,2\}$, each point of $E_i$ converges to $0$ under iteration of $g$ within the petal $P_i$. Note that $g(E_i)$ is an arc in $\partial U\cap (P_i\cup \{0\})$ with an endpoint at $0$ and either it contains $E_i$ or it is contained in $E_i$. If $g(E_i)\supsetneq E_i$ then the endpoint $z_0$ of $E_i$ that is different from $0$ satisfies $g^{\circ k}(z_0) \in (\partial U\cap P_i)\setminus E_i$ for each $k\in \N$. However, $E_i$ is a neighborhood of $0$ in $\partial U\cap P_i$, so $\lim_{k\to\infty}g^{\circ k}(z_0)\neq 0$, a contradiction. We conclude that $g(E_i)\subset E_i$, as desired. 

    Next, suppose that the multiplicity of $0$ is $2$. Then $P_1=P_2\eqqcolon P$. By shrinking $P$ if necessary, we assume that $g$ is injective in $\bar P$ and $\partial U\setminus  \bar P\neq \emptyset$, so $\partial U \cap  P$ has at least two components $A_i\supset E_i\setminus \{0\}$, $i=1,2$. Note that $A_i$ is an arc in $\partial U \cap P$ with an endpoint at $0$ and another endpoint $a_i\in \partial P$. Hence, $g(A_1)\subset A_1$ and $g(A_2)\subset A_2$ or $g(A_1)\subset A_2$ and $g(A_2)\subset A_1$. The map $g\colon \bar P\to g(\bar P)$ induces a bijection between the set $\{0,a_1,a_2\}\subset \partial P$ and $\{0,g(a_1),g(a_2)\}\subset g(\partial P)$. Since $g$ is orientation-preserving, it preserves the cyclic order of the points $0,a_1,a_2$. From this it follows that $g(a_1)\in \bar {A_1}$ and $g(a_2)\in \bar {A_2}$. Therefore, $g(A_1)\subset A_1$ and $g(A_2)\subset A_2$. Arguing as in the previous paragraph, for $i=1,2$ we see that $g(E_i)\subset E_i$, otherwise, the endpoint of $E_i$ that is different from $0$ would not converge to $0$ under iteration of $g$. 
\end{proof}

\subsection{Tangent quasicircles in parabolic dynamics}\label{section:tangent_quasicircles}

We state the main theorem of this section.

\begin{theorem}\label{theorem:germ_tangent_modification}
    Let $K_1,K_2\subset \widehat \C$ be quasicircles such that $K_1\cap K_2=\{0\}$. Let $f$ be a parabolic germ at $0$ such that
    \begin{enumerate}[label=\normalfont(\arabic*)]
        \item for $j=1,2$ there exists an attraction (resp.\ repulsion) vector $v_j$, where $v_1\neq v_2$, and a petal $P_j$ for the direction of $v_j$ and
        \item for $i=1,2$ the set $K_i\cap (P_j\cup \{0\})$ contains a Jordan arc $E_{ij}$ with an endpoint at $0$ that satisfies $f(E_{ij})\subset E_{ij}$ (resp.\ $f^{-1}(E_{ij})\subset E_{ij}$).
    \end{enumerate} 
    Then there exists a quasiconformal map $\Phi\colon \widehat \C\to \widehat \C$ that maps $K_1$ and $K_2$ onto two tangent circles. 
\end{theorem}

The petals $P_1,P_2$ are allowed to be of the same type, or one can be attracting and the other repelling. See Figures \ref{fig:germ3} and \ref{fig:germ2} for an illustration of the assumptions. We prove a preliminary statement for invariant Jordan arcs.

\begin{lemma}\label{lemma:arc_invariant_general}
    Let $E$ be a Jordan arc with an endpoint at $0$ and $f$ be a homeomorphism from $E$ onto a subarc $f(E)$ of $E$ such that $0$ is the unique fixed point of $f$. Then $\bigcap_{k=0}^\infty f^{\circ k}(E)=\{0\}$. If, in addition, $f$ is a parabolic germ at $0$, then there exists an attraction vector $v$ for $f$ such that for any attracting petal $P$ for the direction $v$ we have $f^{\circ k}(E)\subset P\cup \{0\}$ for all sufficiently large $k\in \N$.
\end{lemma}

A consequence of this lemma is that for each sector $S$ at $0$ containing the attraction vector $v$ in its interior the arc $E$ has a subarc $F\subset S\cup \{0\}$ with an endpoint at $0$. Because of this, we say that the arc $E$ is \textit{asymptotic} to the vector $v$.

\begin{proof}
    For the first part, by conjugating $f$ with a homeomorphism between $E$ and $[0,1]$, we may suppose that $E=[0,1]$, $f(E)=[0,a_1]$ for some $a_1\in (0,1)$ and $f\colon [0,1]\to [0,a_1]$ is an increasing homeomorphism having $0$ as the unique fixed point. Let $a_0=1$ and $a_k=f(a_{k-1})$, $k\in \N$. Then $a_{k+1}<a_{k}$ and $f^{\circ k}([0,1])=[0,a_k]$ for $k\in \N\cup \{0\}$. Thus, $\{a_k\}_{k\in \N}$ converges to a point of $[0,1]$ that is necessarily a fixed point of $f$. By assumption $0$ is the only such point. Thus, $\bigcap_{k=0}^\infty f^{\circ k}([0,1])=\{0\}$, as desired.

    Next, suppose that $f$ is a parabolic germ at $0$. By the first part, each point of $E$ is attracted to $0$ under iteration of $f$. By \cite{Milnor:dynamics}*{Lemma 10.1}, each point of $E$ has a neighborhood that is attracted to $0$ along one of the attracting directions of $f$. By the connectedness of $E$, all points of $E$ are attracted along the same attraction vector $v$. Let $P$ be an arbitrary attracting petal in the direction of $v$. Let $0,z_0$ be the endpoints of $E$. For $k\in \N$, let $z_k=f^{\circ k}(z_0)$ and $E(k)$ be the arc of $E$ from $z_{k-1}$ to $z_k$. Note that $f(E_k)=E_{k+1}$ for $k\in \N$. By the above, there exists $N\in \N$ such that $E(k)\subset P$ for $k\geq N$. Since $E=\bigcup_{k=1}^\infty E(k)\cup\{0\}$, it follows that $f^{\circ k}(E)\subset P\cup \{0\}$ for $k\geq N$. 
\end{proof}

Next, we reduce the proof of Theorem \ref{theorem:germ_tangent_modification} to the case that $K_1=\R\cup \{\infty\}$.

\begin{lemma}\label{lemma:germ_tangent_reduction}
    Let $f,K_1,K_2$ be as in the assumptions of Theorem \ref{theorem:germ_tangent_modification}. Then there exists a quasiconformal map $\psi\colon \widehat \C\to \widehat \C$ and a parabolic germ $g$ at $0$ such that $\psi(K_1)=\R\cup \{\infty\}$ and $g,\psi(K_1),\psi(K_2)$ satisfy the assumptions of Theorem \ref{theorem:germ_tangent_modification}.
\end{lemma}

\begin{proof}
    By replacing $f$ with $f^{-1}$, we may suppose that $P_1$ is an attracting petal corresponding to an attracting vector $v_1$. For $i=1,2$ consider a Jordan arc $E_{i1}\subset K_i\cap (P_1\cup \{0\})$ with an endpoint at $0$ such that $f(E_{i1})\subset f(E_{i1})$. Since $f$ is a homeomorphism in $P_1$ and each point of $P_1$ is attracted to $0$ under iteration, we conclude that for each $z\in E_{i1}\setminus \{0\}$ the point $f(z)$ lies in the open subarc of $E_{i1}$ from $0$ to $z$.
    
    Let $D$ be the Jordan region whose boundary is $K_1$ and $K_2\subset \bar D$. By Theorem \ref{theorem:quasidisk_quasicircle}, $D$ is a quasidisk. By Proposition \ref{proposition:quasidisk_extension} there exists a quasiconformal map $\psi$ of $\widehat \C$ such that $\psi(K_1)=\R\cup \{\infty\}$, $\psi(0)=0$, $\psi(D)=\mathbb H$, and $\psi|_D$ is conformal. Here $\mathbb H$ denotes the upper half-plane. We conjugate $f$ with $\psi$ to obtain a holomorphic map $g$ in a neighborhood of $0$ in $\mathbb H$. Note that $g$ extends continuously to the real line near $0$ and maps it into itself. Thus, by Schwarz reflection we may extend $g$ to a holomorphic map in a neighborhood of $0$ that fixes $0$.

    Without loss of generality, suppose that $E_{11}$ is mapped under $\psi$ to an interval $(0,s_0)$. By the behavior of $f$ in $E_{11}$, we conclude that $0<g(x)<x$ for each $x\in(0,s_0)$. This implies that $0\leq g'(0)\leq 1$.  Note that the petal $P_2$ for $f$ is either attracting or repelling. By assumption, for $i=1,2$ there exists an arc $E_{i2}\subset K_i\cap (P_2\cup \{0\})$ with properties analogous to those of $E_{i1}$ above. Let $(-t_0,0)=\psi(E_{12})$. We claim that $g$ is a parabolic germ at $0$. If $P_2$ is a repelling petal, then we have $g(x)<x$ for each $x\in (-t_0,0)$, which implies that $g'(0)\geq 1$. Thus, $g'(0)=1$. Clearly $g$ is not the identity map, so it is a parabolic germ at $0$. 
    
    Next, suppose that $P_2$ is an attracting petal. Fix $j\in \{1,2\}$. By Lemma \ref{lemma:arc_invariant_general}, if $S_j$ is an arbitrary sector at $0$ that contains $v_j$ then $f^{\circ k}(E_{ij})\subset S_j\cup \{0\}$ for all sufficiently large $k\in \N$ and for $i=1,2$. Thus, the arc $E_{ij}$ is asymptotic to the attracting vector $v_j$. Since $f$ is parabolic, there exists a repulsion vector $v$ that is between the attraction vectors $v_1,v_2$ and a repelling petal $P$ for the direction $v$ that is disjoint from small sectors $S_j$, $j=1,2$, as above. In particular, we may find a sufficiently small repelling petal $P$ as above that is contained in $D$. That is, each point of $P$ is attracted to $0$ under iteration of $f^{-1}$. This proves that there are points of $\mathbb H$ that are attracted to $0$ under iteration of $g^{-1}$. If $|g'(0)|<1$, then all points in a neighborhood of $0$ would be attracted to $0$ under iteration of $g$ by \cite{Milnor:dynamics}*{Lemma 8.1}, a contradiction. Thus, the only remaining case is that $g'(0)=1$. Hence, $g$ is again a parabolic germ.

    It remains to prove that the conditions of Theorem \ref{theorem:germ_tangent_modification} remain true for $g$. Let $F_{11}=(0,s_0)=\psi(E_{11})$ and $F_{21}=\psi(E_{21})$, which satisfy $g(F_{i1})\subset F_{i1}$ for $i=1,2$ by the dynamics of $f$. Clearly, each point of $F_{11}$ is attracted to $0$ in the direction of the positive real axis. We claim that the same is true for $F_{21}$. By the dynamics of $f$, each point of $F_{21}$ is attracted to $0$ under iteration of $g$. By Lemma \ref{lemma:arc_invariant_general} there exists an attraction vector $v\in \bar{\mathbb H}$ for $g$ such that $F_{21}$ is asymptotic to $v$. If $v$ is not in the direction of the positive real axis, then there exists a repulsion direction between the positive real axis and the vector $v$. In particular, there exists a repelling petal $Q$ for $g$ in $\mathbb H$ that lies between (in the sense of cyclic order) $F_{11}$ and $F_{21}$. This implies that there exists a repelling petal $\psi^{-1}(Q)$ for $f$ that is between $E_{11}$ and $E_{21}$. Note that $E_{11}$ and $E_{21}$ are asymptotic to $v_1$, in view of Lemma \ref{lemma:arc_invariant_general}. Thus, if $Q$ is sufficiently small, then $\psi^{-1}(Q)$ is necessarily contained in a small sector that contains $v_1$ and that does not contain repulsion vectors of $f$. So, points of $\psi^{-1}(Q)$ cannot be attracted to $0$ in the direction of a repulsion vector of $f$ under iteration of $f^{-1}$. This is a contradiction. Thus, the claim is proved. Now, if we consider any attracting petal $Q_1$ for $g$ for the direction of the positive real axis, then by replacing $F_{i1}$ with an iterate, in view of Lemma \ref{lemma:arc_invariant_general}, we may have $F_{i1}\subset Q_1\cup \{0\}$ for $i=1,2$. In the same way we argue for $F_{i2}=\psi(E_{i2})$, $i=1,2$.
    \end{proof}

    \begin{proof}[Proof of Theorem \ref{theorem:germ_tangent_modification}]
        By Lemma \ref{lemma:germ_tangent_reduction}, we may assume that $K_1=\R\cup \{\infty\}$. Without loss of generality suppose that $K_2$ is in the closure of the lower half-plane. Since the real line $\R$ is locally preserved under $f$ near $0$, we conclude that $P_1$ and $P_2$ are parabolic petals in the direction of the positive and negative real axis, respectively. By replacing $f$ with its inverse if necessary, we may assume that $P_1$ is an attracting petal in the direction of the positive real axis. Thus, $f$ has an expansion at $0$ of the form
        $$f(z)=z+az^{n+1}+O(z^{n+2}),$$
        where $a<0$ and $n\in \N$.  We conjugate $f$ by $z\mapsto 1/z$ to the parabolic germ $g$ at $\infty$ with expansion
        $$g(z)=z-\frac{a}{z^{n-1}}+O(z^{-n}).$$
        
        Let $J$ be the image of $K_2$ under $z\mapsto 1/z$. Then $J$ contains an arc $J^+$ with an endpoint at $\infty$ that is invariant and is attracted to $\infty$ by $g$ in the direction of the positive real axis. Similarly, $J$ contains an arc $J^-$ with an endpoint at $\infty$ that is invariant and is attracted to $\infty$ by $g$ or $g^{-1}$ in the direction of the negative real axis. In order to prove that there exists a quasiconformal homeomorphism of $\widehat \C$ that maps $\R$ and $J\setminus \{\infty\}$ to parallel lines, by Theorem \ref{theorem:tangent:graph} it suffices to prove that there exists $L\geq 1$ such that
        \begin{align}\label{theorem:germ_tangent:inclusion}
        J\setminus \{\infty\}\subset \{(x,y)\in \R^2:  L^{-1}(1+|x|)^{-n+1}\leq y\leq L(1+|x|)^{-n+1}\}.
        \end{align}
        It suffices to show the inclusion for the arcs $J^+$ and $J^-$. We focus on $J^+$.
        
        Let $w = \phi(z) = z^n/(-na)$, which maps the sector $\Delta=\{re^{i\theta}: r>0,\,\, |\theta|<\pi/n\}$ conformally onto the slit plane $\C\setminus (-\infty,0]$.  We set
        $$G(w)= \phi \circ g\circ \phi^{-1}(w) = w+1+o(1)$$ 
        as $w\to\infty$. The calculation of the expansion can be found in \cite{Milnor:dynamics}*{(10.2), p.~107}.  Let $R>0$ be sufficiently large so that 
        \begin{align}\label{eqn:parabolicestimate}
            |G(w)-w-1|<1/4. 
        \end{align}
        In particular, if $w\in \mathbb H_{>R}$, then $\Re(G(w))>\Re(w)+3/4$, so $G(\mathbb H_{>R})\subset \mathbb H_{>R}$. Note that for any $k\in\N$ the map $G^{\circ k}\colon \mathbb H_{>R}\to \mathbb H_{>R}$ is a conformal embedding. Since the real axis is locally invariant under $G$ near $\infty$, we may consider an even larger $R>0$ so that $G^{-1}(\mathbb H_{>R}\cap \R)\subset \R$. This implies that $G(\mathbb H_{>R}\cap \mathbb H)\subset \mathbb H_{>R}\cap \mathbb H$.

        Since $J^+$ is invariant under $g$ and is attracted to $\infty$ in the direction of the positive real axis, by replacing it with an iterate if necessary (see Lemma \ref{lemma:arc_invariant_general}), we may assume that $J^+\subset \Delta$. Moreover, each arbitrarily small sector at $0$ that contains the positive real axis contains some iterate of $J^+$ under $g$. Thus, the same is true for iterates of $\phi(J^+)$ under $G$ (given that $\phi$ is a power map and maps small sectors to small sectors). Thus, by passing to an iterate of $J^+$ under $g$, we may assume that $\phi(J^+)\subset \mathbb H_{>R}$. 

        Let $p_0\in \R\cap \mathbb H_{>R}$ and  $p_k=G(p_{k-1}) = G^{\circ k}(p_0)$, $k\in \N$. Also, let $q_0 \in \phi(J^+)$ and $q_k=G(q_{k-1}) = G^{\circ k}(q_0)$, $k\in \N$. For $k\in \N$ denote by $J^+(k)$ the subarc of $\phi(J^+)$ from $q_{k-1}$ to $q_k$. Consider a ball $B$ such that $\bar B\subset \mathbb H_{>R}$ and $[p_0,p_1]\cup J^+(1) \subset B$.  Note that there exists $t>1$ such that $tB\subset \mathbb H_{>R}$. By Koebe's distortion theorem \cite{Pommerenke:conformal}*{Theorem 1.3}, we have
        \begin{align}\label{theorem:germ:k1}
            |G^{\circ k}(w_1)- G^{\circ k}(w_2)| \simeq  |(G^{\circ k})'(p_0)|\cdot  |w_1-w_2|
        \end{align}
        for all distinct $w_1,w_2\in B$ and $k\in \N$. Thus, for each $w\in J^+(1)$ and $k\in \N$ we have
        $$\frac{|p_k -  G^{\circ k}(w)|}{|p_k-p_{k+1}|} \simeq  \frac{|p_0-w|}{|p_0-p_1|}.$$
        By compactness, we have $|p_0-w|\simeq 1$ for $w\in J^+(1)$, so 
        $$\frac{|p_k -  G^{\circ k}(w)|}{|p_k-p_{k+1}|} \simeq 1$$
        for all $k\in \N$ and $w\in J^+(1)$. Also, by {\eqref{eqn:parabolicestimate}}, we have 
        \begin{align}\label{theorem:germ:k2}
            3/4\leq |p_k-p_{k+1}| \leq 5/4 
        \end{align}
        for all $k\in \N\cup\{0\}$. Thus, 
        $$ |p_k-G^{\circ k}(w)| \simeq 1$$
        for all $k\in \N$ and $w\in J^+(1)$. This implies that $\Im (w)\lesssim 1$ for all $w\in J^+(k)$, $k\in \N$.

        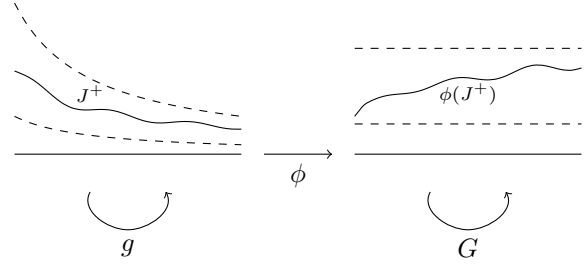
\begin{figure}
        \centering
        \begin{tikzpicture}
            \begin{scope}
                \draw (0,0)--(3,0);
                \draw[domain=0:3, smooth, variable=\x]
    plot ({\x}, {1/(\x+1)*(1+0.1*sin(6*\x r))+0.1});
                \draw[->] (1,-0.5) to[out=-120, in=180] (1.5,-1) to[out=0,in=-60] (2,-0.5);
                \node[below] at (1.5,-1) {$g$};
                \node at(1,0.8) {$\scriptstyle J^+$};
                \draw[dashed,domain=0:3, smooth, variable=\x]
    plot ({\x}, {0.5/(\x+1)});
    \draw[dashed,domain=0:3, smooth, variable=\x]
    plot ({\x}, {2/(\x+1)});
            \end{scope}
               
                \draw[->] (3.3,0)--(4.2,0) node[below,pos=0.5] {$\phi$};
                
            \begin{scope}[shift={(4.5,0)}]
                \draw (0,0)--(3,0);
                \draw[domain=0:3, smooth, variable=\x]
    plot ({\x}, {0.5+0.4*sqrt(\x)*(1+0.1*sin(6*\x r))});
                \draw[->] (1,-0.5) to[out=-120, in=180] (1.5,-1) to[out=0,in=-60] (2,-0.5);
                \node[below] at (1.5,-1) {$G$};
                \node at(1.5,0.8) {$\scriptstyle\phi(J^+)$};
                \draw[dashed] (0,0.4)--(3,0.4);
                \draw[dashed] (0,1.4)--(3,1.4);
            \end{scope}
        
        \end{tikzpicture}
        \caption{The map $g$ is conjugate to $G$ under $\phi$. The set $\phi(J^+)$ lies between two horizontal lines so $J^+$ is close to the graph of a negative power function.}
        \label{fig:asymptotics}
    \end{figure}

    Next, we prove that $\Im (w)$ is also bounded from below. By \eqref{theorem:germ:k1} and \eqref{theorem:germ:k2}, we have $|(G^{\circ k})'(p_0)|\simeq 1$ for all $k\in \N$. By Koebe's distortion theorem \cite{Pommerenke:conformal}*{Theorem 1.3} we have $|(G^{\circ k})'(w)|\simeq 1$ for all $k\in \N$ and for all $w\in B$. Let $r>0$ such that $B(w,r)\subset \mathbb H_{>R}\cap \mathbb H$ for each $w\in J^+(1)$.  Let $w\in J^+(1)$ be arbitrary. We apply Koebe's distortion theorem \cite{Pommerenke:conformal}*{Corollary 1.4} to the conformal map $G^{\circ k}|_{B(w,r)}$, which takes values in $\mathbb H_{>R}\cap \mathbb H$, and we obtain 
    $$\dist(G^{\circ k}(w), \R) \geq \dist(G^{\circ k}(w) , \partial G^{\circ k}(B(w,r))) \simeq r |(G^{\circ k})'(w)|\simeq 1$$
    for every $k\in \N$.  This proves that $\dist(J^+(k),\R)\gtrsim 1$ for each $k\in \N$. Thus, $\dist(\phi(J^+),\R)\gtrsim 1$, as desired.
    
    Therefore, $\phi(J^+)$ is between two horizontal lines $\Im(w)=\tau_1$ and $\Im(w)=\tau_2$, where $0<\tau_1<\tau_2$. If $\tau\in \R$ is fixed and  $\Im (\phi (x+iy))=\tau$, note that as $x\to \infty$ we have $y=o(x)$. A direct computation shows that 
        $$y=\frac{ -a \tau }{x^{n-1}}(1+o(1)).$$
    This proves that there exists $L\geq 1$ such that
        $$ J^+\subset \{(x,y)\in \R^2:  L^{-1} (1+|x|)^{-n+1} \leq y\leq L (1+|x|)^{-n+1}\},$$
    as illustrated in Figure \ref{fig:asymptotics}.
    The proof is complete.
    \end{proof}

\begin{proof}[Proof of Theorem \ref{theorem:qr_germ}]
    Suppose that there exists a quasiconformal map $\phi\colon \widehat \C\to \widehat \C$ that maps $K_1\setminus \{0\}$ and $K_2\setminus \{0\}$ to lines parallel to the real axis. Let $g(z)=z+1$ and $f=\phi^{-1}\circ g\circ \phi(z)$. Then $f$ satisfies the assumptions in the statement. 

    Conversely, by conjugating $f$ near $0$ with a quasiconformal map $\psi$ as in the statement, we can reduce to the case that $f$ is a parabolic germ. Note that the conjugacy $\psi$ is only defined on subarcs $K_1',K_2'$ of $K_1,K_2$ rather than on the entire curves, so we invoke Theorem \ref{theorem:rickman} to extend $\psi(K_i')$, $i=1,2$ to suitable quasicircles. By Theorem \ref{theorem:germ_tangent_modification}, there exists a quasiconformal map of $\widehat \C$ that maps these quasicircles to tangent circles. Finally, by Lemma \ref{lemma:tangent_local} we obtain the desired conclusion for the initial quasicircles $K_1,K_2$.
\end{proof}

\section{Tangent Fatou components}

In this section we provide sufficient conditions so that two Fatou components of a geometrically finite rational map that have a common boundary point can be quasiconformally mapped to two disks. 

\subsection{Component dependent bounds}
We first establish a general lemma about holomorphic germs. 

\begin{lemma}\label{lemma:lift}
    Let $f_1,f_2$ be holomorphic germs at $0$ such that $f_1(z)=az^m+O(z^{m+1})$ and $f_2(z)=\lambda z+bz^{n+1}+O(z^{n+2})$ for some $a,\lambda \neq 0$ and $m,n\in \N$. 
    \begin{enumerate}[label=\normalfont(\arabic*)]
        \item\label{lemma:lift:existence} There exists a holomorphic germ $h$ at $0$ such that $f_2\circ f_1=f_1\circ h$ and $h'(0)^m=\lambda$. 
        \item\label{lemma:lift:choices} There are exactly $m$ choices for $h$, one for each local branch of $f_1^{-1}$ so that $h=f^{-1}_1\circ f_2\circ f_1$.
        \item\label{lemma:lift:1} If $\lambda=1$, then we may take
    $$h(z)=z+ c z^{nm+1}+O(z^{nm+2}) \quad \text{as $z\to 0$},$$
    where $c\neq 0$ if and only if $b\neq 0$.
    \end{enumerate}
    \end{lemma}

\begin{proof}
    The map $f_1$ is locally $m$-to-$1$ near $0$, so there exists a conformal map $\phi$ in a neighborhood of $0$ that fixes $0$ and satisfies $f_1(z)=\phi(z)^m$. Next, we have
    $$f_2(z^m)= \lambda z^m+ bz^{m(n+1)}+O(z^{m(n+2)})= \lambda z^m \left( 1+ \frac{b}{\lambda}z^{mn}+O(z^{m(n+1)})\right).$$
    Since this map is locally $m$-to-$1$ near $0$, there exists a conformal map $\widetilde h$ in a neighborhood  of $0$ that fixes $0$ and satisfies $f_2(z^m)= \widetilde h(z)^m$; there are $m$ choices of $\widetilde h$, differing by multiplication with an $m$-th root of unity and we fix one such map. Thus, 
    $$\widetilde h(z)=\mu z \left( 1+ \frac{b}{\lambda}z^{mn}+O(z^{m(n+1)})\right)^{1/m}= \mu z+\frac{\mu b}{m\lambda} z^{nm+1}+O(z^{nm+2})$$
    for a choice $\mu$ of the $m$-th root of $\lambda$. Finally, we set $h=\phi^{-1}\circ \widetilde h\circ \phi$, which satisfies $h'(0)=\mu$. Observe that $h$ satisfies the claimed identity $f_2\circ f_1=f_1\circ h$. This proves \ref{lemma:lift:existence}.

    If $z\mapsto z^{1/m}$ is a branch of the $m$-th root defined in an appropriate open set not containing $0$, then $f_1^{-1}(z)=\phi^{-1}(z^{1/m})$ is a branch of $f_1^{-1}$. Let $g=f_1^{-1}\circ f_2\circ f_1$ and $\widetilde g= \phi\circ g\circ \phi^{-1}$, defined in appropriate open sets. We have
    $$\widetilde g(z)^m= \phi( g(\phi^{-1}(z)))^m= f_1( g(\phi^{-1}(z)))= f_2( f_1(\phi^{-1}(z)))=f_2(z^m)= \widetilde h(z)^m.$$
    This implies that $\widetilde h$ differs from $\widetilde g$ by multiplication with an $m$-th root of unity. So, $\widetilde g$ and $g$ may be extended in a neighborhood of $0$. Thus, the chosen branch of $f_1^{-1}$ yields a map $g$ satisfying \ref{lemma:lift:existence}, as desired. 
    
    Next, suppose that $\lambda=\mu=1$. We have
    \begin{align*}
       \phi(z+\delta)-\phi(z)-\phi'(z)\delta&= \int_{[z,z+\delta]}(\phi'(\zeta) -\phi'(z))\, d\zeta\\ &= \int_{[z,z+\delta]} \int_{[z,\zeta]}\phi''(s)\, ds\, d\zeta =O(\delta^2).
    \end{align*}
    We expand $h(z)=z+ cz^{l}+O(z^{l+1})$, where $c\neq 0$. For $\delta=h(z)- z$ we have
    \begin{align*}
       \phi(h(z))= \phi( z+\delta) &= \phi( z)+\phi'( z)\delta+ O(\delta^2)\\
       &= \phi(z)+(\phi'(0)+O(z))(cz^l +O(z^{l+1}))+O(z^{2l}). 
    \end{align*}
    Moreover, 
    \begin{align*}
        \widetilde h(\phi(z))= \phi(z)+ \frac{b}{m} \phi(z)^{nm+1}+O(z^{nm+2}).
    \end{align*}
    Thus, 
    $$\frac{b}{m}\phi(z)^{nm+1}+O(z^{nm+2})= \phi'(0)c z^l+ O(z^{l+1}).$$
    If $b\neq 0$, it follows that $l=nm+1$ and $\frac{b}{m}\phi'(0)^{nm+1}=\phi'(0)c$. If $b=0$, then $l>nm+1$. This proves \ref{lemma:lift:1}.
\end{proof}

Each Fatou component of a rational map $f$ is pre-periodic by Sullivan's theorem \cite{Milnor:dynamics}*{Appendix F}. We denote the pre-period of a Fatou component $\Omega$ by $l_\Omega$. That is, $l_{\Omega}$ is the smallest number $n\in \N\cup \{0\}$ such that $f^{\circ n}(\Omega)$ is periodic. For the convenience of the reader, we include a well-known result on geometrically finite rational maps. 

\begin{lemma}\label{lemma:no_infinite}
    Let $f$ be a geometrically finite rational map.
    \begin{enumerate}[label=\normalfont(\arabic*)]
        \item\label{lemma:no_infinite:0} If $U,V$ are distinct Fatou components such that $V$ is simply connected and fixed, and $a_0\in \bar U\cap \bar V$ is a fixed point, then $f(U)\neq V$. 
        \item\label{lemma:no_infinite:00} If $U,V$ are distinct Fatou components such that $V,f(V)$ are Jordan regions and $a_0\in \bar U\cap \bar V$ is not a critical point, then $f(U)\neq f(V)$. 
        \item\label{lemma:no_infinite:1} If a periodic point lies in the boundary of a simply connected Fatou component $U$ of $f$, then $U$ is periodic. 
        \item\label{lemma:no_infinite:2} Each pre-periodic point lies in the boundaries of at most finitely many simply connected Fatou components.
    \end{enumerate}
\end{lemma}
\begin{proof}
    By \cite{LeiYongcheng:local_connectivity}*{Lemma 2.1}, $\partial V$ is locally connected. By \cite{TanYin96long}*{Corollary 3.3}, the point $a_0$ has finitely many prime ends in $V$. In particular, there exist finitely many cross cuts bounding disjoint regions $Z_1,\dots,Z_k \subset U$ whose union is a neighborhood $Z$ of $a_0$ in $V$. Note that $a_0$ cannot be a critical point. So, by taking $Z$ sufficiently small, we assume that $f$ is injective in $Z$. This implies that the regions $f(Z_1),\dots,f(Z_k)$ are disjoint and are in correspondence with the prime ends of $a_0$; one can see this by conformally conjugating $f|_V$ to a Blaschke product in $\D$. Therefore, $f(Z)$ is a neighborhood of $a_0$ in $V$. Since $f$ is injective near $a_0$, no point of $U$ near $a_0$ can be mapped to a point of $f(Z)$. Since $f(Z)$ is a neighborhood of $a_0$ in $V$, no point of $U$ near $a_0$ can be mapped to a point of $V$. Thus  $f(U)\neq V$ and \ref{lemma:no_infinite:0} is true.

    Next, we prove \ref{lemma:no_infinite:00}. By the above argument, given that $a_0,f(a_0)$ correspond to a unique prime end in $V,f(V)$, respectively, we see that there exists a neighborhood $Z$ of $a_0$ in $V$ that is mapped to a neighborhood $f(Z)$ of $a_0$ in $f(V)$. Now we argue as above. Since $f$ is injective near $a_0$, no point of $U$ near $a_0$ can be mapped to a point of $V$ near $a_0$. Thus, $f(U)\neq f(V)$. 
    
    We prove \ref{lemma:no_infinite:1}. Suppose that $U$ is a strictly pre-periodic Fatou component and $a_0\in \partial U$ is a periodic point. By replacing $f$ with an iterate, suppose that $a_0$ is fixed and $V=f(U)$ is fixed. Since $f(\partial U)=\partial V$ and $\partial U$ is connected, the set $\partial V$ is connected, so $V$ is simply connected. However, this contradicts \ref{lemma:no_infinite:0}, so \ref{lemma:no_infinite:1} holds. Now we justify \ref{lemma:no_infinite:2}. By part \ref{lemma:no_infinite:1} and the finiteness of periodic Fatou components, each periodic point lies in the boundary of at most finitely many simply connected Fatou components. Thus, each pre-periodic point has the same property. 
\end{proof}

We now establish the main result of the section. 

\begin{proposition}\label{proposition:component_dependent}
    Let $f$ be a geometrically finite rational map and $U, V$ be Fatou components whose orbits $\{f^{\circ n}(U)\}_{n\geq 0}$, $\{f^{\circ n}(V)\}_{n\geq 0}$ consist of Jordan regions. Suppose that there exists an open set $W$ such that $W\cap \bar U\cap \bar V=\{a_0\}$ and the sets $\partial U\cap \bar W$ and $\partial V\cap \bar W$ are quasiarcs. Assume that one of the following alternatives holds.
    \begin{enumerate}[label=\normalfont(\arabic*)]
        \item\label{proposition:component_dependent:1} $a_0$ is mapped to a parabolic periodic cycle of multiplicity $3$ by an iterate of $f$ that is a local homeomorphism near $a_0$.
        \item\label{proposition:component_dependent:2} $a_0$ is mapped to a parabolic periodic cycle of multiplicity $2$ by an iterate of $f$ that is locally two-to-one at $a_0$.
    \end{enumerate}
    Then there exists a quasiconformal map $\phi$ of $\widehat \C$ that maps $\partial U\cap \bar W$ and $\partial V\cap \bar W$ onto two circular arcs that are tangent to each other.
\end{proposition}
\begin{proof}
    Let $l$ be the pre-period of $a_0$ and let $p$ be a multiple of the period of $b_0=f^{\circ l}(a_0)$ so that $f^{\circ p}$ fixes the attracting directions at the parabolic point $b_0$, i.e., $(f^{\circ p})'(b_0)=1$. We apply Lemma \ref{lemma:lift} for the holomorphic germs $f_1=f^{\circ l}$ and $f_2=f^{\circ p}$ at $a_0$. In both alternative assumptions the map $f_1=f^{\circ l}$ is locally $m$-to-$1$ at $a_0$ and the map $f_2=f^{\circ p}$ has a parabolic fixed point at $b_0$ with multiplicity $n+1$ for some $n,m\in \N$. By Lemma \ref{lemma:lift} we obtain a parabolic germ $h$ at $a_0$ with multiplicity $nm+1$ such that $f_2\circ f_1=f_1\circ h$. Under assumption \ref{proposition:component_dependent:1}, we have $m=1$ and $n=2$, so the multiplicity of the parabolic point $a_0$ of $h$ is $3$. Under assumption \ref{proposition:component_dependent:2}, we have $m=2$ and $n=1$, so the multiplicity of the parabolic point $0$ of $h$ is again $3$.

    Under either assumption, we claim that $f_1(U)$ is an immediate basin of the parabolic point $b_0$. If not, then $f_1(U)$ is disjoint from all attracting petals at $b_0$. By Lemma \ref{lemma:avoid_petals}, any arc of  $f_1(\partial U)$ that contains $b_0$ in its interior is not a quasiarc and $\partial U\cap W$ is not a quasiarc, a contradiction. Similarly, $f_1(V)$ is an immediate basin of $b_0$.     
    
    Suppose \ref{proposition:component_dependent:1} holds. The map $f_1$ is a local homeomorphism near $a_0$ and we have $h=f_1^{-1}\circ f_2\circ f_1$, where $f_1^{-1}$ is a local inverse branch of $f_1$ that maps $b_0$ to $a_0$.  Given that $a_0$ is not a critical point of $f_1$, by Lemma \ref{lemma:no_infinite} \ref{lemma:no_infinite:00} the sets $f_1(U)$ and $f_1(V)$ are distinct Fatou components each contained in one of the two parabolic basins of the map $f_2$ for the point $b_0$. By Lemma \ref{lemma:preserve_parabolic} each of the boundaries $f_1(\partial U)$ and $f_1(\partial V)$ contains two non-overlapping arcs with a common endpoint at $b_0$, contained in opposite repelling petals of $f_2$, and each of these arcs is mapped into itself under a branch of $f_2^{-1}$. It follows that each of $\partial U,\partial V$ contains two non-overlapping arcs with a common endpoint at $a_0$ contained in opposite repelling petals of $h$, and each of these arcs is mapped into itself under $h^{-1}$. By Theorem \ref{theorem:rickman}, we can extend the quasiarcs $\partial U\cap \bar W$ and $\partial V\cap \bar W$ to quasicircles $K_1,K_2$ as in Theorem \ref{theorem:germ_tangent_modification}. Therefore, by Theorem~\ref{theorem:germ_tangent_modification}, there exists a quasiconformal map of $\widehat \C$ that maps $\partial U\cap \bar W$ and $\partial V\cap \bar W$ onto two circular arcs, as desired. 

    Next, we assume that \ref{proposition:component_dependent:2} holds. Then $f_1$ is locally $2$-to-$1$ at $a_0$. Consider a small repelling petal $P$ of $f_2$ at $b_0$ that is a Jordan region. By Lemma \ref{lemma:preserve_parabolic}, $f_1(\partial U)\cap P$ contains two non-overlapping arcs with an endpoint at $b_0$ and each of these arcs is mapped into itself under a branch of $f_2^{-1}$. The set $f_1^{-1}(P)$ has two components $P_1,P_2$, and in the closure of each of them $f_1$ topologically conjugates $h$ to $f_2$. Thus, $P_1,P_2$ are repelling petals of $h$. Moreover, each of $\partial U,\partial V$ contains two non-overlapping arcs with a common endpoint at $a_0$, contained in opposite repelling petals of $h$, and each of these arcs is mapped into itself under $h^{-1}$. Therefore, we are in the setting of Theorem \ref{theorem:germ_tangent_modification} as before.    
\end{proof}

\subsection{Component independent bounds}

If $U\subset \widehat \C$ is a hyperbolic region, we denote by $\rho_U$ the hyperbolic density in $U$. Also, recall that $\Hyp$ is the upper half-plane. 

\begin{proposition}\label{prop:comparisonhyperbolicmetric}
    Let $V,Q \subset \Hyp$ be Jordan regions such that $V\subset Q$ and the sets
    $$
    I = \partial V \cap \R  \,\,\,\text{and}\,\,\, J= \partial Q \cap \R
    $$
    are closed intervals so that $\dist(I,\R\setminus J)>0$. Assume that ${Q \setminus \bar V}$ is a quadrilateral, whose vertical sides are the closures of the two components of $J\setminus I$, and  $\mod (Q\setminus \bar V) \geq K^{-1}$ for some $K\geq 1$. Then 
    $$\rho_{\mathbb H}(z)\leq \rho_Q(z)\leq C(K) \rho_{\mathbb H}(z)\,\,\, \text{for all $z\in V$.} $$
\end{proposition}
\begin{proof}
    The lower bound follows from the Schwarz lemma; see \cite{BeardonMinda:hyperbolic}*{Theorem 8.1}. To prove the upper bound, we use the estimate (see \cite{BeardonMinda:hyperbolic}*{Theorem 8.2})
    \begin{align}\label{prop:comparison:ineq}
        \rho_Q(z) \leq \frac{2}{\dist(z, \partial Q)}, \,\,\, z\in Q.
   \end{align}
    We write $\partial Q = J \cup \Gamma$, where $\Gamma = \overline{\partial Q \cap \Hyp}$. Let  $Q'$  be the double of $Q$ along $J$ and $V'$ be the double of $V$ along $I$. Let $A = Q' \setminus \bar {V'}$, which is an annulus. Since $\mod (Q\setminus \bar V)\geq K^{-1}$, we have 
    $$\mod A\geq (2K)^{-1}$$
    by the discussion in Section \ref{section:modulus}.  Let $z_0 \in V$ and $w_0 \in \partial Q'$ so that 
    $$|w_0 - z_0| = \dist(z_0, \partial  Q') = \dist(z_0, \Gamma).$$ 
    Note that $A$ separates $z_0, \overline{z_0}$ from $w_0, \infty$.
    By Teichm\"uller's module theorem \cite{LehtoVirtanen:quasiconformal}*{Section II.1.3}, there is an increasing function $\Psi\colon (0,\infty) \to (0,\infty)$ so that
    $$
    \frac{1}{2K}\leq \mod A \leq \Psi\left(\frac{|w_0-z_0|}{|z_0-\overline{z_0}|}\right) = \Psi\left(\frac{\dist(z_0, \Gamma)}{2\Im(z_0)}\right).
    $$
    Therefore, there exists $\delta = \delta(K)\geq 1$ so that $\dist(z_0, \Gamma) \geq \delta^{-1}\Im(z_0)$. Also, note that $\dist(z_0,\partial Q)=\dist(z_0,\Gamma)$ or $\dist(z_0,\partial Q)=\Im(z_0)$. Thus, by \eqref{prop:comparison:ineq},
    $$
    \rho_Q(z_0) \leq \frac{2}{\dist(z_0, \partial Q)} \leq \frac{2\delta }{\Im(z_0)} = 2\delta \rho_{\Hyp}(z_0).
    $$
    This proves the upper bound.
\end{proof}

Next, we state the main result of the section.  

\begin{proposition}\label{proposition:circularizable_gf}
    Let $f$ be a geometrically finite rational map and $U, V$ be Fatou components whose forward orbits consist of Jordan regions. There exists a constant $L=L(f)\geq 1$ such that if $\bar U\cap \bar V\neq \emptyset$ and the pair $U,V$ is quasiconformally circularizable, then it is $L$-quasiconformally circularizable. 
\end{proposition}

The proof uses the characterization of tangent quasicircles given in Theorem~\ref{theorem:tangent}, which says roughly that the quasiconformal circularizability of the pair $U,V$ depends on the hyperbolic geometry of $U^*=\widehat \C\setminus \bar U$, and on its relation with the Euclidean or chordal metric on $\partial V$. The main idea in the proof is to blow up conformally $V$ to a large Fatou component and to use the invariance under conformal maps of hyperbolic and Euclidean geometry, thanks to Koebe's distortion theorem. An obstacle that we have to overcome is that iterates of $f$ will be conformal in a neighborhood of $\bar V$, but not in $U^*$. Hence, the hyperbolic geometry of $U^*$ is not preserved. Instead, we replace $U^*$ with a smaller Jordan region $Q'$ that contains $V$ and in which $f$ is conformal. So, iterates of $f$ preserve the hyperbolic geometry of $Q'$. Finally, we rely on Proposition \ref{prop:comparisonhyperbolicmetric} to compare the hyperbolic metrics of $Q'$ and $U^*$.

    \begin{proof}
    Since critical points on the Julia set are pre-periodic, by Lemma \ref{lemma:no_infinite} \ref{lemma:no_infinite:2}, there are at most finitely many simply connected Fatou components that contain a critical point on their boundary. Choose $k_0\geq 1$ so that whenever a critical point lies in the closure of a simply connected Fatou component $\Omega$, the pre-period of $\Omega$ is at most $k_0$. Let $U,V$ be two disjoint Fatou components as in the statement such that $\bar U\cap \bar V=\{a_0\}$. Without loss of generality, we assume that $l_U \leq l_V$. If $l_V<k_0$, note that there are finitely many pairs $U,V$ with $l_U\leq l_V<k_0$. Thus, if $U,V$ are quasiconformally circularizable, then they are $L$-quasiconformally circularizable for some uniform $L\geq 1$. Now, we suppose that $l_V \geq  k_0$ and let $m=l_V-k_0$.  We set $V_m=f^{\circ m}(V)$, $U_m=f^{\circ m}(U)$, and $a_m=f^{\circ m}(a_0)$. Note that $l_{V_m}=k_0$ and $l_{U_m}\leq k_0$. As there are only finitely many Fatou components of pre-period $k_0$, by Theorem \ref{theorem:tangent} it suffices to prove that the map 
    \begin{align}\label{proposition:circularizable_gf:claim}
        \id: (\partial V\setminus \{a_0\}, d_{V, U}) \rightarrow (\partial V \setminus \{a_0\}, \chi)\,\,\, \text{is $\eta$-quasi-M\"obius}
    \end{align}
    for some $\eta$ depending only on $V_m$. In the following arguments, we use the term ``uniformly" to signify that a statement holds without dependence on $U,V,m$, but only on $V_m$ and on some constructions given below that depend on $V_m$.

    By assumption, $V_m$ is a Jordan region. Since $V_m$ is not periodic, $\overline{V_{m}}$ contains no parabolic periodic points by Lemma \ref{lemma:no_infinite} \ref{lemma:no_infinite:1}. Thus, $\bar{V_m}$ is disjoint from the accumulation points of the post-critical set. Recall that the post-critical set is the closure of the post-critical orbit ${\bigcup_{c\in \operatorname{Crit}(f)}\bigcup_{n\geq 1}f^{\circ n}(c)}$.  By the choice of $k_0$, $\bar {V_m}$ has a positive distance from the post-critical set. Consider a  Jordan region $N\supset \bar{V_{m}}$ so that $\bar N$ is disjoint from the post-critical set of $f$ and $N$ does not contain the entire boundary of any other Fatou component of pre-period at most $k_0$. Let $W$ be a Jordan region Fatou component with pre-period $l_W\leq k_0$, so $\partial W\not\subset N$. Due to the local connectivity of $\partial W$, at most finitely many components of $N\cap \partial W$ can intersect $\bar{V_m}$. Let $J\subset \partial W$ be a component of $N\cap \partial W$ that intersects $\bar {V_m}$ in exactly one point $z_0$ and is a quasiarc, as in Figure \ref{fig:quadrilateral}. Using Rickman's theorem (Theorem \ref{theorem:rickman}), we can construct a quasicircle in $N$ that contains non-overlapping subarcs $I_1,I_2$ of $J$ with $z_0\in I_1\cap I_2$, intersects $\bar {V_m}$ only at $z_0$, and separates $\partial N$ from $V_m$. Let $Q$ be the quasidisk bounded by this quasicircle and which contains $V_m$. Consider the quadrilateral $R=Q\setminus \bar {V_m}$ with vertical sides $I_1,I_2$; see Figure \ref{fig:quadrilateral}. With this process, to each of the finitely many pairs $V_m,W$, where $l_{V_m}=k_0$ and $l_W\leq k_0$, and to each of the finitely many arcs $J$ as above we assign a quasidisk $Q$ and a quadrilateral $R=Q\setminus \bar {V_m}$ with $\bar R\subset N$.

    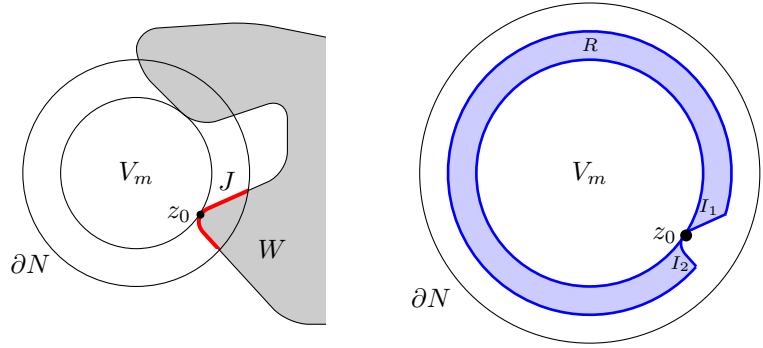
\begin{figure}
        \centering
        \begin{tikzpicture}
            \begin{scope}
            
                \begin{scope}
                \clip (0,-2) rectangle (2.5,2);
                \draw[rounded corners=8, fill=black!20] (4,-2)--(2,-2)--(0.7,-0.58)--(2,0)--(2,1)--(0.8,0.6)--(0,1.4)--(0,1.9)--(1.5,2)--(4,1.5)--cycle;
                \node at (1.8,-1) {$W$};
                \end{scope}

                \begin{scope}
                \clip (0,0) circle (1.5cm);
                \clip (-1,-1) rectangle (2,0.5);
                \draw[rounded corners=8, red, line width=1.5pt] (4,-2)--(2,-2)--(0.7,-0.58)--(2,0)--(2,1)--(0.8,1.2)--(0,1.4)--(0,2)--(4,1.5)--cycle;
                \end{scope}
    
                \fill (0.85,-0.55) circle (1.5pt) node[left] {$z_0$};
                \node at (1.2,-0.1) {$J$};
                \draw (0,0) circle (1cm) node {$V_m$};
                \draw (0,0) circle (1.5cm);
                \node at (-1.4,-1.2) {$\partial N$}; 
            \end{scope}

            \begin{scope}[shift={(6,0)}, scale=1.5]
            
                \draw[blue,line width=1pt, fill=blue!20]  (-17.5:1.25) arc[start angle=-17.5, end angle=318.5, radius=1.25];
                \node at (1.05,-0.3) {$\scriptstyle I_1$};
                \node at (0.8,-0.8) {$\scriptstyle I_2$};
                \draw[blue, line width=1pt, fill=white] (0,0) circle (1cm) node {\color{black}$V_m$};
                \draw (0,0) circle (1.5cm);
            
                \begin{scope}
                \clip (0,0) circle (1.25cm);
                \clip (-1,-1) rectangle (2,0.5);
                \draw[rounded corners=8, blue, line width=1pt, fill=white] (4,-2)--(2,-2)--(0.72,-0.58)--(2,0)--(2,1)--(0.8,1.2)--(0,1.4)--(0,2)--(4,1.5)--cycle;
                \end{scope}
                \fill (0.85,-0.55) circle (1.5pt) node[left] {$z_0$};
                \node at (-1.4,-1.1) {$\partial N$}; 
                \node at (0,1.13) {$\scriptstyle R$};
            \end{scope}
        \end{tikzpicture}
        \caption{The construction of a quadrilateral $R$ corresponding to the triple $V_m,W,J$. The set $\partial Q=\{z_0\}\cup (\partial R\setminus \partial V_m )$ is the ``outer" boundary curve of $R$.}\label{fig:quadrilateral}
    \end{figure}

    Since $\bar N$ is disjoint from the post-critical set, there exists a Jordan region $N'\supset \bar V$ such that $f^{\circ m}\colon  N' \rightarrow N$ is conformal, and in fact it extends to a conformal map from a neighborhood of $\bar{N'}$ onto a neighborhood of $\bar N$. Thus, a component of $\partial U_m\cap N$ is an arc $J$ such that $J\cap \bar {V_m}=\{a_m\}$. The pair $U,V$ is quasiconformally circularizable, so $\partial U$ is a quasicircle. Since $f^{\circ m}$ is bi-Lipschitz in $\bar {N'}$, the set $J$ is a quasiarc. Thus, the triple $V_m,U_m,J$ corresponds to a quasidisk $Q$ and a quadrilateral $R$ among the ones constructed above; see Figure~\ref{fig:quadrilateral}. Let $Q',R'$ be their preimages under $f^{\circ m}$. By conformality, we have $\mod R'=\mod R$.

    Since the pair $U,V$ is quasiconformally circularizable and $f^{\circ m}$ is conformal in a neighborhood of $a_0$, by Lemma \ref{lemma:tangent_local} we see that the pair $V_m,Q^*$, where $Q^*=\widehat \C\setminus \bar Q$, is quasiconformally circularizable with uniform constants, given that we have constructed a finite collection of quasidisks like $Q$.
    
    By Proposition~\ref{prop:comparisonhyperbolicmetric}, for $z \in V$, the hyperbolic metric $\rho_{U^*}(z)|dz|$ is comparable to $\rho_{Q'}(z)|dz|$, depending only on $\mod R$, and therefore with uniform constants. Hence, 
    $$
    \id\colon (\partial V\setminus \{a_0\}, d_{V, U}) \to (\partial V \setminus \{a_0\}, d_{V, (Q')^*})
    $$ is uniformly bi-Lipschitz; recall the definition of $d_{V,U}$ from Section \ref{section:uniformization}. On the other hand, since the map $f^{\circ m}\colon  Q' \rightarrow Q$ is conformal, the map 
    $$
    f^{\circ m}\colon (\partial V\setminus \{a_0\}, d_{V, (Q')^*}) \to (\partial V_{m} \setminus \{a_{m}\}, d_{V_m, Q^*})
    $$
    is an isometry. Since the pair $V_m,Q^*$ is quasiconformally circularizable, by Theorem~\ref{theorem:tangent} the map 
    $$\id\colon  (\partial V_{m}\setminus \{a_{m}\}, d_{V_{m}, Q^*}) \to (\partial V_{m}\setminus \{a_{m}\} , \chi)$$
    is uniformly quasi-M\"obius. By Lemma \ref{lemma:koebe_qm}, which is an implication of Koebe's distortion theorem, the conformality of the map $f^{\circ m}\colon N' \to N$ implies that 
    $$
    (f^{\circ m})^{-1}\colon (\partial V_{m} \setminus \{a_{m}\} , \chi ) \to (\partial V\setminus \{a_0\}, \chi)
    $$
    is  uniformly quasi-M\"obius. Altogether, the composition of the above maps is
    $$\id\colon (\partial V\setminus \{a_0\}, d_{V, U}) \rightarrow (\partial V\setminus \{a_0\}, \chi)$$ 
    and therefore it is uniformly quasi-M\"obius, which proves the desired \eqref{proposition:circularizable_gf:claim}.
    \end{proof}

\section{Quasidisk Fatou components}

\begin{proposition}\label{proposition:quasidisk_uniform}
    Let $f$ be a geometrically finite rational map and $U$ be a Fatou component whose forward orbit consists of Jordan regions. There exists a constant $L=L(f)\geq 1$ such that if $U$ is a quasidisk, then it is an $L$-quasidisk.
\end{proposition}

\begin{proof}
    By Lemma \ref{lemma:no_infinite} \ref{lemma:no_infinite:2}, there are finitely many simply connected Fatou components $\Omega$ whose closures contain critical points. We choose $k_0\geq 1$ so that whenever a critical point of $f$ lies in the closure of a simply connected Fatou component $\Omega$, the pre-period $l_\Omega$ of $\Omega$ is at most $k_0$. There are finitely many Fatou components $U$ with $l_U< k_0$. Thus, if $U$ is a quasidisk, then it is an $L$-quasidisk for some uniform $L\geq 1$. 
    
    Next, we suppose that $l_U\geq k_0$. Let $V=f^{\circ m}(U)$ for $m=l_U-k_0\geq 0$, so $l_V=k_0$. By assumption, $V$ is a Jordan region. We will show that $U$ is a uniform quasidisk, depending only on $V$, but not on $m$. This is sufficient, since there are finitely many Fatou components $V$ with $l_V=k_0$. Since $V$ is not periodic, $\bar V$ does not contain any parabolic periodic points by Lemma \ref{lemma:no_infinite} \ref{lemma:no_infinite:1}. Thus, $\bar{V}$ is disjoint from the accumulation points of the post-critical set. By the choice of $k_0$, $\bar {V}$ has a positive distance from the post-critical set. Consider a  Jordan region $N\supset \bar{V}$ that is disjoint from the post-critical set of $f$. 
    
    The map $g=f^{\circ m}\colon U\to V$ is conformal since $V$ is simply connected and $g$ has no critical points.  Moreover,  $g^{-1}$ extends conformally to the neighborhood $N$ of $\bar V$. Since $U$ is a quasidisk, $\partial U$ is a quasicircle by Theorem \ref{theorem:quasidisk_quasicircle}. The fact that $g$ is a diffeomorphism implies that $\partial V$ also satisfies the Ahlfors condition and is a quasicircle. There are only finitely many Fatou components of pre-period $k_0$, so $\partial V$ is a uniform quasicircle. By Lemma \ref{lemma:koebe_qm}, the map $g^{-1}\colon (\partial  V, \chi)\to (\partial U, \chi)$ is uniformly quasi-M\"obius. Thus, $\partial U$ is a uniform quasicircle, being the quasi-M\"obius image of a quasicircle. Hence, $U$ is a uniform quasidisk.  
\end{proof}

Finally, we include a characterization of Fatou components of geometrically finite rational maps that are quasidisks. The proof is similar to the proof of \cite{LuoNtalampekos:piecewise_qc}*{Theorem 1.7} and is presented in Appendix \ref{appendix}.

\begin{proposition}\label{prop:quasidiskgf}
    Let $f$ be a geometrically finite rational map, and let $U$ be a Fatou component whose orbit $\{f^{\circ n}(U)\}_{n\geq 0}$ consists of Jordan regions.
    Then $U$ is a quasidisk if and only if
    \begin{enumerate}[label=\normalfont(\arabic*)]
        \item\label{prop5:quasidisk:pereq:1} for each $V \in \{f^{\circ n}(U)\}_{n\geq 0}$, $\partial V$ contains at most one parabolic periodic point $a_V$, in which case $V$ is an immediate basin of $a_V$ and
        \item\label{prop5:quasidisk:pereq:2}  if $a_V$ has parabolic multiplicity $2$, then no point $z\in \partial U$ is mapped to $a_V$ by an iterate of $f$ that is a local homeomorphism near $z$. 
    \end{enumerate}
    \end{proposition}

\begin{remark}\label{remark:prop:quasidisk}
    Note that if $U$ is a quasidisk that is an immediate basin of a parabolic periodic point $a_0$, then condition \ref{prop:quasidisk:pereq:2} implies that the multiplicity of $a_0$ is at least $3$. 
    See Figure~\ref{fig:JuliaSet} for an illustration of the non-local homeomorphism condition.
\end{remark}    

\begin{figure}
    \centering
    \begin{tikzpicture}
    \node[opacity=.4] at (0,0){\includegraphics[width=0.8\linewidth]{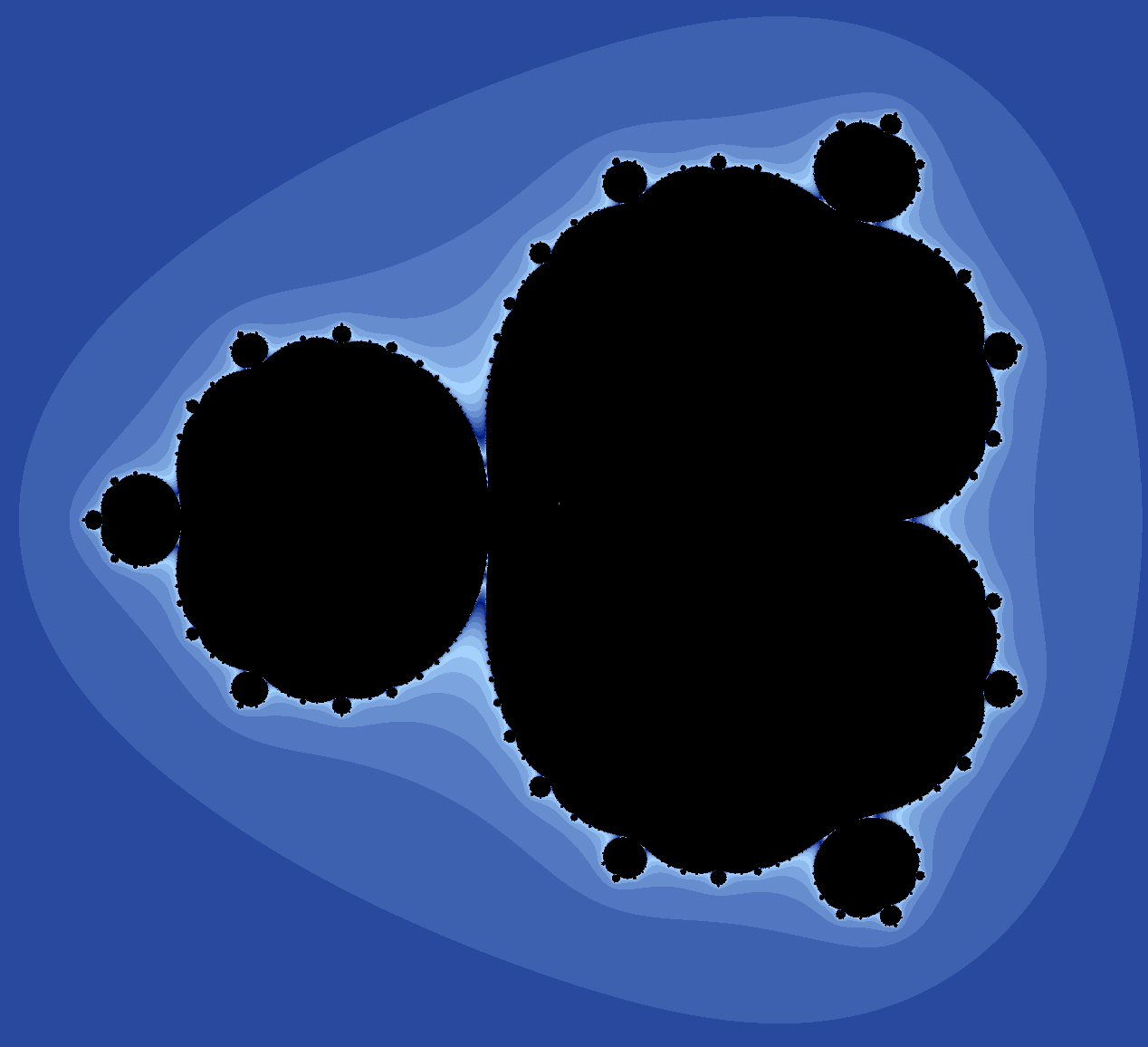}};
    \fill[white] (2.9,0.02) circle (1.5pt) node[left] {\color{white}$a_0$};
    \end{tikzpicture}
    \caption{The Julia set of $z(z+1)^2$. Each bounded Fatou component, except for the one that contains the cusp at $a_0=0$, is a quasidisk. The unique cusp point $a_0$ is a parabolic fixed point of multiplicity $2$. Each point $a\neq a_0$ that is eventually mapped to $a_0$ is mapped there by an iterate of $f$ that is locally $2$-to-$1$.}
    \label{fig:JuliaSet}
\end{figure}

\section{Separated Fatou components}

\begin{proposition}\label{proposition:separated_modulus}
    Let $f$ be a geometrically finite rational map and $U, V$ be Fatou components whose orbits $\{f^{\circ n}(U)\}_{n\geq 0}$, $\{f^{\circ n}(V)\}_{n\geq 0}$ consist of Jordan regions. Then there exists a constant $L = L(f)\geq 1 $ such that if $\overline{U} \cap \overline{V} = \emptyset$, then
    $$
        \mod A(\bar U, \bar V) \geq L^{-1}.
    $$
\end{proposition}

Recall that $A(\bar U,\bar V)$ is the annulus whose complementary components are $\bar U$ and $\bar V$.

\begin{proof}
    By Lemma \ref{lemma:no_infinite} \ref{lemma:no_infinite:2} there exists a number $k_0\geq 1$ so that whenever a critical point lies in the closure of a simply connected Fatou component $\Omega$, the pre-period $l_\Omega$ of $\Omega$ is at most $k_0$. Without loss of generality, we assume that $l_U \leq l_V$. Since there are only finitely many Fatou components of pre-period at most $k_0$, we may assume that $l_V \geq  k_0$. For $m=l_V-k_0$ we set $V_m=f^{\circ m}(V)$ and $U_m=f^{\circ m}(U)$. Then $l_{V_m}=k_0$ and $l_{U_m}\leq k_0$. 
    
    As in the proof of Proposition \ref{proposition:circularizable_gf}, there exists a Jordan region $N\supset \bar {V_m}$ so that $\bar N$ is disjoint from the post-critical set of $f$ and $N$ does not contain the boundary of any Fatou component of pre-period at most $k_0$. Let $W$ be a Jordan region Fatou component with pre-period $l_W\leq k_0$, so $\partial  W\not\subset N$. Due to the local connectivity of $\partial W$, if $N\cap \partial W$ has infinitely many components, their diameters must converge to $0$. Thus, if $N_1$ is a Jordan region with $\bar {V_m}\subset N_1\subset \bar{N_1}\subset N$, then at most finitely many components of $N\cap \partial W$ can intersect $N_1$. By shrinking $N_1$, we can ensure that the only components of $N\cap \partial W$ that intersect $N_1$, if any, are the ones that also intersect $\bar {V_m}$. Let $A=N_1\setminus \bar {V_m}$. With this process, to each of the finitely many pairs $V_m,W$, where $l_{V_m}=k_0$ and $l_W\leq k_0$, we associate an annulus $A$, so that if a component of $N\cap \partial W$ intersects $A$, then it also intersects $\bar {V_m}$. See Figure \ref{fig:annulus} for an illustration.

    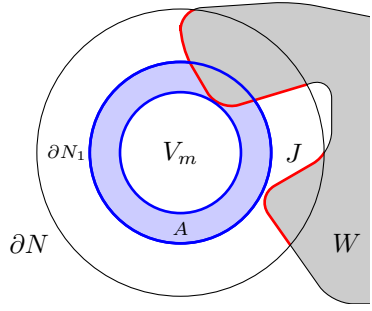
\begin{figure}
        \centering
        \begin{tikzpicture}
                \draw[blue, line width=1pt, fill=blue!20] (0,0) circle (1.2cm); 
                \begin{scope}
                \clip (0,-2) rectangle (2.5,2);
                \draw[rounded corners=8, fill=black!20] (4,-2)--(2,-2)--(1,-0.58)--(2,0)--(2,1)--(0.5,0.55)--(0,1.4)--(0,1.9)--(1.5,2)--(4,1.5)--cycle;
                \node at (2.2,-1.2) {$W$};
                \end{scope}

                \begin{scope}
                \clip (0,0) circle (1.9cm);
                \draw[rounded corners=8, red, line width=1pt] (4,-2)--(2,-2)--(1,-0.58)--(2,0)--(2,1)--(0.5,0.55)--(0,1.4)--(0,1.9)--(1.5,2)--(4,1.5)--cycle;
                \end{scope}

                \draw[blue, line width=1pt] (0,0) circle (1.2cm); 
                \draw[blue, line width=1pt, fill=white] (0,0) circle (0.8cm) node {$\color{black}V_m$};
                \node at (1.5,-0) {$J$};
                
                \draw (0,0) circle (1.9cm);
                \node at (-2,-1.2) {$\partial N$};
                \node at (-1.5,0) {$\scriptstyle \partial N_1$};
                \node at (0,-1) {$\scriptstyle A$};
        \end{tikzpicture}
        \caption{The annulus $A$ associated to the pair $V_m,W$ has the property that each component of $N\cap \partial W$ is either disjoint from $A$ or it intersects $\bar {V_m}$. Shown are two components of $N\cap \partial W$.  This figure illustrates the possibility that $\overline{V}\cap \overline{U} = \emptyset$, but $\overline{f^{\circ m}(V)}\cap\overline{f^{\circ m}(U)} \neq \emptyset$.}\label{fig:annulus}
    \end{figure}

    There exists a Jordan region $N'\supset \bar V$ such that $g=f^{\circ m}\colon N'\to N$ is conformal. Let $N_1'=g^{-1}(N_1)$ and $A'= g^{-1}(A)$. Since $g\colon A' \to A$ is conformal, we have $\mod A' = \mod A$. 
    
    We claim that the annulus $A'$ is disjoint from $U$. Suppose instead that $A'\cap U\neq\emptyset$. Since $\bar U\cap \bar V=\emptyset$, we have $A'\setminus U\neq \emptyset$. By the connectedness of $A'$, we obtain $\partial U\cap A'\neq \emptyset$. Thus, there exists a component $J'$ of $\partial U\cap N'$ such that $J'\cap A'\neq \emptyset$. Note that $J'\cap \bar{V}=\emptyset$. This implies that there exists a component $J=g(J')$ of $N\cap \partial U_m$ such that $J\cap A\neq \emptyset$ and $J\cap \bar{V_m}=\emptyset$. This contradicts the properties of the annulus $A$. Therefore, $A'\cap U=\emptyset$. 
    
    It follows that $A'$ separates $\bar U$ and $\bar V$, so $\mod A(\bar U, \bar V) \geq \mod A' =\mod A$. Since $A$ is an annulus among a finite collection of annuli constructed above (corresponding to the various pairs $V_m,W$), we see that $\mod A(\bar U,\bar V)$ is uniformly bounded from below, as desired. 
\end{proof}

\begin{proposition}[\cite{Herron:uniform}*{Corollary 2.4}]\label{proposition:separated_qc}
    Let $L\geq 1$. Let $U,V\subset \widehat \C$ be $L$-quasidisks such that $\bar U\cap \bar V=\emptyset$ and $\mod A (\bar U,\bar V)\geq L^{-1}$. Then there exists a $K(L)$-quasi\-conformal map of $\widehat \C$ that maps $U$ and $V$ onto disks.
\end{proposition}

\begin{corollary}\label{corollary:separated_quasidisk}
    Let $f$ be a geometrically finite rational map and $U, V$ be Fatou components whose forward orbits consist of Jordan regions. Then there exists a constant $K=K(f)\geq 1$ such that if $U,V$ are quasidisks with $\bar U\cap \bar V=\emptyset $, then the pair $U,V$ is $K$-quasiconformally circularizable. 
\end{corollary}

\begin{proof}
    By Proposition \ref{proposition:separated_modulus}, we have $\mod A(\bar U,\bar V)\geq L_1^{-1}$ for some $L_1=L_1(f)\geq 1$. Also, $U$ and $V$ are $L_2$-quasidisks for some $L_2=L_2(f)\geq 1$ by Proposition \ref{proposition:quasidisk_uniform}. Finally, by Proposition \ref{proposition:separated_qc} there exists a $K(L_1,L_2)$-quasiconformal map of $\widehat \C$ that maps $U$ and $V$ to disks. 
\end{proof}

\section{Proofs of main results}\label{section:proofs}

The next result allows us to turn some Fatou components of a geometrically finite rational map into disks with a quasiconformal map of the sphere. 
\begin{theorem}\label{theorem:uniformization_general}
    Let $f$ be a geometrically finite rational map whose Julia set is not a Jordan curve.
    Let $\{U_i\}_{i\in I}$ be a collection of at least two Fatou components of $f$ whose forward orbits consist of Jordan regions. The following are equivalent.
    \begin{enumerate}[label=\normalfont(\arabic*)]
        \item\label{theorem:uniformization_general:1} (Quasiconformal Uniformization) There exists a quasiconformal map $g$ of $\widehat \C$ that maps $U_i$ to a disk for each $i\in I$ and is $1$-quasiconformal in $\widehat \C\setminus \bigcup_{i\in I}U_i$.
        \item\label{theorem:uniformization_general:2} (Dynamical criterion) 
        For each $i\in I$ the Fatou component $U_i$ is a quasidisk and if $\bar {U_i}\cap \bar{U_j}\neq \emptyset$ for some $i\neq j$ then $\bar {U_i}\cap \bar{U_j}$ consists of a single point $a_0$ that satisfies one of the following alternatives.
        \begin{itemize}
            \item $a_0$ is mapped to a parabolic periodic cycle of multiplicity $3$ by an iterate of $f$ that is a local homeomorphism near $a_0$.
            \item $a_0$ is mapped to a parabolic periodic cycle of multiplicity $2$ by an iterate of $f$ that is locally two-to-one near $a_0$.
        \end{itemize}
    \end{enumerate}
     Moreover, in this case the map $g|_{\widehat \C\setminus \bigcup_{i\in I}U_i}$ is unique up to post-composition with conformal automorphisms of $\widehat \C$.
\end{theorem}

We establish some general preliminary statements. 

\begin{lemma}\label{lemma:lift:arcs_invariant}
    Let $\phi(z)=z^d$, $d\geq 2$. Let $f,h$ be orientation-preserving homeomorphisms in neighborhoods of $\bar \D$ that fix $0$ such that $h$ is a lift of $f$ under $\phi$.
    \begin{enumerate}[label=\normalfont (\arabic*)]
        \item\label{lemma:lift:arc} Let $E\subset \D$ be a Jordan arc with an endpoint at $0$ and an endpoint at $\partial \D$ that is locally invariant near $0$ under $f$.
        \item Suppose that some lift $\widetilde E$ of $E$ is locally invariant near $0$ under $h$.
    \end{enumerate}
    Then every lift of $E$ is locally invariant near $0$ under $h$. The same conclusion holds for the lifts of any arc $F$ satisfying \ref{lemma:lift:arc} such that $E\cap F=\{0\}$.
\end{lemma}
\begin{proof}
    The arc $E$ has $d$ preimages $E_1,\dots,E_d$ under $\phi$. Each $E_i$ is a Jordan arc with endpoints at $0$ and at some point $z_i \in  \partial \D$. The arcs are disjoint except at $0$, which is their common endpoint. We assume that the points $z_i\in \partial \D$, $i\in \{1,\dots,d\}$, have the counter-clockwise order. Note that $h$ acts transitively on these arcs in the sense that there exists a neighborhood $U$ of $0$ such that for each $i\in \{1,\dots,d\}$ there exists a unique $j\in \{1,\dots,d\}$ such that $h(E_j\cap U)\subset E_i$. Hence, $h$ induces a bijection $\tilde h\colon \Z_d\to \Z_d$ with the property that  $\tilde h(j)=i$ if $h(E_j\cap U)\subset E_i$. Since $h$ is an orientation-preserving homeomorphism, the map $\tilde h$ must be a translation. By assumption, a lift of $E$ is locally invariant under $h$ near $0$. This implies that $\tilde h$ has a fixed point, so it is the identity map. Thus, every lift of $E$ is locally invariant under $h$.  

    Next, consider an arc $F$ satisfying \ref{lemma:lift:arc} such that $E\cap F=\{0\}$. Then $F$ has $d$ lifts each of which is between two lifts of $E$ and $h$ acts transitively on the lifts of $F$. Since each lift of $E$ is locally invariant near $0$ under $h$, and $h$ is an orientation-preserving homeomorphism, we see that each lift of $F$ is also locally invariant near $0$ under $h$. 
\end{proof}

\begin{lemma}\label{lem:isolatedpre-periodic}
    Let $f$ be a geometrically finite rational map. If $U_1,U_2$ are simply connected Fatou components of $f$, then every isolated point of $\bar {U_1}\cap \bar {U_2}$ is pre-periodic.  
\end{lemma}

\begin{proof}
    By conjugating $f$ with a M\"obius transformation, we assume that $\infty$ lies in the Fatou set. Let $a_0\in \bar {U_1}\cap \bar {U_2}$ be an isolated point. If the orbit of $a_0$ contains critical points or parabolic points, then it is pre-periodic since $f$ is geometrically finite. 
    Suppose that the orbit of $a_0$ does not contain critical points or parabolic points. By Sullivan's theorem, the components $U_1,U_2$ are eventually periodic. By passing to an iterate, we assume that $f(U_1)$ and $f(U_2)$ are fixed Fatou components, not necessarily distinct. Then by \cite{McMullen:Hausdorff_GF}*{Theorem 6.5} the point $a_0$ lies in the \textit{radial Julia set} of $f$. By definition, this means that there exists $r>0$ such that for every $\varepsilon>0$ there exists a neighborhood $\Omega$ of $a_0$ and $n\in \N$ such that $\diam(\Omega)<\varepsilon$ and the map 
    $$f^{\circ n}\colon \Omega\to B(f^{\circ n}(a_0),r)$$
    is conformal. By assumption, if $\Omega$ is sufficiently small, then $a_0$ is the unique point of $\bar {U_1}\cap \bar {U_2}$ in $\Omega$. Also, by shrinking $r$ we may assume that $\diam f(U_i)>2r$, so $f(U_i)$ is not contained in $B(f^{\circ n}(a_0),r)$ for $i=1,2$.

    By \cite{LeiYongcheng:local_connectivity}*{Lemma 2.1}, $U_1,U_2$ and $f(U_1),f(U_2)$ have locally connected boundary. Let $I$ be the infinite set of indices $n\in \N$ such that there exists a neighborhood $\Omega_n$ of $a_0$ with $\Omega_n\cap \bar {U_1}\cap \bar {U_2}=\{a_0\}$ and the map $f_n=f^{\circ n}\colon \Omega_n\to B(a_n,r)$ is conformal, where $a_n=f^{\circ n}(a_0)$. Consider conformal maps $\phi_i\colon \D\to U_i$, $i=1,2$, which extend continuously to the boundaries by Carath\'eodory's theorem. For $n\in I$ and $i=1,2$, let $P_{i,n}\subset \partial \D$ be the union of the components of $\phi_i^{-1}(\Omega_n\cap \partial U_i)$ that intersect $\phi_i^{-1}(a_0)$. Also, let $\psi_i\colon \D\to f(U_i)$, $i=1,2$, be conformal maps. Let $i\in \{1,2\}$. Consider the conformal map $g_i=\psi_i^{-1}\circ f_n\circ \phi_i$ defined in a neighborhood of $P_{i,n}$ in $\D$. Note that $g_i$ extends continuously to $P_{i,n}$ and let $Q_{i,n}=g(P_{i,n})$. Each component of $Q_{i,n}$ intersects $\psi_i^{-1}(a_n)$ and is mapped by $\psi_i$ to a curve that passes through $a_n$ and has endpoints on $\partial B(a_n,r)$; see Figure \ref{fig:isolated}. By the uniform continuity of $\psi_i$ in $\bar \D$, there exists a uniform constant $\delta>0$ such that for $J_{i,n}=\psi_i^{-1}(a_n)\cap Q_{i,n}$ we have
    \begin{align}\label{lem:isolated:delta}
        \dist(J_{i,n}, \partial \D\setminus Q_{i,n})\geq \delta.
    \end{align}

    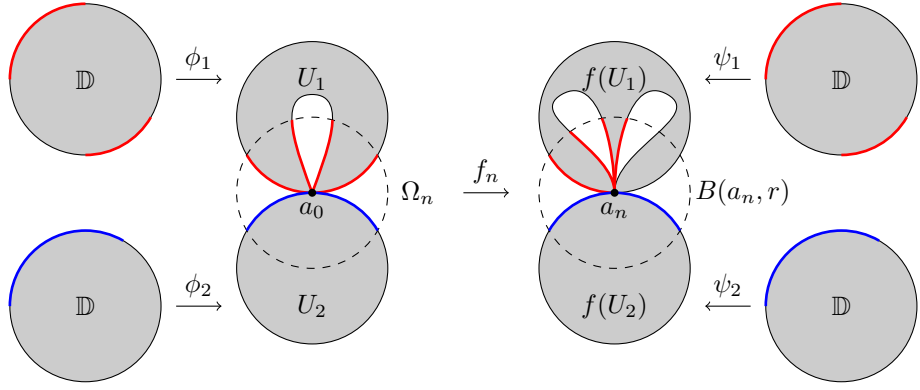
\begin{figure}
        \centering
        \begin{tikzpicture}
            \draw[fill=black!20] (0,1.5) circle (1cm) node {$\D$};
            \draw[fill=black!20] (0,-1.5) circle (1cm) node {$\D$};
            \draw[line width=1pt,red] (0,0.5) arc[start angle=-90, end angle=-30, radius=1];
            \draw[line width=1pt,red] (-1,1.5) arc[start angle=180, end angle=90, radius=1];
            \draw[line width=1pt,blue] (-1,-1.5) arc[start angle=180, end angle=60, radius=1];

            \draw[->] (1.2,1.5)--(1.8,1.5) node[above, pos=0.5] {$\phi_1$};
            \draw[->] (1.2,-1.5)--(1.8,-1.5) node[above, pos=0.5] {$\phi_2$};

            \draw[fill=black!20] (3,1) circle (1cm);\node at (3,1.5) {$U_1$};
            \draw[fill=black!20] (3,-1) circle (1cm);\node at (3,-1.5){$U_2$};
            \draw[fill=white] (3,0)to[out=110,in=180] (3,1.3) to[out=0,in=70](3,0);
            \draw[dashed] (3,0) circle (1cm); \node at (4.4,0) {$\Omega_n$};
            \begin{scope}
                \clip (3,0) circle (1cm);
                \draw[line width=1pt,red] (3,1) circle (1cm);
                \draw[line width=1pt,red] (3,0)to[out=110,in=180] (3,1.3) to[out=0,in=70](3,0);
            \end{scope}
            \draw[line width=1pt,blue] (3,0) arc[start angle=90, end angle=150, radius=1];
            \draw[line width=1pt,blue] (3,0) arc[start angle=90, end angle=30, radius=1];
            \fill (3,0) circle (1.5pt) node[below] {$a_0$};

            \draw[->] (5,0)--(5.6,0) node[above,pos=.5] {$f_n$};

            \begin{scope}[xshift=4cm]
            \draw[fill=black!20] (3,1) circle (1cm);\node at (3,1.5) {$f(U_1)$};
            \draw[fill=black!20] (3,-1) circle (1cm);\node at (3,-1.5){$f(U_2)$};
            \draw[fill=white] (3,0)to[out=90,in=30] (2.3,1.3) to[out=-150,in=90](3,0);
            \draw[fill=white] (3,0)to[out=90,in=150] (3.7,1.3) to[out=-30,in=90](3,0);
            \draw[dashed] (3,0) circle (1cm); \node at (4.7,0) {$B(a_n,r)$};
        
            \begin{scope}
                \clip (3,0) circle (1cm);
                \draw[line width=1pt,red] (3,0) arc[start angle=-90, end angle=-150, radius=1];
                \draw[line width=1pt,red] (3,0)to[out=90,in=30] (2.3,1.3) to[out=-150,in=90](3,0);
                \draw[line width=1pt,red] (3,0)to[out=90,in=150] (3.7,1.3);
            \end{scope}
            
            \draw[line width=1pt,blue] (3,0) arc[start angle=90, end angle=150, radius=1];
            \draw[line width=1pt,blue] (3,0) arc[start angle=90, end angle=30, radius=1];
            \fill (3,0) circle (1.5pt) node[below] {$a_n$};
            \end{scope}
            
            \draw[->] (8.8,1.5)--(8.2,1.5) node[above, pos=0.5] {$\psi_1$};
            \draw[->] (8.8,-1.5)--(8.2,-1.5) node[above, pos=0.5] {$\psi_2$};

            \begin{scope}[xshift=10cm]
            \draw[fill=black!20] (0,1.5) circle (1cm) node {$\D$};
            \draw[fill=black!20] (0,-1.5) circle (1cm) node {$\D$};
            \draw[line width=1pt,red] (0,0.5) arc[start angle=-90, end angle=-30, radius=1];
            \draw[line width=1pt,red] (-1,1.5) arc[start angle=180, end angle=90, radius=1];
            \draw[line width=1pt,blue] (-1,-1.5) arc[start angle=180, end angle=60, radius=1];
            \end{scope}
        \end{tikzpicture}
        \caption{The sets $P_{i,n}$ and their images $Q_{i,n}=g(P_{i,n})$.}
        \label{fig:isolated}
    \end{figure}
    
    If a point $b\in \partial f(U_i)$ satisfies $\psi_i^{-1}(b) \cap  Q_{i,n}\neq \emptyset$ for both $i=1,2$, then $\phi_i^{-1}(f_n^{-1}(b))\cap P_{i,n}\neq \emptyset$, so $f_n^{-1}(b)\cap \bar {U_i}\neq \emptyset$ for both $i=1,2$. Thus, $a_0=f_n^{-1}(b)$ and $a_n=b$. We conclude that $a_n$ is the unique point of $\partial f(U_i)$ that satisfies $\psi_i^{-1}(a_n) \cap  Q_{i,n}\neq \emptyset$ for both $i=1,2$. In particular, if $a_m\neq a_n$, then for some $i\in \{1,2\}$ we have $\psi_i^{-1}(a_m)\subset \partial \D\setminus Q_{i,n}$, so by \eqref{lem:isolated:delta} we have     
    $$\dist(J_{i,n},J_{i,m}) \geq \delta \quad \text{for $i=1$ or $i=2$.}$$
    Consider the space $X$ of compact subsets of $\partial \D\times \partial \D$ equipped with the Hausdorff metric. Then for $n,m\in I$ with $a_n\neq a_m$, any two points 
    $$J_{1,n}\times J_{2,n},\,\, J_{1,m}\times J_{2,m}\in X$$
    have distance at least $\delta$. Since $X$ is compact, we conclude that the set $\{a_n:n\in I\}$ is finite. Thus, there exist $n\neq m$ such that $f^{\circ n}(a_0)=f^{\circ m}(a_0)$ and $a_0$ is pre-periodic, as desired.   
\end{proof}

\begin{proof}[Proof of Theorem~\ref{theorem:uniformization_general}]
We first prove the sufficiency. Suppose that $U_i$ is a quasidisk for each $i\in I$, as in \ref{theorem:uniformization_general:2}. By Theorem \ref{theorem:schottky}, it suffices to prove that for every pair of distinct $i,j$ there exists a uniformly quasiconformal homeomorphism of $\widehat \C$ that maps $U_i$ and $U_j$ to disks. If $\bar{U_i}\cap \bar{U_j}=\emptyset$, then the conclusion follows from Corollary \ref{corollary:separated_quasidisk}. Suppose that $\bar{U_i}\cap \bar{U_j}=\{a_0\}$. By the assumptions, combined with Proposition \ref{proposition:component_dependent} and Lemma \ref{lemma:tangent_local}, we see that the pair $U_i,U_j$ is quasiconformally circularizable. By Proposition \ref{proposition:circularizable_gf},  the pair $U_i, U_j$ is uniformly quasiconformally circularizable, as desired. The last uniqueness statement follows from the uniqueness part of Theorem \ref{theorem:schottky}.

We now prove the necessity. Suppose that there exists a quasiconformal map $g$ of $\widehat \C$ such that $g(U_i)$ is a disk for each $i\in I$, as in \ref{theorem:uniformization_general:1}. It is obviously necessary that $U_i$ is a quasidisk for each $i\in I$. Also, if $\bar{U_i}\cap \bar{U_j}\neq \emptyset$ for some $i\neq j$, then $\bar{U_i}\cap \bar{U_j}$ consists of a single point $a_0$, given that the Julia set is not a Jordan curve.

By Lemma~\ref{lem:isolatedpre-periodic}, $a_0$ is pre-periodic. Let $b_0$ be a periodic point in the orbit of $a_0$. Observe that the point $b_0$ must be repelling or parabolic; indeed, there are no Cremer points for geometrically finite rational maps \cite{Milnor:dynamics}*{Theorem 11.17}.
Let $l$ be the pre-period of $a_0$ and $p$ be a multiple of the period of $b_0=f^{\circ l}(a_0)$ so that $f^{\circ p}$ fixes all simply connected Fatou components that contain $b_0$ in their boundary; there are finitely many such Fatou components by Lemma \ref{lemma:no_infinite} \ref{lemma:no_infinite:1}. Define $f_1=f^{\circ l}$, $f_2=f^{\circ p}$. By Lemma \ref{lemma:lift} there exists a parabolic germ $h$ at $a_0$ such that $f_2\circ f_1=f_1\circ h$.

Suppose that $b_0$ is repelling for $f_2$, so by Lemma \ref{lemma:lift}, $a_0$ is repelling for $h$. Given that $f_1(U_i)$ and $f_1(U_j)$ are Jordan regions that are fixed under $f_2$, we see that $f_2$ locally preserves their boundaries near $b_0$. Suppose first that $f_1=f^{\circ l}$ is locally injective near $a_0$. Then we see directly that $h$ locally preserves $\partial U_i$ and $\partial U_j$ near $a_0$. Now, suppose that $f_1$ has a critical point at $a_0$. If $f_1^{-1}$ is a branch of the inverse of $f_1$ that maps a neighborhood of $b_0$ in $f_1(U_i)$ onto a neighborhood of $a_0$ in $U_i$, then we may choose $h$ so that $h=f_1^{-1}\circ f_2\circ f_1$; see Lemma \ref{lemma:lift} \ref{lemma:lift:choices}. Thus, we may have that $h$ locally preserves $\partial U_i$ near $a_0$. By Lemma \ref{lemma:lift:arcs_invariant}, $\partial U_j$ is also locally invariant under $h$ near $a_0$. In both cases, by applying Lemma \ref{lemma:repelling_invariant} we see that no quasiconformal map of the sphere can map $U_i$ and $U_j$ to two disks, a contradiction.

Therefore, $b_0$ is a parabolic point for $f_2$. Let $m\geq1$ be the local degree of $f_1$ at $a_0$ and $n+1\geq 2$ be the parabolic multiplicity of the point $b_0$ for the map $f_2$. Then the parabolic multiplicity of $a_0$ for the map $h$ is $nm+1$ by Lemma \ref{lemma:lift} \ref{lemma:lift:1}. We have $nm+1=2$ if and only if $f_1$ is locally injective near $a_0$ and the multiplicity of $b_0$ for the map $f_2$ is $2$. In this case, any Fatou component that contains $b_0$ in its boundary, including the immediate basin, has a cusp; see Lemma \ref{lemma:avoid_petals}. Therefore, the same is true for $a_0$ by the local injectivity of $f_1$. Thus, $U_i$ and $U_j$ are not quasidisks, a contradiction. We conclude that the parabolic multiplicity of $a_0$ for $h$ is strictly greater than $2$. 

By Lemma \ref{lemma:avoid_petals} and since $U_i,U_j$ are quasidisks, we see that $f_1(U_i)$ and $f_1(U_j)$ are immediate basins of $b_0$, not necessarily distinct. Note that in a neighborhood of $a_0$ in each of the sets $U_i$ and $U_j$ the map $f_1$ conjugates $h$ to $f_2$. Thus, $U_i,U_j$ contain attracting petals $P_i,P_j$, respectively, of $h$ at $a_0$. In addition, points of $U_i,U_j$ near $a_0$ are absorbed under iteration of $h$ by $P_i,P_j$, respectively. If $nm+1\geq 4$, then $h$ has at least $3$ attracting directions, so there exists a third attracting direction and a corresponding petal $P$  that is disjoint from $U_i$ and $U_j$ and can be taken to be a quasidisk. By assumption, there exists a quasiconformal map of the sphere that sends $U_i$ and $U_j$ to tangent disks. This contradicts Lemma \ref{lemma:room3}.  Therefore, $nm+1 = 3$. This implies that $m = 2$ and $n=1$ or $m=1$ and $n=2$, as required in the two alternatives of Theorem \ref{theorem:uniformization_general} \ref{theorem:uniformization_general:2}.
\end{proof}

\begin{proof}[Proof of Theorem \ref{theorem:rational_uniformization}]
    Suppose that there exists a quasiconformal map $g$ of $\widehat \C$ that maps each Fatou component of $f$ onto a disk. Condition \ref{r:jordan} is immediate. Each Fatou component is a quasidisk. By Proposition \ref{prop:quasidiskgf} and Remark \ref{remark:prop:quasidisk}, if $f$ has a parabolic periodic point $a_0$, then its multiplicity must be at least $3$. On the other hand, if the multiplicity is strictly greater than $3$, then at least three Fatou components intersect at $a_0$, which is a contradiction. Thus \ref{r:parabolic3} holds. If $a_0$ lies in the boundary of a Fatou component $U$, then $U$ is periodic by Lemma \ref{lemma:no_infinite} \ref{lemma:no_infinite:1} and $U$ is an immediate basin of $a_0$ by Proposition \ref{prop:quasidiskgf}. This establishes condition \ref{r:parabolic_basins} of Theorem \ref{theorem:rational_uniformization}. Finally, suppose that $U,V$ are Fatou components with $\bar U\cap \bar V\neq \emptyset$. Since the Julia set of $f$ is not a Jordan curve, we have $\partial U\neq \partial V$. Since $g(U)$ and $g(V)$ are tangent disks, $\bar U\cap \bar V$ must consist of a single point $a_0$. By Theorem \ref{theorem:uniformization_general}, the point $a_0$ is mapped to a parabolic periodic cycle of multiplicity $3$ by an iterate of $f$ that is a local homeomorphism near $a_0$, as required in condition \ref{r:parabolic_tangent} of Theorem \ref{theorem:rational_uniformization}.

    Conversely, suppose that conditions \ref{r:jordan}, \ref{r:parabolic3}, \ref{r:parabolic_basins}, \ref{r:parabolic_tangent} hold and that the Julia set of $f$ is not homeomorphic to $\mathbb S^1$ or $\mathbb S^2$. If there exists only one Fatou component, then this component is a Jordan region by \ref{r:jordan}, so its complement, i.e., the Julia set, has non-empty interior. However, this forces the Julia set to be the entire sphere \cite{Milnor:dynamics}*{Corollary 4.11}, a contradiction. Thus, there exist at least two Fatou components. 
    
    Our goal is to verify the assumptions of Theorem \ref{theorem:uniformization_general} \ref{theorem:uniformization_general:2}. By assumption \ref{r:parabolic3}, there are no parabolic points of multiplicity $2$. Let $U$ be a Fatou component of $f$. By \ref{r:parabolic_basins}, the orbit of $\partial U$ contains the orbit of at most one parabolic periodic point $a$, in which case an immediate basin of $a$ is in the orbit of $U$. Thus, Proposition \ref{prop:quasidiskgf} applies and $U$ is a quasidisk. By condition \ref{r:parabolic_tangent}, if $U,V$ are distinct Fatou components with $\bar U\cap \bar V\neq \emptyset$, then the intersection $\bar U\cap \bar V$ consists of a single point $a_0$ that is not critical and it is eventually mapped to a parabolic periodic cycle of multiplicity $3$. Note that $f(U)$ and $f(V)$ are Fatou components that are Jordan regions by \ref{r:jordan}, and they are distinct by Lemma \ref{lemma:no_infinite} \ref{lemma:no_infinite:00}, given that $a_0$ is not a critical point. Applying condition \ref{r:parabolic_tangent} to the pair $f(U),f(V)$, we see that $f(a_0)$ is the unique point in the intersection $\bar{f(U)}\cap \bar{f(V)}$ and $f(a_0)$ is not a critical point. Hence, inductively, the entire forward orbit of $a_0$ is not critical. Hence, $a_0$ is mapped to a parabolic periodic cycle of multiplicity $3$ by an iterate of $f$ that is a local homeomorphism near $a_0$, as in the assumptions of Theorem \ref{theorem:uniformization_general}. 
    
    We now apply Theorem \ref{theorem:uniformization_general}, which implies the existence of a quasiconformal homeomorphism $g\colon \widehat \C\to \widehat \C$ that maps each Fatou component of $f$ onto a disk and is $1$-quasiconformal on the Julia set $J(f)$. Since $J(f)$ has measure zero \cite{McMullen:Hausdorff_GF}*{Corollary 6.2}, the $1$-quasiconformality condition holds vacuously for any such quasiconformal map $g$. By the last part of Theorem \ref{theorem:uniformization_general} we conclude that $g|_{J(f)}$ is unique up to post-composition with M\"obius transformations. 
\end{proof}

The proof of Theorem \ref{theorem:polynomial_uniformization} is obtained by exactly the same argument as above, applied to bounded Fatou components only.

\begin{proof}[Proof of Proposition \ref{prop:julia_jordan}]
    Suppose that the Julia set of $f$ is a Jordan curve. Then the Fatou components of $f$ can be either attracting or parabolic; see the discussion after Theorem 1.5 in \cite{LuoNtalampekos:piecewise_qc}. This implies that $f$ has no critical points on $J(f)$. Let $U$ be a Fatou component. By Proposition \ref{prop:quasidiskgf}, $U$ is a quasidisk if and only if either $f$ has no parabolic periodic points or $f$ has a parabolic periodic point of multiplicity at least $3$. Since there are only two Fatou components the multiplicity of parabolic points is at most $3$; see \cite{Milnor:dynamics}*{Section 10}. Finally, note that $f$ has no parabolic periodic points if and only if it is hyperbolic, given that it has no critical points on the Julia set \cite{Milnor:dynamics}*{Theorem 19.1}.
\end{proof}

\begin{proof}[Proof of Corollary \ref{corollary:unicritical}]
    Suppose that $f$ is a unicritical geometrically finite polynomial. If there is no bounded Fatou component, the statement is vacuously true, so we assume that there exists a bounded Fatou component. All bounded periodic Fatou components are either hyperbolic or parabolic. Thus, the unique (finite) critical point of $f$ is contained in the Fatou set and there are no critical points on the Julia set. 
    
    Suppose that there exists a quasiconformal map of $\C$ that maps each bounded Fatou component of $f$ onto a disk. Suppose first that $J(f)$ is not homeomorphic to $\mathbb S^1$.  By Theorem \ref{theorem:polynomial_uniformization}, conditions \ref{r:parabolic3} and \ref{r:parabolic_tangent} hold. Suppose that $f$ has no parabolic points. By \ref{r:parabolic_tangent}, we conclude that bounded Fatou components have disjoint closures. Since there are no critical points on the Julia set, $f$ is hyperbolic \cite{Milnor:dynamics}*{Theorem 19.1}. Now, if $J(f)$ is homeomorphic to $\mathbb S^1$, by Proposition \ref{prop:julia_jordan}, either $f$ has a parabolic periodic point of multiplicity $3$ or $f$ is hyperbolic, as desired. 

    Conversely, suppose that $f$ is hyperbolic and all bounded Fatou components have disjoint closures. If $J(f)$ is homeomorphic to $\mathbb S^1$, then by Proposition \ref{prop:julia_jordan} there exists a quasiconformal map $g$ of $\C$ that maps the bounded Fatou component of $f$ onto a disk. If $J(f)$ is not homeomorphic to $\mathbb S^1$, then conditions \ref{r:parabolic3}, \ref{r:parabolic_basins}, \ref{r:parabolic_tangent} are vacuously true, so by Theorem \ref{theorem:polynomial_uniformization} there exists a quasiconformal map $g$ of $\C$ that maps each bounded Fatou component of $f$ onto a disk. 
    
    Finally, suppose that $f$ has a parabolic periodic point of multiplicity $3$. Since $f$ is unicritical, there is a unique periodic cycle of bounded Fatou components, all of which are immediate basins of this parabolic point; see \cite{Milnor:dynamics}*{Theorem 10.15}. Hence, conditions \ref{r:parabolic3} and \ref{r:parabolic_basins} hold. By Lemma \ref{lem:evenpara} below, condition \ref{r:parabolic_tangent} holds. Thus, by Theorem \ref{theorem:polynomial_uniformization} we conclude that there exists a quasiconformal map $g\colon \C\to \C$ that maps each bounded Fatou component of $f$ to a disk. 
\end{proof}

\begin{lemma}\label{lem:evenpara}
    Let $f$ be a unicritical polynomial with a parabolic periodic point of multiplicity at least $3$. Then every contact point between Fatou components is eventually mapped to the parabolic cycle.
\end{lemma}

\begin{proof}
    Since $f$ is a geometrically finite polynomial, all bounded Fatou components are Jordan regions \cite{RoeschYin:boundary_fatou}. Let $U,V$ be bounded Fatou components such that $\bar U\cap \bar V\neq \emptyset$. Since $f$ is a polynomial, the Julia set is the boundary of the basin $\mathcal A(\infty)$, so $\bar U\cap \bar V$ consists of a single point $a_0$, which is not critical. Since there are no critical points on the Julia set, for each $n\in \N$ the components $f^{\circ n}(U)$ and $f^{\circ n}(V)$ are distinct and touch at $f^{\circ n}(a_0)$; see Lemma \ref{lemma:no_infinite} \ref{lemma:no_infinite:00}. Since $U$ and $V$ are pre-periodic, the point $a_0$ is pre-periodic. Thus, by passing to iterates, it suffices to show that every periodic contact point between two periodic Fatou components is parabolic. 
        
    Let $\{U_1,\dots, U_k\}$ be the collection of bounded periodic Fatou components, and let $\{a_1,\dots, a_l\}$ be the collection of all contact points between $U_i$ and $U_j$, $i\neq j$. We show that $\{a_1,\dots,a_l\}$ is exactly the parabolic cycle. Since $f$ is unicritical, $\{U_1,\dots, U_k\}$ form a single orbit; see \cite{Milnor:dynamics}*{Theorem 10.15}. In particular, $f$ permutes the collection $\{U_1,\dots, U_k\}$ and $f^{\circ k}(U_i)=U_i$ for each $i\in \{1,\dots,k\}$. Moreover, $f$ maps each contact point to a contact point and $f^{\circ k}$ fixes contact points. Therefore $f$ is a bijection on $\{a_1,\dots,a_l\}$.

    We consider a combinatorial bipartite graph $\Gamma$ with two types of vertices $V_1, V_2$: vertices in $V_1$ are in correspondence with components in $\{U_1,\dots,U_k\}$ and vertices in $V_2$ are in correspondence with contact points $\{a_1,\dots, a_l\}$. A vertex $v\in V_1$ is connected to $w\in V_2$ if the corresponding contact point for $w$ is on the boundary of the component for $v$. By construction, $\Gamma$ is a finite union of trees, and the map $f$ induces a natural graph isomorphism $f_*\colon \Gamma \rightarrow \Gamma$. By the above, $f_*$ is a permutation on $V_1$.

    Let $T$ be a component of $\Gamma$. We claim that $T$ contains exactly one vertex in $V_2$ that corresponds to a point in the parabolic periodic cycle. Indeed, let $A = V_1 \cap T$ and $B = V_2 \cap T$. Since each vertex in $V_2$ has valence at least $2$, every endpoint of $T$ is in $V_1$. Since $f_*$ acts transitively on $V_1$, all vertices of $A$ have the same valence. Thus $A$ is exactly the set of endpoints of $T$. Since $T$ is bipartite, there is a unique vertex $b\in B$ which connects to every vertex in $A$, i.e., $T$ is a star. 
    Since the components corresponding to the vertices of $A$ are parabolic and the parabolic multiplicity is at least $3$, the contact point corresponding to $b$ is parabolic. 
\end{proof}

\begin{proof}[Proof of Proposition \ref{corollary:degree2}]
    Suppose that $f(z) = z^d+c$ is hyperbolic, and the closures of any two distinct bounded Fatou components are disjoint. Let $\Omega_1,\dots, \Omega_p$ be the unique cycle of periodic Fatou components of $f$. Then there exist disjoint open sets $V_i$, $i\in \{1,\dots,p\}$, and open sets $U_i$, $i\in \{1,\dots,p\}$ such that $\Omega_i \Subset U_i \Subset V_i$ and the map $f^{\circ p}\colon U_i \to V_i$ is polynomial-like. By the straightening theorem \cite{CarlesonGamelin:Dynamics}*{Theorem VI.1.1}, this polynomial-like map is hybrid equivalent to a unicritical polynomial with an attracting fixed point whose immediate basin contains the critical point. Thus, the corresponding straightening parameter lies in the main hyperbolic component of the Multibrot set. Since the neighborhoods $V_i$, $i\in \{1,\dots,p\}$, are chosen pairwise disjoint, the associated renormalization is primitive. Thus, $c$ is contained in the main hyperbolic component of a primitive Multibrot copy. See \cite{DudkoLyubichMLCFeigenbaum}*{\S 2.2.3--2.2.4} and \cite{MilnorOrbitPortraits}*{\S~6} for the discussion for the Mandelbrot set, and the case for Multibrot set is similar.
        
    Conversely, suppose that $c$ lies in the main hyperbolic component of a primitive Multibrot copy. Then the distinct periodic bounded Fatou components have disjoint closures. We claim that the same is true for all bounded Fatou components. Suppose, to the contrary, that two distinct bounded Fatou components $U, V$ have intersecting closures.
    Since $f$ is hyperbolic, the critical point lies in the Fatou set. Therefore, the images $f^{\circ n}(U), f^{\circ n}(V)$ remain distinct bounded Fatou components whose closures meet at a point (Lemma \ref{lemma:no_infinite} \ref{lemma:no_infinite:00}). Since every bounded Fatou component is eventually periodic, this contradicts the disjointedness of the closures of distinct periodic bounded Fatou components. Hence all bounded Fatou components have pairwise disjoint closures.

    Suppose $f(z) = z^d +c$ has a parabolic periodic point of multiplicity $3$. Let $a$ be a parabolic periodic point with period $p$. Since $f$ is unicritical, $f$ has only one periodic cycle of bounded Fatou components. So the two attracting basins of the parabolic periodic point are contained in the same cycle. Thus, $(f^{\circ p})'(a) = -1$. In the family of unicritical polynomials the multiplier gives a branched covering map of a hyperbolic component to $\D$
    \cite{MilnorOrbitPortraits}*{Theorem~6.5 and Remark~6.4}. Thus the period-doubling points correspond to exactly those with multiplier $-1$. Therefore, $c$ is a period-doubling bifurcation point.
    
    Conversely, suppose $c$ is a period-doubling bifurcation point. Then there exists a parabolic periodic point $a$ of period $p$ with $(f^{\circ p})'(a) = -1$. Therefore, there are $2n$ attracting basins at $a$ for some $n\in \N$; see \cite{Milnor:dynamics}*{Lemma 10.4}. Since $f$ is unicritical, there is a unique cycle of attracting basins; see \cite{Milnor:dynamics}*{Theorem 10.15}. Thus, there are exactly two attracting basins at $a$, interchanged by $f^{\circ p}$. In particular, the parabolic multiplicity is $3$. 
\end{proof}

\appendix
\section{Necessary and sufficient conditions for quasidisks}
\label{appendix}

We establish necessary and sufficient conditions so that a Fatou component of a geometrically finite rational map is a quasidisk.

\begin{theorem}[Proposition~\ref{prop:quasidiskgf}]\label{thm:quasidiskgf_appendix}
    Let $f$ be a geometrically finite rational map, and let $U$ be a Fatou component whose orbit $\{f^{\circ n}(U)\}_{n\geq 0}$ consists of Jordan regions.
    Then $U$ is a quasidisk if and only if
    \begin{enumerate}[label=\normalfont(\arabic*)]
        \item\label{prop:quasidisk:pereq:1} for each $V \in \{f^{\circ n}(U)\}_{n\geq 0}$, $\partial V$ contains at most one parabolic periodic point $a_V$, in which case $V$ is an immediate basin of $a_V$ and 
        \item\label{prop:quasidisk:pereq:2}  if $a_V$ has parabolic multiplicity $2$, then no point $z\in \partial U$ is mapped to $a_V$ by an iterate of $f$ that is a local homeomorphism near $z$. 
    \end{enumerate}
    \end{theorem}

    In particular, condition~\ref{prop:quasidisk:pereq:2} implies that if $a_V$ has multiplicity $2$, then $U$ cannot be a periodic component; see Figure~\ref{fig:JuliaSet}.

\subsection{Quasiconformal extension of circle maps}
Let $T\colon\mathbb S^1\to\mathbb S^1$ be an orienta\-tion-preserving covering map. A \emph{Markov partition} $\mathcal P(T;E)$ consists of a finite forward invariant set
$$
E=\{e_0,\dots,e_r\}\subset\mathbb S^1
$$
in cyclic order and the closed arcs $A_j$ from $e_j$ to $e_{j+1}$, where $e_{r+1}=e_0$, such that $T$ is injective on $\operatorname{int}(A_j)$.
Equivalently, if $T(\operatorname{int}(A_j))\cap\operatorname{int}(A_k)\neq\emptyset$, then $\operatorname{int}(A_k)\subset T(\operatorname{int}(A_j))$; see \cite{LuoNtalampekos:piecewise_qc}*{Definition 3.1}.

We set $b_{jk}=1$ if $T(A_j)\supset A_k$ and $b_{jk}=0$ otherwise. A word
$w=(j_1,\dots,j_n)$ is \emph{admissible} if
$b_{j_1j_2}=\cdots=b_{j_{n-1}j_n}=1$. The corresponding arc is defined
recursively by
$$
A_{j_1\dots j_n}
=\overline{(T|_{\operatorname{int}(A_{j_1})})^{-1}
(\operatorname{int}(A_{j_2\dots j_n}))}.
$$
Equivalently, if $E_1=E$ and $E_n=T^{-(n-1)}(E)$, then the arcs $A_w$ with
$|w|=n$ are the closures of the components of $\mathbb S^1\setminus E_n$;
we call them the \emph{arcs of level $n$}.

We say that $T$ is \emph{expansive} if there is $\delta>0$ such that, for
every two distinct $x,y\in\mathbb S^1$, one has
$|T^{\circ n}(x)-T^{\circ n}(y)|>\delta$ for some $n\geq 0$.
For $e\in E$, the symbol $e^+$ (respectively, $e^-$) denotes the access to
$e$ from the positively (respectively, negatively) oriented side. The access
$e^+$ is \emph{hyperbolic} with multiplier $\lambda>1$ if there is $L\geq1$
such that the following holds. Whenever $I_1,I_2$ are consecutive arcs of level $n$
on that side, $e$ is an endpoint of $I_1$, and $I_2$ is the arc farther
from $e$, then
$$
L^{-1}\lambda^{-n}\leq\diam I_i\leq L\lambda^{-n},
\qquad i=1,2.
$$
The access $e^+$ is \emph{parabolic} with exponent $N\in\mathbb N$ if
$$
\frac{L^{-1}}{n^{1/N}}\leq\diam I_1\leq\frac{L}{n^{1/N}},
\qquad
\frac{L^{-1}}{n^{1+1/N}}\leq\diam I_2
\leq\frac{L}{n^{1+1/N}}.
$$
The definitions on the negative access $e^-$ are analogous. A vertex $e\in E$ is
\emph{symmetrically hyperbolic} if both accesses are hyperbolic with the same
multiplier, and it is \emph{symmetrically parabolic} if both accesses are
parabolic with the same exponent. These definitions are given in
\cite{LuoNtalampekos:piecewise_qc}*{Section 3.5}.

We now formulate one of the main extension results from \cite{LuoNtalampekos:piecewise_qc}. Consider the following hypotheses. We set $T_j=T|_{\operatorname{int}(A_j)}$, $j\in\{0,\dots,r\}$.
\begin{enumerate}[label=\normalfont(M\arabic*)]
    \item\label{appendix:M1} Each of the two accesses of every vertex in $E$ is
    hyperbolic or parabolic.
    \item\label{appendix:M2} For every $j\in \{0,\dots,r\}$ there are open neighborhoods $U_j$
    of $\operatorname{int}(A_j)$ and $V_j$ of
    $T_j(\operatorname{int}(A_j))$ in $\mathbb C$ such that $T_j$ extends to
    a homeomorphism from $U_j$ onto $V_j$, the sets $U_j$ are pairwise disjoint, and
    $$
    \bigcup_{k:\,b_{jk}=1}U_k\subset V_j.
    $$
    \item\label{appendix:M3} Define $U_{jk}=T_j^{-1}(U_k)$ for admissible
    pairs and define $U_w$ recursively in the same way for admissible words.
    There is $K\geq1$ such that, for every admissible $w$ and every
    $0\leq m\leq |w|$, the restriction
    $T^{\circ m}|_{U_w}$ is $K$-quasiconformal.
\end{enumerate}

\begin{theorem}[\cite{LuoNtalampekos:piecewise_qc}*{Theorem 4.1}]
\label{thm:appendix-circle-extension}
    Let $F,G\colon\mathbb S^1\to\mathbb S^1$ be expansive covering maps with Markov partitions that satisfy \ref{appendix:M1}--\ref{appendix:M3}. Suppose that $h\colon \mathbb S^1\to \mathbb S^1$ is a topological conjugacy between $F$ and $G$ that carries the Markov partition of $F$ onto that of $G$. Assume, moreover, that each vertex of the Markov partition of $F$ is symmetrically hyperbolic or symmetrically parabolic, and that $h$ preserves these types. Then the conjugacy $h$ is quasisymmetric.
\end{theorem}

A similar extension result for piecewise analytic rather than piecewise quasiconformal maps (as in \ref{appendix:M3}) appears in \cite{LyubichMerenkovMukherjeeNtalampekos:David}. Moreover, this result requires $F$ and $G$ to extend analytically in small neighborhoods of the endpoints of the corresponding Markov partitions, so the notions of hyperbolicity and parabolicity are defined naturally through the power series expansion. However, in our setting such extensions are not available, so we apply instead the more general Theorem \ref{thm:appendix-circle-extension}.

\subsection{Preparation for the proof of Theorem \ref{thm:quasidiskgf_appendix}}
We need the following two variations of \cite{LuoNtalampekos:piecewise_qc}*{Lemma 7.1}.

\begin{lemma}[Periodic case]\label{lem:expand}
     Let $f$ be a geometrically finite rational map and $U$ be a fixed Fatou component. Suppose that $U$ is a Jordan region and $\partial U$ contains at most one parabolic point $a$, in which case $a$ has parabolic multiplicity at least $3$ and  $U$ is an immediate basin of $a$.
     If $P(f) \cap \partial U \neq \emptyset$, then define
     $$
    \mathcal{A} \coloneqq P(f) \cap \partial U\quad  \text{and} \quad 
    \mathcal{B}\coloneqq f^{-1}(\mathcal{A}) \cap \partial U.
    $$
    Otherwise, let $a\in \partial U$ be a repelling fixed point and define 
    $$
    \mathcal{A} \coloneqq \{a\} \quad \text{and} \quad  \mathcal{B}\coloneqq f^{-1}(\mathcal{A}) \cap \partial U.
    $$
    Let $\{\gamma_i\}$ and $\{I_j\}$  be the components of $\partial U \setminus \mathcal{A}$ and $\partial U \setminus \mathcal{B}$, respectively.
    Then there exist open neighborhoods $N_i$ of $\gamma_i$ and $D_j$ of $I_j$ so that
    \begin{itemize}
        \item (Disjointedness) the sets $\{D_j\}$ are pairwise disjoint,
        \item (Containment) if $I_j \subset \gamma_k$, then $D_j \subset N_k$, 
        \item (Conformality) if $f(I_j) = \gamma_i$, then $f\colon D_j \rightarrow N_{i}$ is conformal.
    \end{itemize}
\end{lemma}
\begin{proof}
    Let us sketch the proof of the lemma, since it is similar to the proof of \cite{LuoNtalampekos:piecewise_qc}*{Lemma 7.1}. After appropriate normalization, we assume that $\overline{U} \subset\C$.

    We first claim that there exists a continuum $Q \subset \overline{U}$ so that 
    $$
    f(Q) \subset Q \text{ and } Q \cap \partial U = \mathcal{A}.
    $$
    Let $Q_0 \subset \overline{U}$ be a closed attracting petal if $U$ is parabolic or a small forward invariant disk about its attracting fixed point if $U$ is (super-)attracting.  
    Choose arcs $E_x \subset \overline{U}$ landing at the vertices of $x\in \mathcal{A}$ that are dynamically compatible in the sense that 
    \begin{itemize}
        \item if $x$ is the parabolic fixed point, then $E_x \subset Q_0$ and $f(E_x) \subset E_x$ and 
        \item otherwise, $E_{f(x)} \subset f(E_x)$.
    \end{itemize}
    This is possible by considering the Koenigs coordinate at repelling periodic points and taking pull-backs at pre-periodic points. Connect the interior endpoints of these arcs to $Q_0$ by compact arcs in $U$. Let $T$ be their union with $Q_0$. Compatibility near the vertices ensures that $f(T)\setminus T$ is contained in a compact set $C\subset U$. Since $Q_0$ attracts compact subsets of $U$, all sufficiently high forward images of $C$ lie in $Q_0$. Therefore, the set
$$
Q=\overline{\bigcup_{n\ge0}f^{\circ n}(T)}
$$
is a compact forward-invariant continuum and $Q\cap\partial U=\mathcal A$. This completes the proof of the claim.

    Let $X\coloneqq \widehat\C \setminus (Q \cup P(f))$ and $Y \coloneqq \widehat\C \setminus f^{-1}(Q \cup P(f))$. For $\epsilon>0$ and a component $\gamma_i$ of $\partial U\setminus \mathcal A$ we set
    $$
    N_{i, \epsilon} \coloneqq N_{\epsilon, X}(\gamma_i) = \{z\in X:h_X(z,\gamma_i)<\epsilon\}.
    $$
    Recall that $h_X$ denotes the hyperbolic metric on $X$. We claim that for sufficiently small $\epsilon$, $N_{i, \epsilon}$ is contained in some  simply connected domain $H_i \subset X$ such that the domains in $\{H_i\}$ are pairwise disjoint.
    
    Since $Q$ is a non-degenerate continuum through each vertex $v \in \mathcal{A}$, by comparison and Koebe's estimate (see \cite{BeardonMinda:hyperbolic}*{Section 8}), we have
    \begin{equation}\label{eq:estimatenearv}
    \rho_X(z) \geq \rho_{\widehat\C\setminus Q}(z)\gtrsim \frac{1}{|z-v|} \quad \text{for $z\in X$ near $v$.}
    \end{equation}
    We construct $H_i$ as follows. Let $v$ be an endpoint of $\gamma_i$. We first construct two arcs landing at $v$ that locally bound a sector containing $\gamma_i$ such that the $h_X$-distance of the two arcs to $\gamma_i$ is bounded from below. If $v$ is a repelling periodic point, construct the two arcs of $H_i$ in a Koenigs fundamental annulus and iterate them by the inverse return map. By compactness, there is a uniform lower bound for the distance between the two arcs in the metric $|dz|/|z-v|$. So, in view of \eqref{eq:estimatenearv} the distance between each of the two arcs and $\gamma_i$ is bounded from below in the $h_X$-metric. If $v$ is a parabolic fixed point, the existence of the arcs follows from the construction given in \cite{LuoNtalampekos:piecewise_qc}*{Lemma 7.1}; see \cite{LuoNtalampekos:piecewise_qc}*{Figure 7.1}. If $v$ is a strictly pre-periodic point, then the two arcs are constructed by pull-backs. We can connect the arcs at the two endpoints of $\gamma_i$ to obtain a closed curve $\alpha$ which bounds a simply connected domain $H_i \subset X$ and the $h_X$-distance between $\alpha$ and $\gamma_i$ is bounded from below. By choosing disjoint sectors at shared vertices, we may construct the regions $\{H_i\}$ so that they are pairwise disjoint. The claim follows.

    We choose $\epsilon$ sufficiently small so that $N_i\coloneqq N_{i,\epsilon}$ is contained in $H_i$ for each $i$.  Suppose that $f(I_j) = \gamma_i$.  Let $D_j$ be the component of $f^{-1}(N_i)$ that contains $I_j$. Since $H_i$ is simply connected and avoids critical values, the map $f\colon D_j\rightarrow N_i$ is conformal.
    Since the neighborhoods $\{N_i\}$ are pairwise disjoint, the sets $\{D_j\}$ have the same property. Suppose that $I_j \subset \gamma_{k}$. Since $Q \cup P(f)$ is forward invariant, the map $f\colon Y \rightarrow X$ is a covering and the inclusion map $\iota\colon Y \rightarrow X$ is distance-decreasing by the Schwarz--Pick lemma.
    Thus,
    $$
    D_j = N_{\epsilon, Y}(I_j) \subset N_{\epsilon, X}(I_j) \subset N_{\epsilon, X}(\gamma_k) = N_k.
    $$
    This completes the proof.
\end{proof}

\begin{lemma}[Strictly pre-periodic case]\label{lem:expandprep}
     Let $f$ be a geometrically finite rational map, and $U$ be a non-periodic Fatou component such that $V = f(U)$ is a fixed Fatou component. Suppose that both $U, V$ are Jordan regions, $\partial U$ contains no parabolic point, and $\partial V$ contains a unique parabolic fixed point $a$ of parabolic multiplicity $2$ such that $V$ is an immediate basin of $a$. Assume further that $\deg_z(f^{\circ n})\ge2$ whenever $z\in\partial U$ and $f^{\circ n}(z)=a$ for some $n\in \N$.
     Let 
     \begin{align*}
    \mathcal{A}' \coloneqq P(f) \cap \partial V, &\text{ and }
    \mathcal{B}'\coloneqq f^{-1}(\mathcal{A}') \cap \partial V\\
    \mathcal{A} \coloneqq f^{-1}(\mathcal{A}') \cap \partial U, &\text{ and }  \mathcal{B} \coloneqq f^{-1}(\mathcal{B}') \cap \partial U.
    \end{align*}
    Let $\{\gamma_i\}$ and $\{I_j\}$ be the components of $\partial U\setminus \mathcal A$ and $\partial U\setminus \mathcal B$, respectively.
    Let $S\colon \partial U \rightarrow \partial U$ be an orientation-preserving covering that is a lift of $f\colon \partial V \rightarrow \partial V$ under $f\colon \partial U \rightarrow \partial V$, i.e., $S$ satisfies
     $$
     f\circ S = f \circ f.
     $$
    Then there exist open neighborhoods $N_i$ of $\gamma_i$ and $D_j$ of $I_j$ so that
    \begin{itemize}
        \item (Disjointedness) the sets $\{D_j\}$ are pairwise disjoint,
        \item (Containment) if $I_j \subset \gamma_k$, then $D_j \subset N_k$,
        \item (Conformality) if $S(I_j) = \gamma_i$, then $S\colon D_j \rightarrow N_{i}$ is conformal.
    \end{itemize}
\end{lemma}
\begin{proof}
    We adapt the construction given in the proof of Lemma~\ref{lem:expand}.
    As in Lemma~\ref{lem:expand}, construct a continuum $Q\subset\overline V$ such that
    $$
        f(Q)\subset Q\quad\text{and}\quad Q\cap\partial V=\mathcal A'.
    $$
    We set
    $$
        X=\widehat{\mathbb C}\setminus(Q\cup P(f)),\qquad Y=f^{-1}(X),\qquad Z=f^{-2}(X).
    $$
    Forward invariance of $Q$ and $P(f)$ gives that $Z\subset Y$, and that the maps
    $$
        f\colon Y\longrightarrow X,\qquad f^{\circ 2}\colon Z\longrightarrow X
    $$
    are coverings. We equip these surfaces with their hyperbolic metrics. For $\epsilon>0$, let 
    $$
    N_{\epsilon, Y}(\gamma_i) = \{z\in X:h_Y(z,\gamma_i)<\epsilon\}, \quad N_{\epsilon, Z}(I_j) = \{z\in X:h_Z(z,I_j)<\epsilon\}.
    $$
    We claim that for all sufficiently small $\epsilon$, the neighborhood $N_{\epsilon, Y}(\gamma_i)$ (resp.\ $N_{\epsilon, Z}(I_j)$) is contained in a simply connected domain $H_i \subset Y$ (resp.\ $L_j \subset Z$) such that the sets $\{H_i\}$ (resp.\ $\{L_j\}$) are pairwise disjoint. The construction is the same as in Lemma~\ref{lem:expand} away from vertices eventually mapped to parabolic points. If a vertex $b$ is eventually mapped to $a$, let $k$ be the local degree of an iterate taking $b$ to $a$. By hypothesis $k\ge2$. Let $\gamma^\pm$ be the two arcs from the family $\{\gamma_i\}$ with an endpoint $b$. Then $\gamma^\pm$ are separated by a positive angle at $b$, so we can choose local sectors bounded by two arcs with $h_Y$-distance to $\gamma^{\pm}$ bounded from below. Indeed, in suitable local coordinates, $f$ takes the form $z\mapsto z^k$, so the preimage of $V$ near $b$ consists of $k$ components whose closures meet at $b$.
    Since $V$ has a cusp at $a$, the two boundary arcs of each component form an interior angle of $2\pi/k$ at $b$. This proves the claim. 

    Set $N_i=N_{\epsilon,Y}(\gamma_i)$ and $D_j = N_{\epsilon, Z}(I_j)$ for sufficiently small $\epsilon$ so that the previous claim holds. Then the sets $\{N_i\}$ are pairwise disjoint and so are the sets $\{D_j\}$. Suppose that $S(I_j)=\gamma_i$. Since $H_i$ and $L_j$ are simply connected, we may analytically extend $S\colon I_j \rightarrow \gamma_i$ and $T \coloneqq S^{-1}\colon \gamma_i \rightarrow I_j$ to maps
    \begin{align*}
        S\colon L_j\longrightarrow Y\quad \text{such that}\quad  f\circ S=f^{\circ 2}, \\
        T\colon H_i \longrightarrow Z\quad \text{such that} \quad f = f^{\circ 2} \circ T.
    \end{align*}
    By the invariance of the hyperbolic density under the coverings $f$ and $f^{\circ 2}$, we conclude that $S, T$ preserve hyperbolic lengths measured in $Z, Y$. Let $z\in D_j=N_{\epsilon,Z}(I_j)$. Choose a rectifiable path in $Z$ from a point of $I_j$ to $z$ with hyperbolic length less than $\epsilon$. Every point of this path lies in $D_j\subset L_j$, so its image under $S$ starts on $\gamma_i$ and has the same length. Hence $S(z)\in N_i$, proving $S(D_j)\subseteq N_i$. The same argument shows that $T(N_i) \subseteq D_j$. Since $S\circ T, T \circ S$ are identity on $\gamma_i$ and $I_j$ respectively, we conclude that $S(D_j) = N_i$ and the map $S\colon D_j \rightarrow N_i$ is conformal. Finally, if $I_j\subset\gamma_k$, the Schwarz--Pick lemma for the inclusion $Z\subset Y$ gives
    $$
    D_j= N_{\epsilon,Z}(I_j)\subset N_{\epsilon,Y}(I_j) \subset N_{\epsilon,Y}(\gamma_k)=N_k.
    $$
    This completes the proof.
\end{proof}

\begin{proposition}\label{proposition:rational_quasidisk}
    Let $f$ be a rational map and suppose that $U$ is a Jordan region, $f(U)$ is a quasidisk, and the map $f\colon U\to f(U)$ is proper. Then $U$ is a quasidisk.
\end{proposition}
\begin{proof}
    Since $f(U)$ is a quasidisk, there exists a quasiconformal map $\phi$ of $\widehat \C$ that maps $f(U)$ onto a disk.  By the measurable Riemann mapping theorem \cite{AstalaIwaniecMartin:quasiconformal}*{Theorem 5.3.4}, there exists a quasiconformal map $\psi$ of $\widehat \C$ whose Beltrami coefficient is the same as that of $\phi\circ f$. Then the map $\phi\circ f\circ \psi^{-1}$ is rational (see \cite{AstalaIwaniecMartin:quasiconformal}*{Theorem 5.5.1}) and maps the Jordan region $\psi(U)$ onto a disk. 

    Therefore, by Theorem \ref{theorem:quasidisk_quasicircle}, it suffices to prove that if $g$ is a rational map and $J$ is a Jordan curve that is mapped onto a circle, then $J$ is a quasicircle. For this, it suffices to prove that the Ahlfors condition \eqref{condition:ahlfors} holds near critical points of $g$ on $J$. By precomposing $g$ with a conformal map near a critical point $z_0\in J$ and post-composing with a M\"obius transformation, we may assume that $g(z)=(z-z_0)^d$ in a neighborhood of $z_0$ for some $d\geq 2$ and $g$ maps an arc $A$ of $J$ that passes through $z_0$ onto a line segment passing through $0$. We observe that $A$ is the union of two non-overlapping segments with a common endpoint at $0$. Thus, the Ahlfors condition holds trivially on the arc $A$.
\end{proof}

\subsection{Proof of Theorem~\ref{thm:quasidiskgf_appendix}}   
    We first prove the necessity. Suppose that $U$ is a quasidisk. Let $l$ be the pre-period of $U$. We claim that if $a_0 \in \partial f^{\circ l}(U)$ is a parabolic periodic point, then $f^{\circ l}(U)$ must be an immediate basin of $a_0$. Indeed, suppose not. Since $f^{\circ l}(U)$ is not an immediate basin of $a_0$, $f^{\circ l}(U)$ is locally contained in the complement of the attracting petals of $a_0$ and it approaches $a_0$ in a repelling direction. By Lemma \ref{lemma:avoid_petals} \ref{lemma:avoid_petals:2}, $\partial U$ cannot be a quasicircle, a contradiction. Hence, condition \ref{prop:quasidisk:pereq:1} holds. For condition \ref{prop:quasidisk:pereq:2}, if the parabolic multiplicity of $a_0$ is $2$, then  $f^{\circ l}(U)$ has an inward cusp at $a_0$. Suppose for contradiction that there exists $z \in \partial U$ so that $f^{\circ n}(z) = a_0$ and $f^{\circ n}$ is locally injective at $z$. Then $U$ has an inward cusp at $z$, contradicting the fact that $U$ is a quasidisk. This proves condition~\ref{prop:quasidisk:pereq:2}.

    We now prove the sufficiency. The argument is a modification of \cite{LuoNtalampekos:piecewise_qc}*{Theorem 1.7}. Under the uniformization maps, the first return map on $\partial U$ induces two piecewise analytic maps on $\mathbb{S}^1$, one coming from a neighborhood inside $U$ and the other from a pinched neighborhood in $\widehat{\C}\setminus \overline{U}$. We will show that these two maps are quasisymmetrically conjugate.
    
    \smallskip 
    \noindent
    {\bf Case (A): $U$ is periodic.}
    By replacing $f$ with an iterate, we assume that the period of $U$ is equal to $1$. If $P(f) \cap \partial U \neq \emptyset$, then let 
    $$
    \mathcal{A} \coloneqq P(f) \cap \partial U \quad \text{and} \quad
    \mathcal{B}\coloneqq f^{-1}(\mathcal{A}) \cap \partial U.
    $$
    Otherwise, $U$ is a (super-)attracting Fatou component and $f$ is a covering map of degree at least $2$ on $\partial U$. In this case, we choose a repelling fixed point $a$ of $f$ on $\partial U$ and let 
    $$
    \mathcal{A} \coloneqq \{a\} \quad \text{and} \quad  \mathcal{B}\coloneqq f^{-1}(\mathcal{A}) \cap \partial U.
    $$
    Let $\phi\colon \D \rightarrow U$ be a conformal map. Then 
    $$
    F\coloneqq\phi^{-1} \circ f \circ \phi\colon \D \rightarrow \D
    $$ 
    is a Blaschke product. Let $\mathcal{B}_{F} \coloneqq \phi^{-1}(\mathcal{B})$. Then $\mathcal{P}(F; \mathcal{B}_{F})$ is a Markov partition for $F$ that satisfies conditions \ref{appendix:M1}, \ref{appendix:M2}, \ref{appendix:M3}, as follows from \cite{LuoNtalampekos:piecewise_qc}*{Theorem 1.4}.

    Let $\psi\colon \widehat\C\setminus \overline{\D} \rightarrow \widehat\C\setminus \overline{U}$ be a conformal map. Since $\partial U$ is a Jordan curve, $\psi$ extends to a homeomorphism of the closures. By Lemma~\ref{lem:expand}, there exists a neighborhood $N$ of $\partial U \setminus \mathcal{A}$ and a neighborhood $N' \subset f^{-1}(N) \cap N$ of $\partial U \setminus \mathcal{B}$ so that $f\colon N' \rightarrow N$ is conformal in each component of $N'$. 
    Thus, the map
    $$
    G\coloneqq \psi^{-1} \circ f \circ \psi
    $$ 
    is defined in a neighborhood of $\mathbb S^1 \setminus \mathcal{B}_{G}$ in $\widehat \C \setminus  \D$, where $\mathcal{B}_{G} = \psi^{-1}(\mathcal{B})$. By the Schwarz reflection principle, $G$ defines a piecewise analytic covering map on $\mathbb S^1$ with Markov partition $\mathcal{P}(G; \mathcal{B}_{G})$. Note that $G$ satisfies condition \ref{appendix:M2} by Lemma \ref{lem:expand}. Also, condition \ref{appendix:M3} is automatic.
    
    The map 
    $$
    h=\psi^{-1}\circ \phi \colon \mathbb S^1\to \mathbb S^1
    $$ 
    is a topological conjugacy between $F$ and $G$ that sends $\mathcal{B}_F$ to $\mathcal{B}_G$. In order to show that $U$ is a quasidisk, by the fundamental theorem of conformal welding \cite{AstalaIwaniecMartin:quasiconformal}*{Theorem 5.10.1}, it suffices to show that $h$ is quasisymmetric. This will follow from Theorem \ref{thm:appendix-circle-extension}, once we verify that $h$ maps each symmetrically hyperbolic (resp.\ parabolic) vertex of the Markov partition of $F$ to a symmetrically hyperbolic (resp.\ parabolic) vertex of the Markov partition associated to $G$. 
    
    Let $b \in \mathcal{B}_G$. We claim that $b$ is symmetrically hyperbolic (resp.\ parabolic) if and only if $h^{-1}(b)$ has the same property. Let $r$ be the pre-period of $b$ and $s$ be the period of $G^{\circ r}(b)$. The map $f_1=f^{\circ r}$ is locally $k$-to-$1$ near $\tilde b=\psi(b)$ for some $k\in \N$ and the map $f_2=f^{\circ s}$ is locally injective at $f_1(\widetilde b)$ and fixes that point. By Lemma \ref{lemma:lift} there exists a conformal map $w$ near $\widetilde b$ that fixes $\widetilde b$ and satisfies $f_2\circ f_1=f_1\circ w$. Moreover, $w$ may be taken to locally preserve $U$ and $\partial U$ near $\widetilde b$. Consider the map $W=\psi^{-1}\circ w\circ \psi$, extended via Schwarz reflection. Note that $W$ is a conformal map near $b$ that fixes $b$. Moreover, $G^{\circ (r+s)}=G^{\circ r}\circ W$ near $b$. This implies that for each $n\in \N$ the map $W$ sends arcs of level $n+s$ (of the Markov partition associated to $G$) near $b$ to arcs of level $n$ near $b$. 

    Suppose that $f_1(\widetilde b)$ is a repelling fixed point of $f_2$. By Lemma \ref{lemma:lift} the point $\widetilde b$ is a repelling fixed point of $w$. Thus, $b$ is a repelling fixed point of $W$. This implies that the lengths of consecutive arcs of level $n$ one of which has $b$ as an endpoint is comparable to $|W'(b)|^{-n/s}$; see \cite{LyubichMerenkovMukherjeeNtalampekos:David}*{Lemma 4.17} for an argument. Thus, $b$ is symmetrically hyperbolic. On the other side, the point $\phi^{-1}(f_1(\widetilde b))$ is a topologically repelling fixed point of the Blaschke product $F^{\circ s}$. This implies that $h^{-1}(b)=\phi^{-1}(\widetilde b)$ is a symmetrically hyperbolic point of the Markov partition of $F$, as desired.

    Next, suppose that $f_1(\widetilde b)$ is the parabolic fixed point $a$ of $f_2$. By assumption, the multiplicity of $a$ is $\nu\geq 3$.  By Lemma \ref{lemma:lift} the point $\widetilde b$ is a parabolic fixed point of $w$ with multiplicity 
    $$1+k(\nu-1)\geq \nu\geq 3.$$
    By \cite{LuoNtalampekos:piecewise_qc}*{Lemma 7.2}, this implies that $b$ is a parabolic fixed point of $W$. Thus, $b$ is symmetrically parabolic; see \cite{LyubichMerenkovMukherjeeNtalampekos:David}*{Lemma 4.18}. On the other side, the point $\phi^{-1}(f_1(\widetilde b))$ is a parabolic fixed point of $F^{\circ s}$ because it attracts some points of $\D$ but repels nearby points of $\mathbb S^1$. Therefore, $h^{-1}(b)$ is symmetrically parabolic.

\smallskip 
    \noindent{\bf Case (B): $U$ is strictly pre-periodic.}
    Let $l$ be the pre-period of $U$ and $p$ be the period of $f^{\circ l}(U)$. By replacing $f$ with an iterate, we may assume that $l=p=1$. By assumption, $f(U)$ is a Jordan region. The map $f\colon \partial U \rightarrow \partial f(U)$ is a covering map between Jordan curves. If $f(U)$ does not contain a parabolic point of multiplicity $2$ on its boundary, then by Case A it is a quasidisk. By Proposition~\ref{proposition:rational_quasidisk}, $U$ is also a quasidisk. Next, suppose that $\partial f(U)$ contains a unique parabolic point $a$, which has multiplicity $2$ and $f(U)$ is an immediate basin of $a$. Let 
    $$
    \mathcal{A}' \coloneqq P(f) \cap \partial f(U)\quad \text{and}\quad
    \mathcal{B}'\coloneqq f^{-1}(\mathcal{A}') \cap \partial f(U).
    $$
    Let $\mathcal{A} = f^{-1}(\mathcal{A}') \cap \partial U$ and $\mathcal{B} = f^{-1}(\mathcal{B}') \cap \partial U$.
    
    Consider any covering map $S\colon\partial U \longrightarrow \partial U$ with $f\circ S = f \circ f$. Let $\phi\colon \D \rightarrow U$ be a conformal map, which extends to a homeomorphism of the closures. Let 
    $$
    F\coloneqq \phi^{-1}\circ S \circ \phi \colon \mathbb{S}^1 \rightarrow \mathbb{S}^1,
    $$
    and let $\mathcal{B}_{F} \coloneqq \phi^{-1}(\mathcal{B})$. Similarly, let $\psi\colon  \widehat\C\setminus \overline{\D} \rightarrow \widehat\C\setminus \overline{U}$ be a conformal map, which extends to a homeomorphism of the closures. Let
    $$
    G\coloneqq \psi^{-1} \circ S \circ \psi\colon \mathbb{S}^1 \rightarrow \mathbb{S}^1,
    $$ 
    and let $\mathcal{B}_{G} = \psi^{-1}(\mathcal{B})$.
    By the Schwarz reflection principle, both $F$ and $G$ define a piecewise analytic covering map on $\mathbb S^1$ with Markov partitions $\mathcal{P}(F; \mathcal{B}_{F})$ and $\mathcal{P}(G; \mathcal{B}_{G})$ respectively. The map 
    $$
    h=\psi^{-1}\circ \phi \colon \mathbb S^1\to \mathbb S^1
    $$ 
    is a topological conjugacy between $F$ and $G$ that sends $\mathcal{B}_F$ to $\mathcal{B}_G$. As in Case A, we will show that $h$ is quasisymmetric. By Lemma~\ref{lem:expandprep}, $F$ and $G$ satisfy condition \ref{appendix:M2}; also \ref{appendix:M3} is immediate. It remains to show that $b \in \mathcal{B}_F$ is symmetrically hyperbolic (resp.\ parabolic) if and only if $h(b)$ has the same property.

    Let $r$ be the pre-period of $b$ and $s$ be the period of $F^{\circ r}(b)$ under $F$. We set $\widetilde{b} = \phi(b)$. Then $f^{\circ (1+r)}(\widetilde{b})$ is periodic. We claim that $b$ and $h(b)$ are symmetrically hyperbolic (resp.\ parabolic) if and only if $f^{\circ (1+r)}(\widetilde{b})$ is repelling (resp.\ parabolic). Let $f_1=f^{\circ (1+r)}$ and $f_2=f^{\circ s}$. The map $f_1$ is locally $k$-to-$1$ near $\widetilde b$ for some $k\in \N$ and the map $f_2$ is locally injective at $f_1(\widetilde b)$ and fixes that point. By Lemma \ref{lemma:lift} there exists a conformal map $w$ near $\widetilde b$ that fixes $\widetilde b$ and satisfies $f_2\circ f_1=f_1\circ w$. Moreover, $w$ may be taken to locally preserve $U$ and $\partial U$ near $\widetilde b$. Consider the map $W=\phi^{-1}\circ w\circ \phi$, extended via Schwarz reflection. Note that $W$ is a conformal map near $b$ that fixes $b$. Moreover, we have $f\circ S^{\circ (r+s)}=f\circ S^{\circ r}\circ w$. This implies that for each $n\in \N$ the map $W$ sends arcs of level $n+s$ (of the Markov partition associated to $F$) near $b$ to arcs of level $n$ near $b$. 
    
    Suppose that $f_1(\widetilde{b})$ is repelling. By Lemma \ref{lemma:lift}, the point $\widetilde b$ is a repelling fixed point of $w$. Thus, $b$ is a repelling fixed point of $W$ and thus it is symmetrically hyperbolic. Analogously, by considering the map $\psi^{-1}\circ w\circ \psi$, which has a repelling fixed point at $h(b)$, we see that $h(b)$ is symmetrically hyperbolic.

    Finally, suppose $f_1(\widetilde{b})$ is parabolic. Let $k$ be the local degree of $f_1$ at $\widetilde{b}$. By Lemma~\ref{lemma:lift}, $w$ has a parabolic fixed point at $\widetilde b$ of multiplicity
    $$
    1+k(\nu-1).
    $$
    Note that if $\nu \geq 3$, then $1+k(\nu-1) \geq \nu \geq 3$. If $\nu = 2$, then by condition~\ref{prop:quasidisk:pereq:2}, we have $k \geq 2$. Thus, $1+k(\nu-1) =1+k \geq 3$ in this case as well.  Since $V$ is an immediate basin of the parabolic fixed point $f_1(\widetilde{b})$ of $f_2$, and $w|_U$ is a lift $f_2|_V$, we conclude that $w$ has an attracting direction in $U$. Since the parabolic multiplicity is $\geq 3$, we conclude that $w$ has an attracting direction in $\widehat\C \setminus \overline{U}$ at $\widetilde{b}$. Moreover, $w$ is repelling along $\partial U$ at $\widetilde{b}$.
    By reflection, we conclude that $\phi^{-1}\circ w\circ \phi$ and $\psi^{-1}\circ w\circ \psi$ are parabolic with multiplicity $\geq 3$. Therefore, both $b$ and $h(b)$ are symmetrically parabolic.
    This proves the claim. 
\qed

\bibliography{biblio}
\end{document}